\documentclass[11pt]{article}
\usepackage{amsthm}
\usepackage{mathrsfs}

\usepackage{tikz}
\usepackage{pgf}
\usetikzlibrary{arrows}
\usetikzlibrary{calc}
\usetikzlibrary{decorations.pathmorphing}
\usepackage{xfrac}    
\usepackage{soul}
\setstcolor{red}
\setulcolor{red}

\usepackage{mathrsfs}
\DeclareMathAlphabet{\mathpzc}{OT1}{pzc}{m}{it}

\usepackage{color}
\usepackage{amsmath}
\usepackage{graphicx}
\usepackage{graphics}
\usepackage{float}
\usepackage{amssymb}%
\usepackage{latexsym}
\usepackage{psfrag}
\usepackage{subfigure}
\usepackage{accents}

\newtheorem{theorem}{Theorem}[section]
\newtheorem{corollary}[theorem]{Corollary}
\newtheorem{definition}[theorem]{Definition}
\newtheorem{lemma}[theorem]{Lemma}
\newtheorem{proposition}[theorem]{Proposition}

\newtheorem{assumption}[theorem]{Assumption}

\numberwithin{equation}{section}
\numberwithin{table}{section}
\numberwithin{figure}{section}

\usepackage[a4paper]{geometry}
\newtheorem{remark}[theorem]{Remark}
\newtheorem{example}[theorem]{Example}
\newcommand{\noi}{\noindent}

\newcommand{\cA}{{\cal A}}

\newcommand{\cH}{{\cal H}}

\newcommand{\cL}{{\cal L}}
\newcommand{\cP}{{\cal P}}
\newcommand{\cQ}{{\cal Q}}
\newcommand{\cR}{{\cal R}}

\newcommand{\cT}{{\cal T}}

\newcommand{\cV}{{\cal V}}

\newcommand{\cZ}{{\cal Z}}
\newcommand{\bx}{x}

\newcommand{\bv}{v}

    \newcommand\quotient[2]{
        \mathchoice
            {
                \text{\raise1ex\hbox{$#1$}\Big/\lower1ex\hbox{$#2$}}%
            }
            {
                #1\,/\,#2
            }
            {
                #1\,/\,#2
            }
            {
                #1\,/\,#2
            }
    }

\newcommand{\ri}{{\rm i}}

\newcommand{\beq}{\begin{equation}}
\newcommand{\eeq}{\end{equation}}
\newcommand{\beqs}{\begin{equation*}}
\newcommand{\eeqs}{\end{equation*}}
\newcommand{\bit}{\begin{itemize}}
\newcommand{\eit}{\end{itemize}}
\newcommand{\ben}{\begin{enumerate}}
\newcommand{\een}{\end{enumerate}}
\newcommand{\bal}{\begin{align}}
\newcommand{\eal}{\end{align}}
\newcommand{\bals}{\begin{align*}}
\newcommand{\eals}{\end{align*}}
\newcommand{\bse}{\begin{subequations}}
\newcommand{\ese}{\end{subequations}}
\newcommand{\bpr}{\begin{proposition}}
\newcommand{\epr}{\end{proposition}}
\newcommand{\bre}{\begin{remark}}
\newcommand{\ere}{\end{remark}}
\newcommand{\bpf}{\begin{proof}}
\newcommand{\epf}{\end{proof}}
\newcommand{\ble}{\begin{lemma}}
\newcommand{\ele}{\end{lemma}}
\newcommand{\bco}{\begin{corollary}}
\newcommand{\eco}{\end{corollary}}
\newcommand{\bex}{\begin{example}}
\newcommand{\eex}{\end{example}}

\newcommand{\Rea}{\mathbb{R}}
\newcommand{\Com}{\mathbb{C}}

\newcommand{\dive}{\mathop{{\rm div}}}
\newcommand{\grad}{\mathop{{\rm grad}}}

\def\XXint#1#2#3{{\setbox0=\hbox{$#1{#2#3}{\int}$}
     \vcenter{\hbox{$#2#3$}}\kern-.5\wd0}}

\usepackage{hyperref}
\definecolor{myblue}{rgb}{0,0,0.6}
\hypersetup{colorlinks=true,linkcolor=myblue,citecolor=myblue,filecolor=myblue,urlcolor=myblue}
\newcommand*{\N}[1]{\left\|#1\right\|}

\allowdisplaybreaks[4]

\newcommand{\tfa}{\text{ for all }}
\newcommand{\tfor}{\text{ for }}

\newcommand{\tin}{\text{ in }}

\newcommand{\tas}{\text{ as }}
\newcommand{\tand}{\text{ and }}

\newcommand{\vertiii}[1]{{\left\vert\kern-0.25ex\left\vert\kern-0.25ex\left\vert #1
    \right\vert\kern-0.25ex\right\vert\kern-0.25ex\right\vert}}

\definecolor{jwcol}{RGB}{27, 137, 18}  

\definecolor{dalcol}{rgb}{0.8,0,0}

\definecolor{escol}{rgb}{0,0,0.8}
\definecolor{estcol}{rgb}{0,0.5,0}
\definecolor{esnewcol}{rgb}{0,0.5,0}

\newcommand{\supp}{{\rm supp}}

\newcommand{\Cosc}{C_{\rm{osc}}}

\newcommand{\Csol}{C_{\rm sol}}

\newcommand{\Creg}{{C_{\rm reg}}}

\newcommand{\tr}{{\rm tr}}

\newcommand{\e}{\epsilon}

\newcommand{\domaingen}{\Omega}

\newcommand{\RPMLo}{R_{\rm PML, -}}
\newcommand{\RPMLt}{R_{\rm PML, +}}

\newcommand{\Rtr}{R_{\tr} }

\newcommand{\fdspace}{\Hilbert_h}

\newcommand{\Coscil}{C_{\rm osc}}

\renewcommand{\Re}{\operatorname{Re}}

\newcommand{\eq}{:=}

\newcommand{\ccurl}{\curl}

\newcommand{\BH}{H}

\newcommand{\BN}{N}

\newcommand{\BP}{P}

\newcommand{\CJ}{\mathcal J}

\newcommand{\CT}{\mathcal T}

\newcommand{\LF}{\mathscr F}

\usetikzlibrary{positioning}
\usepackage{lscape}

\newcommand{\curl}{{\rm curl}\,}
\newcommand{\Ker}{{\rm Ker}\,}
\newcommand{\Phash}{\operator^{\#}}
\newcommand{\Pihash}{\Pi^{\#}_h}

\newcommand{\Rs}{(\operator^*)^{-1}}
\newcommand{\Rhash}{(\Phash)^{-1}}

\newcommand{\Pone}{\mathcal{D}}
\newcommand{\Ptwo}{\mathcal{E}}
\newcommand{\Hilbert}{\cH}
\newcommand{\Hilbertzero}{\cV}
\newcommand{\Hilbertalt}{H}
\newcommand{\Hilbertzeroalt}{V}
\newcommand{\Palt}{\cQ}

\newcommand{\Hpw}[1]{H^{#1}_{\rm pw}}
\newcommand{\Hpwo}[1]{H^{#1}_{{\rm pw}, \wn }}

\newcommand{\newell}{m}
\newcommand{\newellfour}{\ell}
\newcommand{\newellthree}{\ell}
\newcommand{\newnewell}{\newell}

\newcommand{\wn}{k}

\newcommand{\operator}{P}
\newcommand{\smoother}{S}

\newcommand{\newZ}{Z}
\newcommand{\newf}{g}
\newcommand{\hK}{h_K}

\definecolor{jeffColor}{RGB}{102, 0, 204}

\title{
Sharp 
error bounds for 
edge-element discretisations of 
the high-frequency 
Maxwell equations}

\author{
T.~Chaumont-Frelet\thanks{Inria Univ.~Lille and Laboratoire Paul Painlev\'e, 59655 Villeneuve-d'Ascq, France, {\tt theophile.chaumont@inria.fr}}, \quad
J.~Galkowski\thanks{Department of Mathematics, University College London, 25 Gordon Street, London, WC1H 0AY, UK,   \tt J.Galkowski@ucl.ac.uk},\quad
E.~A.~Spence\thanks{Department of Mathematical Sciences, University of Bath, Bath, BA2 7AY, UK, \tt E.A.Spence@bath.ac.uk }
}
\date{\today}

\begin{document}
\pagenumbering{arabic}

\maketitle

\begin{abstract}
We prove sharp wavenumber-explicit error bounds for first- or second-family-N\'ed\'elec-element (a.k.a.~edge-element) conforming discretisations,
of arbitrary (fixed) order, of 
the variable-coefficient time-harmonic Maxwell equations posed in a bounded domain with perfect electric conductor (PEC) boundary conditions. The PDE coefficients are allowed to be piecewise regular and complex-valued; this set-up therefore includes scattering from a PEC obstacle and/or variable real-valued coefficients, with the radiation condition approximated by a perfectly matched layer (PML). 

In the analysis of the $h$-version of the finite-element method, 
with fixed polynomial degree $p$, applied to the time-harmonic Maxwell equations, the \emph{asymptotic regime} is when the meshwidth, $h$, is small enough (in a wavenumber-dependent way) that the Galerkin solution is quasioptimal independently of the wavenumber, while the \emph{preasymptotic regime} is the complement of the asymptotic regime.

The results of this paper are the first preasymptotic error bounds for the time-harmonic Maxwell equations using first-family N\'ed\'elec elements or higher-than-lowest-order second-family N\'ed\'elec elements. Furthermore, they are the first wavenumber-explicit results, even in the asymptotic regime, for Maxwell scattering problems with a non-empty scatterer.
\end{abstract}

\section{Introduction}

\subsection{Statement of the main result}\label{sec:statement}

We consider the time-harmonic Maxwell equations
\beq\label{eq:Maxwell}
\wn ^{-2}\curl (\mu^{-1} \curl E) - \epsilon E = f,
\eeq
with wavenumber $\wn$, posed in a bounded Lipschitz domain $\Omega\subset \Rea^3$ with outward-pointing unit normal vector $n$ and diameter $L$, where $E \in H_0(\curl, \Omega)$ (i.e., $E\in H(\curl,\Omega)$ with $E\times n=0$ on $\partial \Omega$), the data $f \in (H_0(\curl, \Omega))^*$, and the coefficients $\mu$ and $\epsilon$
(the relative permeability and relative permittivity, respectively)  satisfy $\Re \mu, \Re \epsilon \geq c>0$ (in the sense of quadratic forms) in $\Omega$.
We are interested in this problem when $\wn L\gg 1$, i.e., the high-frequency regime. 

This setting includes the radial-perfectly-matched-layer approximation to the scattering problem where the scattering is caused by variable $\mu$ and $\epsilon$ and/or a perfect-electric-conductor obstacle; see \S\ref{sec:PML}.

We study approximations to the solution of~\eqref{eq:Maxwell} using the $h$-version of the finite-element method ($h$-FEM), where accuracy is increased by decreasing the meshwidth $h$ while keeping the polynomial degree $p$ constant, and the (conforming) approximation space consists of the first family (also called the ``first type" or ``first kind") of N\'ed\'elec finite elements \cite{Ne:80}, whose definition is recapped in 
\S\ref{sec:Nedelec} below
\footnote{
Recall that N\'ed\'elec elements are often called edge elements because at the lowest order basis functions and degrees of freedom are associated with the edges of the mesh; at higher order the geometrical identification of basis functions and degrees of freedom is more complicated.
}; note that we choose the convention that $p=1$ corresponds to lowest-order N\'ed\'elec elements.
Since the second family of N\'ed\'elec finite elements 
 \cite{Ne:86}, \cite[\S8.2]{Mo:03}, \cite[\S15.5.1]{ErGu:21} contains the first family, and our results depend only on best-approximation properties of the space, 
 our results also hold for second-family N\'ed\'elec finite elements.

We work in norms where each derivative is scaled by $\wn ^{-1}$; in particular,
\beq\label{eq:1knorm}
\N{E}^2_{H_\wn (\curl,\Omega)}:= \wn ^{-2}\N{\curl E}^2_{L^2(\domaingen)} + \N{E}^2_{L^2(\domaingen)}.
\eeq

If \eqref{eq:Maxwell} has a solution for every $f\in L^2(\Omega)$, we define
\beq\label{eq:Csol}
\Csol= \Csol(k):=
 \sup_{0\neq f \in L^2(\Omega)} \bigg\{\frac{
\| E\|_{L^2(\Omega)}
}{
\| f\|_{L^2(\Omega)}
}
\,:\, E \text{ satisfies  \eqref{eq:Maxwell}}
\bigg\};
\eeq
otherwise $\Csol := \infty$.
This definition implies that $\Csol=\infty$ if \eqref{eq:Maxwell} does not have a unique solution for every $f\in L^2(\Omega)$.
Recall that, with the norm convention \eqref{eq:1knorm}, the $L^2(\Omega)\to H_\wn(\curl,\Omega)$ and $(H_\wn(\curl,\Omega))^*\to H_\wn(\curl,\Omega)$ norms are then both bounded by a $\wn$-independent multiple of $1+\Csol$.

\begin{definition}[$C^\newellthree$ with respect to a partition]\label{def:Crpartition}
For $\newellthree\in \mathbb{N}$, $\Omega$ is $C^\newellthree$ with respect to the partition $\{\Omega_j\}_{j=1}^n$ if 

(i) $\overline{\Omega}= \cup_{j=1}^n \overline{\Omega_j}$, where $\Omega_i\cap \Omega_j=\emptyset$ if $i\neq j$,

(ii) $\Gamma_{i,j}$ is $C^{\newellthree}$ for all $(i,j)$,  where $\partial\Omega_j=\sqcup_{i=1}^{L_j}\Gamma_{i,j}$ is the decomposition of $\partial\Omega_j$ into  its connected components, and 

(iii) for all $i,i',j,j'$, if $\Gamma_{i,j}\cap \Gamma_{i',j'}\neq \emptyset$, then $\Gamma_{i,j}=\Gamma_{i',j'}$.

\end{definition}

\begin{figure}\label{fig:1} 
\begin{center}
\includegraphics{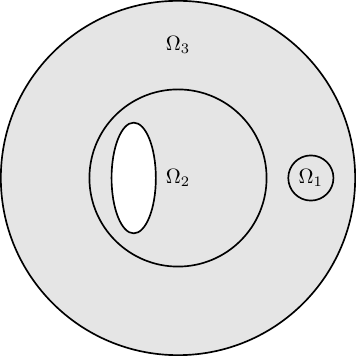}
\caption{An example of $\Omega$ (shaded) with $\overline{\Omega}= \cup_{j=1}^n \overline{\Omega_j}$ satisfying Definition \ref{def:Crpartition}.} 
\end{center}
\end{figure}

This definition implies that if $\Omega$ is $C^\newellthree$ with respect to a partition, then $\partial \Omega$ is  $C^\newellthree$ (since $\partial \Omega=\Gamma_{i,j}$ for some $i,j$).
Figure \ref{fig:1} shows an example of $\overline{\Omega}= \cup_{j=1}^n \overline{\Omega_j}$ satisfying Definition \ref{def:Crpartition}.

\begin{assumption}[Regularity assumptions on $\Omega, \epsilon,$ and $\mu$]
\label{ass:regularity}
For some $\newnewell\in \mathbb{N}$, $\Omega$ is $C^{\newnewell+1}$ with respect to the partition $\{\Omega_j\}_{j=1}^n$ 
and 
$\epsilon\in C^{\newnewell,1}(\overline{\Omega_j})$ 
and $\mu\in C^{\newnewell}(\overline{\Omega_j})$
 for all $j=1,\ldots, n$.
\end{assumption}

\begin{theorem}[The main result]\label{thm:intro}
Suppose that 
Assumption \ref{ass:regularity} holds for an integer $\newnewell\geq1$. 
Given $1\leq p\leq \newnewell$ and $\wn_0, \Coscil>0$ there exist $C_1, C_2, C_3>0$ such that the following holds. 

Let $\Hilbert_h\subset H_0(\curl,\Omega)$ be the space of first- or second-family N\'ed\'elec finite-elements of degree $p$ on a (curved) 
mesh satisfying Assumption \ref{assumption_curved_fem} below, with maximal element size $h$.

For all $k\geq k_0$ and $h>0$ satisfying
\beq\label{eq:threshold}
(\wn h )^{2p}\Csol \leq C_1
\eeq
the Galerkin solution $E_h$ exists, is unique, and satisfies
\beq\label{eq:H1bound}
\N{E-E_h}_{H_\wn (\curl,\Omega)}\leq C_2 \Big(1 + (\wn h )^p \Csol\Big) \min_{v_h \in \fdspace} \N{E-v_h}_{H_\wn (\curl,\Omega)}.
\eeq

Furthermore, 
 if the data $f$ is $\wn$-oscillatory with constant $\Cosc$ and regularity index $\newnewell$ (in the sense of Definition \ref{def:oscil} below), 
then
\beq\label{eq:rel_error}
\frac{\N{E-E_h}_{H_\wn (\curl,\Omega)}}
{
\N{E}_{H_\wn (\curl,\Omega)}
}\leq C_3 \Big(1  + (\wn h )^p \Csol\Big)(\wn h )^p;
\eeq
i.e., the relative $H_\wn (\curl,\Omega)$ error can be made controllably small by making $(\wn h)^{2p}\Csol$ sufficiently small.
\end{theorem}

Note that 
if \eqref{eq:Maxwell} does not have a unique solution for every $f\in L^2(\Omega)$ then $\Csol=\infty$ and \eqref{eq:threshold} is never satisfied.

\bre[The origin of the assumptions of Theorem \ref{thm:intro}]
The abstract version of Theorem \ref{thm:intro} -- Theorem \ref{thm:abs1} below -- is proved assuming only a G\aa rding inequality and elliptic-regularity-type assumptions (see Assumptions \ref{ass:1} and \ref{ass:2} below). Theorem \ref{thm:intro} is then proved by showing that these regularity assumptions are satisfied using the classic regularity results of Weber \cite{We:81}  (see Theorem \ref{thm:Weber} and Lemma \ref{lem:Maxwell_reg} below).
\ere

\bre[The main result applied to differential $r$-forms]
Theorem~\ref{thm:abs1} can also be applied to differential $r$-forms in any dimension. In this case, the operator $\e^{-1}\curl (\mu^{-1}\curl )$ is replaced by $\Pone=*d*d$, where $*$ denotes the Hodge $*$ operator (with respect to the relevant metric); the kernel of $\Pone$ then consists of closed $r$-forms. (Recall that finite-element spaces in this setting are discussed in \cite{Ch:07}.)
\ere

\bre[The norm of the solution operator $\Csol$]\label{rem:Csol}
Theorem \ref{thm:intro} involves $\Csol$ defined in \eqref{eq:Csol}; i.e., the $L^2(\Omega)\to L^2(\Omega)$ norm of the solution operator $f\mapsto E$. We note the following:~(i) With the definition  \eqref{eq:1knorm} of the norm $\|\cdot\|_{H_\wn(\curl,\Omega)}$, the $\wn$-dependence of $\Csol$ is the same as the $L^2(\Omega)\to H_\wn(\curl, \Omega)$ norm of the solution operator.
(ii) When $\mu$ and $\epsilon$ are both constant multiples of the identity in part of the domain, $\Csol \geq C\wn L$ -- this can be proved 
by considering data that is a cut-off function multiplied by a plane wave; see, e.g., \cite[\S1.4.1]{ChMoSp:23}, \cite[Example 3.4]{MeSa:23}.
(iii) When $\epsilon$ and $\mu$ are both real, the problem is self adjoint, and $\Csol$ is given in terms of the distance from $\wn^2$ to the spectrum
\cite[\S2.3, \S3]{ChVe:22}.
(iv) \cite[Theorem 1.6]{GLS2} proved that the norm of the solution operator of the Helmholtz PML problem is bounded by the norm of the solution operator of the corresponding Helmholtz scattering problem; we expect that the same result holds for the Maxwell PML problem. 
\footnote{For the case of no scatterer and Cartesian PML, \cite[Lemma 10]{CuWaXi:25} proved that $\Csol \leq C (\wn L)^2$; i.e., a $\wn L$ loss from the expected estimate.}
If so, then $\Csol \leq C\wn L$ when the problem is nontrapping; see 
\cite[Theorem 1.1]{ChMoSp:23} (for certain nontrapping coefficients) and \cite[\S2]{Ya:88} (for a nontrapping PEC obstacle).
\ere

\subsection{The context and novelty of the main result}\label{sec:context}

\paragraph{The asymptotic and preasymptotic regimes.}

We first discuss the analysis of the $h$-FEM applied to the Helmholtz equation $(\wn ^{-2}\Delta +1)u=f$.
The concepts of the asymptotic and preasymptotic regimes were first introduced by Ihlenburg and Babu\v{s}ka in \cite{IhBa:95a, IhBa:97}. In the \emph{asymptotic regime}, which is now known to be when $h=h(k)$ satisfies $(\wn h )^p\Csol \ll 1$, the sequence of Galerkin solutions are quasioptimal, with quasioptimality constant independent of $\wn $. 
The \emph{preasymptotic regime} is then when $(\wn h)^p \Csol \gg1$. In this regime, 
one expects that 
if $(\wn h )^{2p}\Csol$ is sufficiently small then, for data oscillating at frequency $\lesssim \wn $, the relative error of the Galerkin solution is controllably small.
Note that, since $\Csol$ grows with $\wn L$, $h \wn =o(1)$ in the asymptotic regime, and this is the well-known 
\emph{pollution effect} \cite{BaSa:00}.

\paragraph{State of the art in the asymptotic regime for the Helmholtz $h$-FEM.}

The natural error bounds in the asymptotic regime were proved for Helmholtz problems 
 satisfying only a G\aa rding inequality and an elliptic-regularity shift in \cite{ChNi:20} following 
earlier work by \cite{MeSa:10, MeSa:11, EsMe:12} for constant-coefficient Helmholtz problems. In fact, this earlier work showed that 
the $hp$-FEM does not suffer from the pollution effect when $hk/p\leq C_1$ for $C_1$ sufficiently small, 
$p\geq C_2 \log (\wn L)$ for $C_2$ sufficiently large,  and $\Csol\leq C_3 (\wn L)^N$ for some $C_3, N>0$; this result is now known for variable-coefficient Helmholtz problems by  
 \cite{LSW3, LSW4, GLSW1, BeChMe:24}. 

The error bounds in the asymptotic regime rely on the fact that, since the Helmholtz adjoint solution operator is compact as a map from $L^2$ to $H^1$, 
the $L^2$ norm of the error is asymptotically smaller than the $H^1$ norm by the Aubin--Nitsche lemma (see, e.g., \cite[Theorem 19.1]{Ci:91}). Indeed, with $P$ denoting the Helmholtz operator, Galerkin orthogonality $\langle \operator(u-u_h), v_h\rangle =0$ for all finite-element functions $v_h$ implies that, with $\Pi_h$ the orthogonal projection onto the finite-element space, 
\begin{align}\nonumber
\N{u-u_h}^2_{L^2} 
&= \big\langle \operator^{-1}\operator(u-u_h) , u-u_h\big\rangle,\\ \nonumber
&= \big\langle \operator(u-u_h) , (\operator^*)^{-1} (u-u_h)\big\rangle,\\ \nonumber
&= \big\langle \operator(u-u_h) , (I-\Pi_h)(\operator^*)^{-1}  (u-u_h)\big\rangle, \\
&\leq  C \N{u-u_h}_{H^1}  \big\|(I-\Pi_h)(\operator^*)^{-1}\big\|_{L^2\to H^1}\N{  u-u_h}_{L^2}.\label{eq:AN}
\end{align}
Schatz \cite{Sc:74} used this duality argument in conjunction with a G\aa rding inequality to bound the Helmholtz FEM error; see also \cite{Sa:06} for a more modern perspective. 
(The Maxwell analogue of this result is Lemma \ref{lem:basic} below.)
The results \cite{MeSa:10, MeSa:11, EsMe:12, ChNi:20, LSW3, LSW4, GLSW1, BeChMe:24} discussed above then obtained quasi-optimality (with constant independent of $\wn$) when  $(\wn h )^p\Csol$ is sufficiently small by bounding $\|(I-\Pi_h)(\operator^*)^{-1}\|_{L^2\to H^1}$.

\paragraph{State of the art in the preasymptotic regime for the Helmholtz $h$-FEM.}
The natural bounds in the preasymptotic regime (i.e., the Helmholtz analogues of \eqref{eq:H1bound} and \eqref{eq:rel_error} above) were proved in \cite{GS3} for Helmholtz problems satisfying only a G\aa rding inequality and an elliptic-regularity shift, following earlier work by \cite{Wu:14, ZhWu:13, DuWu:15, BaChGo:17, Pe:20, ChGaNiTo:22}. 
Central to this earlier work was 
the \emph{elliptic projection argument} \cite{FeWu:09, FeWu:11}, which used that the Helmholtz operator is coercive if a sufficiently large multiple of the identity is added. The
 key insight in 
\cite{GS3} is that, in fact, this coercivity can be achieved by adding a smoothing operator $S$, defined in terms of eigenfunctions of the real part of the Helmholtz operator
(a Maxwell analogue of this is Lemma \ref{lem:S1} below).

We highlight that the arguments of \cite{GS3} immediately obtain a splitting analogous to that used to bound $\|(I-\Pi_h)(\operator^*)^{-1}\|_{L^2\to H^1}$ in
\cite{MeSa:10, MeSa:11, EsMe:12, ChNi:20, LSW3, LSW4, GLSW1, BeChMe:24}. Indeed, 
since
\beqs
(\operator^*+\smoother)(\operator^*)^{-1} = I+ \smoother(\operator^*)^{-1}, 
\eeqs
then
\beq\label{eq:splitting}
(\operator^*)^{-1} = (\operator^*+\smoother)^{-1}+ (\operator^*+\smoother)^{-1}\smoother(\operator^*)^{-1}.
\eeq
If $\smoother$ is a smoothing operator such that $\operator+\smoother$ is coercive (with coercivity constant independent of $\wn$) and $\operator$ satisfies the natural assumptions for elliptic regularity, then $(\operator^*+\smoother)^{-1}$ has the regularity shift associated with $(\operator^*)^{-1}$, but its norm is bounded independent of  $\wn$. Furthermore, $(\operator^*+S)^{-1}\smoother(\operator^*)^{-1}$ is smoothing, with norm bounded by the norm of $(\operator^*)^{-1}$.

\paragraph{Duality-argument analysis of the Maxwell $h$-FEM using N\'ed\'elec finite elements.}

Compared to the analysis of the Helmholtz $h$-FEM, the analysis of the Maxwell $h$-FEM is complicated by the large kernel of 
the curl operator. The kernel of $\curl$ does not consist of smooth functions; thus 
neither the solution operator nor its adjoint are compact.  
The duality arguments described above for Helmholtz therefore cannot immediately be applied.

If $\dive (\zeta E)=0$ for some $\zeta$ with $\Re \zeta\geq c>0$ (in the sense of quadratic forms), then $E$ lies in a subspace transverse to the kernel of $\curl$ and  the solution operator increases regularity by the regularity results of Weber \cite{We:80}; see Theorem \ref{thm:Weber} and Lemma \ref{lem:Maxwell_reg} below. 
This is related to the fact that, whereas the embedding $H_0(\curl,\Omega) \hookrightarrow L^2(\Omega)$ is not compact, the embedding $H_0(\curl,\Omega)\cap H(\dive,\zeta,\Omega)\hookrightarrow L^2(\Omega)$ is compact \cite{We:74, We:80, Pi:84} \cite[\S8.4]{Le:86}, where $H(\dive,\zeta,\Omega):=\{ v \in L^2(\Omega) : \nabla \cdot (\zeta v) \in L^2(\Omega)\}$. 

One strategy for proving bounds on the Galerkin error for Maxwell -- first introduced by Monk \cite{Mo:92} -- is to 
\bit
\item[(i)] bound the $\epsilon$-divergence free part of the error using the duality arguments from \cite{Sc:74} (discussed above), and 
\item[(ii)] bound the part of the error that is not $\epsilon$-divergence free using arguments originating from \cite{Gi:88} (discussed below).
\eit
This argument is essentially equivalent to Lemma \ref{lem:basic} below. Notable uses of this type of argument include in \cite{GoPa:03}, in the analysis of Maxwell domain decomposition methods, and in \cite{BrPa:08}, in the analysis of the $h$-FEM with N\'ed\'elec elements applied to the Maxwell PML problem.

Regarding Point (ii) above:~by Galerkin orthogonality, the error is \emph{discretely $\epsilon$-divergence free}, in the sense that $(\epsilon(E-E_h), v_h)_{L^2(\Omega)}=0$ for all $v_h\in \Ker \curl\cap\Hilbert_h$ (see \eqref{eq:Gog2} below). Therefore, the part of the error that is not $\epsilon$-divergence free can be controlled by understanding how much a function that is discretely $\epsilon$-divergence free is not pointwise $\epsilon$-divergence free.
These arguments crucially rely on the existence of 
an interpolation operator that leaves the finite-element space invariant and maps functions in $\Ker \curl$ to functions in $\Ker \curl$ (see \S\ref{sec:interpolation} and Lemma \ref{lem:gammadv1} below). 
The initial versions of this argument in \cite{Gi:88,Mo:92} used standard interpolation operators, at the cost of demanding extra regularity of the Maxwell solution (see \cite[Remark 3.1]{Gi:88}). 
Later refinements of this argument \cite{GoPa:03, Mo:03a} then used 
quasi-interpolation operators with lower -- and, ultimately, minimal -- regularity assumptions; see \cite[\S4.1]{AmBeDaGi:98},
\cite{ArFaWi:00, Sc:01, Ch:07, ChWi:08}, \cite[Chapter 23]{ErGu:21}. 

\paragraph{Current state of the art for wavenumber-explicit bounds on the Maxwell $h$-FEM using N\'ed\'elec finite elements.} 

\bit
\item For real $\mu$ and $\epsilon$, the natural asymptotic error bounds are proved by the combination of \cite[Theorem 4.6, Lemma 5.2]{ChEr:23} and \cite[Theorem 2]{ChVe:22}.
\item The papers \cite{MeSa:21, MeSa:23} show that the $hp$-FEM applied to \eqref{eq:Maxwell} with constant $\mu$ and $\epsilon$ and analytic boundary 
does not suffer from the pollution effect if $p\geq C_1\log (\wn L)$ for any $C_1>0$,
$hk/p\leq C_2$ for sufficiently small $C_2>0$, 
 and $\Csol\leq C_3 (\wn L)^N$ for some $C_3, N>0$.
Although these analyses are geared towards $p$ growing with $k$,
the results in \cite{MeSa:23} for impedance boundary conditions  contain the result that, when $p$ is constant, the Galerkin solution is quasioptimal (with constant independent of $\wn$) when $(\wn h)^{p-1}\Csol$ is sufficiently small (i.e., one power of $\wn h$ away from the optimal result); see \cite[Proof of Lemma 9.5]{MeSa:23} (and note that $p=0$ corresponds to the lowest-order elements in \cite{MeSa:23}, instead of $p=1$ here).
The fixed-$p$ results in \cite{MeSa:21} for when the radiation condition is realised exactly on $\partial\Omega$ are more restrictive; see \cite[Remark 4.19]{MeSa:21}. These arguments essentially use a result equivalent to Lemma \ref{lem:basic} below, and then 
prove approximation results about the adjoint solution operator to bound the second quantity in \eqref{eq:asymptotic1} (following the ideas introduced in the Helmholtz context in \cite{MeSa:10, MeSa:11, EsMe:12, BeChMe:24}).

Analogous results for a regularised formulation of \eqref{eq:Maxwell} -- where the space is embedded in $H^1$ if the boundary is smooth enough -- were obtained in \cite{NiTo:20} (with the $h$-version of this method studied in a $\wn$-explicit way in \cite{NiTo:19}).

\item Very recently, the natural preasymptotic error bounds were proved in \cite[Theorem 4.2]{LuWu:24} when $p=1$ for \eqref{eq:Maxwell} with constant $\mu$ and $\epsilon$ and an impedance boundary condition on $\partial \Omega$, and when the $h$-FEM is implemented using N\'ed\'elec elements of the second family.
Recall that the second-family elements have better approximation properties in the $L^2$ norm than the first family (see, e.g., \cite[\S8.2]{Mo:03}), with this fact crucially used in \cite[Equation A.2]{LuWu:24}.
The analogous error bounds for continuous interior-penalty methods were proved in \cite[Theorem 5.2]{LuWu:24}. 
These results do not use the duality arguments described above; instead the crucial ingredient is a bound on the norm of the Galerkin solution in terms of the data; see \cite[Theorem 4.1]{LuWu:24} and the discussion in \cite[Remark 4.2]{LuWu:24}.

The results of \cite{LuWu:24} built on earlier work studying the same set up and proving the analogous result for other $h$-version FEMs, including
interior-penalty discontinuous Galerkin methods 
\cite[Theorem 6.1]{FeWu:14}, a different continuous interior penalty method using second-family N\'ed\'elec elements 
\cite[Theorem 4.6]{LuWuXu:19}, 
and hybridizable discontinuous Galerkin methods 
\cite[Theorem 4.7]{FeLuXu;16}, \cite[Remark 5.1]{LuChQi:17}.

Finally, we note that, since the preprint of the present paper appeared, \cite{LuWu:25} extended the results of \cite{LuWu:24} to $p>1$ using ideas from the present paper/\cite{GS3}. Indeed, \cite[Theorem 4.3]{LuWu:25} proved the natural preasymptotic error bounds for $p\in \mathbb{Z}^+$ for 
 \eqref{eq:Maxwell} with constant $\mu$ and $\epsilon$ and an impedance boundary condition on $\partial \Omega$, and when the $h$-FEM is implemented using N\'ed\'elec elements of the second family.
\eit 

\paragraph{Summary of the ideas behind the proof of Theorem \ref{thm:intro}.}
Theorem \ref{thm:intro} is proved by 
\bit
\item[(i)] bounding the $\epsilon$-divergence free part of the error using the ideas from the Helmholtz preasymptotic error analysis in \cite{GS3}, and
\item[(ii)] bounding the part of the error that is not $\epsilon$-divergence free using the arguments originating from \cite{Gi:88}. 
\eit
That is, compared to the classic duality argument introduced in \cite{Mo:92, Mo:03a} (and discussed above) we replace the Schatz argument by the arguments in \cite{GS3} and do everything in a $\wn$-explicit way.

Regarding Point (i):~we highlight that even applying the basic elliptic-projection argument (which \cite{GS3} generalises) to N\'ed\'elec-element discretisations of the time-harmonic Maxwell equations has proven difficult up to now, as described in \cite[Remark 4.2(d)]{LuWu:24}.
We use a projection 
$\Pi_0$ that maps into $\Ker \curl$, with then $\Pi_1:=I-\Pi_0$. A priori, there are many different choices for $\Pi_0$. However, the requirement that $\epsilon\Pi_1$ is $L^2$ orthogonal to $\Ker \curl$ (i.e., is $\epsilon$ divergence free) uniquely specifies $\Pi_0$; see Lemma \ref{lem:Pi0} (d). This lemma also shows that $\Pi_0$ is uniquely determined by its other key properties (see Lemma \ref{lem:Pi0} (b) and (c)).

Regarding Point (ii):~these arguments are performed in a $\wn$-explicit way for \eqref{eq:Maxwell} with $\mu$ and $\epsilon$ real-valued in \cite[\S3.3]{ChEr:23}; one slight difference between 
the arguments in \cite{ChEr:23} and those in the present paper is that \cite{ChEr:23} works in the $L^2$ inner product weighted with $\epsilon$, but this is not possible here since $\epsilon$ is complex.

Finally, we highlight that the duality arguments in the present paper have the splitting \eqref{eq:splitting} built in, so that only the 
adjoint solution operator applied to functions with high regularity appears; see \eqref{eq:fishchips1} and \eqref{eq:fishchips2} (and recall that the operator $\smoother$ is smoothing). 

\subsection{Outline}

\S\ref{sec:abstract} states the main result (i.e., Theorem \ref{thm:intro}) in abstract form (see Theorem \ref{thm:abs1} below). 
The proof of Theorem \ref{thm:abs1} is given in \S\ref{sec:abs_proof}; this proof uses intermediate results proved in \S\ref{sec:projections}-\S\ref{sec:Phash}.
The proof of Theorem \ref{thm:intro} is given in \S\ref{sec:Maxwell_proof}, using Theorem \ref{thm:abs1} and the material in \S\ref{sec:Weber} (a recap of the regularity results of Weber \cite{We:81}) and \S\ref{sec:FEM} (a recap of the definition and properties of N\'ed\'elec finite elements).
\S\ref{sec:PML} shows that the Maxwell PML problem  falls into the class of Maxwell problems described in \S\ref{sec:statement}.
\S\ref{app:interpolation} recaps scaling arguments used to prove interpolation results for N\'ed\'elec elements on curved meshes.

\section{The main result in abstract form}\label{sec:abstract}

We saw in \S\ref{sec:context} that sharp preasymptotic bounds for general $p$ have existed for certain Helmholtz problems for 10 years \cite{DuWu:15}, and the general strategy for obtaining the analogous bounds for Maxwell is clear:~in the classic duality argument introduced in \cite{Mo:92, Mo:03a}, replace 
the Schatz argument by these Helmholtz duality arguments. 
However, implementing this strategy has proved difficult, 
as noted recently in \cite[Remark 4.2(d)]{LuWu:24}.
The way we are able to achieve this, 
and obtain Theorem \ref{thm:intro}, is to work in an abstract framework that highlights the underlying mathematical structure of the problem. 
(We note that the generalisation of the preasymptotic bounds from the specific Helmholtz problems in \cite{DuWu:15, LiWu:19} to general Helmholtz problems and arbitrary polynomial degree was also achieved by working in an abstract framework \cite{GS3}.)

This section outlines this abstract framework, and states Theorem \ref{thm:intro} in abstract form as Theorem \ref{thm:abs1}. Both in this section, and in the rest of the paper, the links between the abstract framework and existing Maxwell $h$-FEM analyses are indicated in remarks and/or comments in the text.

\subsection{Abstract framework and assumptions}

Given a Hilbert space $\Hilbertzero$, let $\Hilbertzero^*$ denote the anti-dual space, and let $\langle\cdot,\cdot\rangle_{\Hilbertzero^*\times \Hilbertzero}$ be the duality pairing that is linear with respect to the first argument and anti-linear with respect to the second argument.

\begin{assumption}\label{ass:1}
$\Hilbert $ and $\Hilbertzero $ are Hilbert spaces with $\Hilbert  \subset \Hilbertzero$, 
$\Hilbert$ dense in $\Hilbertzero$, and norms $\vertiii{\cdot}_\Hilbert $ and $\vertiii{\cdot}_{\Hilbertzero }$. 
Given $C_1, C_2, C_{{\Ptwo }}'>0$, $P:\Hilbert\to \Hilbert^*$ and $\Ptwo : \Hilbertzero  \to \Hilbertzero^* $ with 
\beqs
\|\operator\|_{\Hilbert\to \Hilbert^*} 
+
\|\Ptwo\|_{\Hilbertzero\to \Hilbertzero^*} 
\leq C_2
\eeqs
and 
\beqs
  \Re\big\langle \Ptwo  v, v\big\rangle_{\Hilbertzero^*\times\Hilbertzero } \geq C_{{\Ptwo }}' \vertiii{v}^2_{\Hilbertzero } \quad\tfa v\in \Hilbertzero .
\eeqs
In addition, $\operator= \Pone -\Ptwo $ where 
$\Ker \Pone ^* = \Ker \Pone $ and 
\beqs
\Re\big\langle \Pone  v, v\big\rangle_{\Hilbert ^*\times \Hilbert } \geq C_1 \vertiii{v}^2_{\Hilbert } - C_2\vertiii{v}^2_{\Hilbertzero } \quad\tfa v \in \Hilbert .
\eeqs
\end{assumption}

We use later that if $\operator$ satisfies Assumption \ref{ass:1}, then so does $\operator^*$.

Let $\N{\cdot}_{\Hilbertzero } := \sqrt{C_2}\vertiii{\cdot}_{\Hilbertzero }$ and 
\beqs
\N{v}^2_\Hilbert  := \Re\big\langle \Pone  v, v\big\rangle_{\Hilbert ^*\times \Hilbert } + C_2\vertiii{v}^2_{\Hilbertzero };
\eeqs
to see that this is indeed the square of a norm on $\Hilbert$, note that the right-hand side can be written as 
$\langle \Re \cA v, v \rangle_{\Hilbert ^*\times \Hilbert}
:=\tfrac{1}{2}\langle (\cA+ \cA^*) v, v \rangle_{\Hilbert ^*\times \Hilbert}$ for $\cA$ equal to $\Pone$ plus $C_2$ multiplied by the appropriate Riesz map $\Hilbertzero\to\Hilbertzero^*$ in the inner product corresponding to $\vertiii{\cdot}_{\Hilbertzero}$.

These definitions imply that 
\beq\label{eq:Garding}
\Re\big\langle \Pone  v, v\big\rangle_{\Hilbert ^*\times \Hilbert } = \N{v}^2_{\Hilbert } - \N{v}^2_{\Hilbertzero } \quad\tfa v \in \Hilbert 
\eeq
(so that $( u,v)_\Hilbert  = \langle (\Re \Pone ) u,v\rangle_{\Hilbert ^*\times \Hilbert } +( u,v)_{\Hilbertzero }$ by the polarization identity)
and 
\beq\label{eq:P2}
\Re\big\langle\Ptwo  v, v\big\rangle_{\Hilbertzero^*\times\Hilbertzero } \geq C_{{\Ptwo }} \N{v}^2_{\Hilbertzero } \quad\tfa v\in \Hilbertzero 
\eeq
with $C_{{\Ptwo }}:= C_{{\Ptwo }}' (C_2)^{-1}$. Furthermore, by \eqref{eq:Garding}, 
\beq\label{eq:Garding2}
\Re\big\langle \operator v, v\big\rangle_{\Hilbert ^*\times \Hilbert } \geq \N{v}^2_{\Hilbert } - (1+ \N{\Ptwo}_{\Hilbertzero\to\Hilbertzero^*})\N{v}^2_{\Hilbertzero } \quad\tfa v \in \Hilbert .
\eeq

\ble\label{lem:kernel_closed}
$\Ker \Pone $ is closed in $\Hilbertzero $.
\ele

Since the proof of Lemma \ref{lem:kernel_closed} is short, we give it here.

\bpf[Proof of Lemma \ref{lem:kernel_closed}]
Let $\{ u_n\} \in \Ker \Pone $ with $u_n \to u$ in $\Hilbertzero $. We need to show that $u\in \Ker \Pone $. 
By \eqref{eq:Garding}, $\|u_n\|_{\Hilbert }= \|u_n\|_{\Hilbertzero }$. Since $u_n$ is bounded in $\Hilbertzero $, $u_n$ is bounded in $\Hilbert $. 
Since $\Hilbert $ is a Hilbert space, by passing to a subsequence, we see that there exists $w\in \Hilbert $ such that $u_n \rightharpoonup w$ as $n\to \infty$.
We now show that $w\in \Ker \Pone $. 
Let $\cR:\Hilbert^*\to\Hilbert$ be the Riesz map such that $\langle a, b\rangle_{\Hilbert \times \Hilbert ^*} = (a, \cR b)_\Hilbert $ for all $a\in \Hilbert,b\in \Hilbert^*$. 
Since $u_n\in \Ker \Pone $ for all $n$ and $u_n \rightharpoonup w$ as $n\to \infty$,  for all $v\in \Hilbert $, 
\begin{align*}
0= \langle \Pone  u_n, v\rangle_{\Hilbert ^*\times \Hilbert } &= \langle u_n, \Pone ^* v\rangle_{\Hilbert \times\Hilbert ^*} \\
&= (u_n, \cR \Pone ^*v)_\Hilbert  \to (w, \cR \Pone ^* v)_\Hilbert  = \langle w, \Pone ^* v\rangle_{\Hilbert  \times\Hilbert ^*} = \langle \Pone  w, v\rangle_{\Hilbert ^*\times \Hilbert }.
\end{align*}
Therefore $\Pone  w =0$, i.e., $w\in \Ker \Pone $. 

Since $\Ker \Pone ^*= \Ker \Pone $,  $(\Re \Pone ) u_n = (\Re \Pone )w=0$ and thus, since $( u,v)_\Hilbert  = \langle (\Re \Pone ) u,v\rangle_{\Hilbert ^*\times \Hilbert } +( u,v)_{\Hilbertzero }$,
\beqs
(u_n,v)_{\Hilbert } =(u_n,v)_{\Hilbertzero } \quad\tand\quad (w,v)_\Hilbert  = (w,v)_{\Hilbertzero } \quad \tfa v\in \Hilbert.
\eeqs
Therefore, on the one hand,  since $u_n \rightharpoonup w$ in $\Hilbert$ as $n\to \infty$, 
\beqs
(u_n,v)_{\Hilbertzero } =(u_n,v)_{\Hilbert }\to(w,v)_{\Hilbert }=(w,v)_{\Hilbertzero }\quad\tas n\to \infty.
\eeqs
On the other hand $(u_n,v)_{\Hilbertzero } \to (u,v)_{\Hilbertzero }$ since $u_n \to u$ in $\Hilbertzero $. Therefore 
$( w,v)_{\Hilbertzero }
= (u,v)_{\Hilbertzero }$ for all 
$v\in \Hilbert$;
thus $u=w \in \Ker \Pone $.
\epf

By Lemma \ref{lem:kernel_closed}, the $\Hilbertzero $-orthogonal projection onto $\Ker \Pone $ is well-defined; denote this
$\Pi_0^{\Hilbertzero }$ 
and let $\Pi_1^{\Hilbertzero }:= I- \Pi_0^{\Hilbertzero }$. 

Let $\iota:\Hilbertzero \to \Hilbertzero ^*$ be the Riesz map such that 
\beq\label{eq:Riesz}
\langle \iota u,v\rangle_{\Hilbertzero ^*\times \Hilbertzero } :=(u,v)_{\Hilbertzero } \quad\tfa u, v \in \Hilbertzero.
\eeq
We highlight that we write the identification of $\Hilbertzero$ and $\Hilbertzero ^*$ explicitly using $\iota$ because later we consider subspaces of $\Hilbertzero$ and $\Hilbertzero^*$ and need to write the identification of these in terms of the identification $\iota$; see \S\ref{sec:identify} and Part (ii) of Lemma \ref{lem:eta} below.

We now define two non-orthogonal projections $\Pi_0, \Pi_1:\Hilbertzero \to \Hilbertzero$. 
The action of $\iota^{-1}\Ptwo:\Hilbertzero\to \Hilbertzero$ with $\Hilbertzero =\Ker \Pone \oplus (\Ker \Pone )^{\perp}$ can be written as 
\beq\label{eq:matrixP2}
\begin{pmatrix}
\Pi_0^{\Hilbertzero } (\iota^{-1}\Ptwo) \Pi_0^{\Hilbertzero } & \Pi_0^{\Hilbertzero } (\iota^{-1}\Ptwo) \Pi_1^{\Hilbertzero } \\
\Pi_1^{\Hilbertzero } (\iota^{-1}\Ptwo) \Pi_0^{\Hilbertzero }  & \Pi_1^{\Hilbertzero } (\iota^{-1}\Ptwo) \Pi_1^{\Hilbertzero } 
\end{pmatrix}
=:
\begin{pmatrix}
\Ptwo ^{00} & \Ptwo ^{01}\\
\Ptwo ^{10} & \Ptwo ^{11}
\end{pmatrix}.
\eeq
The inequality \eqref{eq:P2} implies, in particular, that $\Ptwo ^{00}$ is invertible as a map from $\Ker \Pone$ to $\Ker \Pone$.

Let $\Pi_0, \Pi_1:\Hilbertzero \to \Hilbertzero$ be defined by 
\beq\label{eq:Pi0}
\Pi_0 := (\Ptwo ^{00})^{-1} \Pi_0^{\Hilbertzero } \iota^{-1}\Ptwo  \quad\tand\quad \Pi_1:= I- \Pi_0.
\eeq
By the matrix form of $\Ptwo $ above,
\beq\label{eq:Pi0matrix}
\Pi_0 = 
\begin{pmatrix}
I & (\Ptwo ^{00})^{-1}\Ptwo ^{01}\\
0 & 0
\end{pmatrix}
\quad\tand\quad
\Pi_1
=
\begin{pmatrix}
0& -(\Ptwo ^{00})^{-1}\Ptwo ^{01}\\
0 & I
\end{pmatrix}.
\eeq

We make the following two remarks:
\bit
\item Lemma \ref{lem:Pi0} below shows that an equivalent characterisation of $\Pi_0$ is that $\Pi_0: \Hilbertzero  \to \Ker \Pone $ is a projection satisfying 
$\Pi_0^{\Hilbertzero} (\iota^{-1}\Ptwo) \Pi_1=0$, 
and Remark \ref{rem:relation1} below shows how 
this characterisation implies that $\Pi_1$ projects to functions that are $\epsilon$-divergence free in the Maxwell case.
\item $\Pi_0$ and $\Pi_1$ depend on $\operator$, although we do not indicate this in the notation for brevity. Our arguments below use \emph{both} $\Pi_0$ and $\Pi_1$ \emph{and} the analogous projections with $\operator$ replaced by $\operator^*$ (note that, since $\Ker \Pone=\Ker \Pone^*$, replacing $\operator$ by $\operator^*$ amounts to replacing $\Ptwo$ by $\Ptwo^*$).
\eit

By \eqref{eq:Garding}, 
\beq\label{eq:simple_norm2}
\N{\Pi_0 v}_{\Hilbert } = \N{\Pi_0 v}_{\Hilbertzero } \quad\tfa v\in \Hilbert,
\eeq
so that, in particular, $\Pi_0:\Hilbert\to \Hilbert$ and $\Pi_1:= I-\Pi_0 :\Hilbert\to \Hilbert$ are both bounded.

\begin{assumption}[Abstract regularity assumptions]\label{ass:2}
Let $\cZ^0 =\Hilbertzero $, $\cZ^1 =\Hilbert$,  $\cZ^{j} \subset \cZ^{j-1}$ for $j=1,\ldots, \newell+1$, with 
$\Hilbertzero^{*}$ dense in $(\cZ^j)^*$ for $j\geq 1$. Let $\Creg>0$, 
$C_1, C_2, C_{{\Ptwo }}'>0$, and let $\operator$ satisfy Assumption \ref{ass:1} with these $C_1, C_2, C_{{\Ptwo }}'>0$.

(i) For $j=1,\ldots, \newell+1$,
\beq\label{eq:proj_reg}
\big\|\Pi_0^{\Hilbertzero}\big\|_{\cZ^j\to \cZ^j}\leq \Creg.
\eeq

(ii) With $\mathsf{D}$ equal $\Pone $ or $\Pone ^*$ or $\Re \Pone $, for $j=2,\ldots, \newell+1$,
\beq\label{eq:er1}
\N{\Pi_1 u}_{\cZ^j} \leq \Creg \Big( \N{\Pi_1 u}_{\Hilbertzero} + \sup_{v\in \Hilbert , \| \iota v\|_{(\cZ^{j-2})^*}=1} \big| \big\langle \mathsf{D} \Pi_1 u, 
\Pi_1 v\big\rangle_{\Hilbert^*\times\Hilbert}\big|
\Big)
\quad\tfa u \in \Hilbert, 
\eeq
with $\Pi_1 u \in \cZ^j$ if the right-hand side is finite.

(iii) With $\mathsf{E}$ equal $\iota^{-1}\Ptwo $ or $\iota^{-1}\Ptwo ^*$ or $\iota^{-1}\Re \Ptwo $, for $j=1,\ldots, \newell+1$,
\beq\label{eq:P2reg}
\N{\mathsf{E}}_{\cZ^j\to \cZ^j} 
\leq \Creg. 
\eeq

(iv) With $\mathsf{E}$ equal $\iota^{-1}\Ptwo $ or $\iota^{-1}\Ptwo ^*$, for $j=1, \ldots, \newell+1$, 
\beq\label{eq:CG1}
\big\|\Pi_0^{\Hilbertzero} u \big\|_{\cZ^j} \leq \Creg \Big( 
\big\| \Pi_0^{\Hilbertzero}\mathsf{E} \Pi_0^{\Hilbertzero}u\big\|_{\cZ^j} + \big\|\Pi_0^{\Hilbertzero}u \big\|_{\Hilbertzero}
\Big)\quad\tfa u\in\Hilbertzero.
\eeq
\end{assumption}

Given the partition $\{\Omega_j\}_{j=1}^n$ from Assumption \ref{ass:regularity}, let 
\beq\label{eq:pw_space}
\Hpw{j}(\Omega):= \Big\{ v\in L^2(\Omega)\, : \, \text{ for all multi-indices $\alpha$ with $|\alpha|\leq j$},\,\, \partial^\alpha (v|_{\Omega_i})\in L^2(\Omega_i)\Big\},
\eeq
and equip $\Hpw{j}(\Omega)$ with the norm 
\beq\label{eq:pw_norm}
\N{v}^2_{\Hpwo{j}(\Omega)}:= 
\sum_{|\alpha|\leq j}\sum_{i=1}^n \int_{\Omega_i} \big| \big(\wn ^{-1}\partial\big)^
\alpha \big( v|_{\Omega_i}\big) \big|^2.
\eeq

\begin{lemma}[Application to Maxwell]\label{lem:Maxwell}
Let $\Hilbertzero  = L^2(\Omega)$ and let $\Hilbert = H_0(\curl , \Omega)$ (i.e., functions in $H(\curl,\Omega)$ with zero tangential trace). 
Given matrix-valued functions $\mu$ and $\epsilon$ with 
\beq\label{eq:coefficients_sign}
\Re \mu^{-1} \geq c>0 \quad\tand\quad \Re \epsilon \geq c>0
\eeq
in $\Omega$ (in the sense of quadratic forms), 
let 
\beq\label{eq:Maxwell_PDE}
\Pone  := \wn ^{-2} \curl \mu^{-1}\curl \quad\tand\quad \Ptwo  := \epsilon.
\eeq
Let 
\beqs
\N{v}^2_{\Hilbertzero } = \N{v}^2_{L^2(\Omega)} \quad\tand\quad
\N{v}^2_{\Hilbert } = \wn ^{-2} \big\|(\Re\mu)^{-1/2} \curl v \big\|^2_{L^2(\Omega)} + \N{v}^2_{L^2(\Omega)}.
\eeqs

(a) Assumption \ref{ass:1} holds with $\Ker \Pone =\Ker \curl$, $\|\operator\|_{\Hilbert\to \Hilbert^*}$ independent of $\wn $, and 
 \eqref{eq:P2} 
satisfied with $C_{\Ptwo }=c$.

(b) Assumption \ref{ass:2} holds, for both $\operator$ and $\operator^*$, if $\Omega,\epsilon,$ and $\mu$ satisfy Assumption \ref{ass:regularity} and, with $\{\Omega_i\}_{i=1}^n$ as in Assumption \ref{ass:regularity},  
$\cZ^j=\newZ^j$ defined by
\beq\label{eq:Zj}
\newZ^j:= H_0(\curl,\Omega) \cap \Big\{ v \in L^2(\Omega) \,:\,v \in \Hpw{j-1}(\Omega)\,\tand\, \curl v \in \Hpw{j-1}(\Omega)\Big\}
\eeq
(observe that $\newZ^1 = H_0(\curl,\Omega)$) and equipped with the norm 
\beq\label{eq:Zjnorm}
\N{v}^2_{\newZ^j_\wn} := \N{v}_{H_\wn(\curl,\Omega)}^2 + \N{v}^2_{\Hpwo{j-1}(\Omega)} + \N{\wn^{-1}\curl v}^2_{\Hpwo{j-1}(\Omega)}.
\eeq
\end{lemma}

Lemma \ref{lem:Maxwell} is proved in \S\ref{sec:lemMaxwellProof} below. 

\bre[The regularity assumptions on $\epsilon,\mu$, and $\partial\Omega$]\label{rem:regularity}
In \S\ref{sec:lemMaxwellProofb} below, we see that 
\bit
\item 
\eqref{eq:proj_reg} 
\eqref{eq:er1},
and \eqref{eq:CG1}  hold when $\mu$ and $\epsilon$ are piecewise $C^{\newnewell}$
and the connected components of $\partial\Omega_j$, $j=1,\ldots,n$, are all $C^{\newnewell+1}$ -- this is shown
using the classic regularity results of Weber \cite{We:81}  (see Theorem \ref{thm:Weber} below) -- and 
 \item \eqref{eq:P2reg} holds when $\epsilon$ is piecewise $C^{\newnewell,1}$ by a standard inequality involving Sobolev norms; see, e.g., \cite[Theorem 1.4.1.1, page 21]{Gr:85}.
\eit
The combination of these requirements is then Assumption \ref{ass:regularity}.
\ere

\bre[$\Pi_1$ projects to functions that are $\epsilon$-divergence free]
\label{rem:relation1}
We show in \eqref{eq:key4} below that $\Pi_0$ can equivalently be defined by the property $\Pi_0^{\Hilbertzero} (\iota^{-1}\Ptwo) \Pi_1=0$.
Therefore, in the Maxwell setting of Lemma \ref{lem:Maxwell}, given $v \in L^2(\Omega)$, $\Pi_1 v\in L^2(\Omega)$ is the solution to 
\beq\label{eq:green_plant1}
\big( \epsilon \Pi_1 v, w\big)_{L^2(\Omega)} = 0 \quad\tfa w \in 
\Ker \curl;
\eeq
i.e.,
$\epsilon \Pi_1 $ is $L^2$ orthogonal to $\Ker \Pone= \Ker \curl$.
Since $\nabla H^1_0(\Omega)\subset \Ker \curl$, $\Pi_1$ projects, in particular, to functions that are $\epsilon$-divergence free.
Finally, \eqref{eq:green_plant1} shows that $\Pi_0$ is equivalent to the projection $\boldsymbol{\Pi}_0^{\rm c}$ defined by \cite[Equation 2.3]{ChEr:23} -- note that \cite{ChEr:23} consider the case when $\epsilon$ is real and weight the $L^2$ inner product with $\epsilon$.
\ere

Having defined the spaces $\newZ^j$, we now define the notion of $\wn$-oscillatory data used in Theorem \ref{thm:intro}.

\begin{definition}\label{def:oscil}
$f$ is $\wn$-oscillatory with constant $\Cosc>0$ and regularity index $\newell$ if one of the two following conditions holds.

(i) $f\in \newZ^{\newell+1}$ and 
\beq\label{eq:Cosc}
\N{f}_{\newZ^{\newell+1}_\wn}
\leq \Cosc \N{f}_{(H_\wn (\curl,\Omega)^*)}.
\eeq

(ii) $f\in 
\newZ^{\newell-1}$
with $\dive f=0$ and \eqref{eq:Cosc} holds with $\newell+1$ replaced by $\newell-1$. 
\end{definition}

\subsection{The Galerkin method}

Let $\Hilbert _h \subset \Hilbert $ be closed, and let $\Pi_h: \Hilbert \to\Hilbert _h$ be the orthogonal projection.
Given $u\in \Hilbert $, we seek an approximation of $u$, $u_h$, satisfying 
\beq\label{eq:Gog}
\big\langle \operator(u-u_h), v_h\big\rangle_{\Hilbert ^*\times \Hilbert } = 0 \quad \tfa v_h \in \Hilbert _h.
\eeq
Observe that, since $P=\Pone-\Ptwo$ and $\Ker \Pone=\Ker \Pone^*$, the Galerkin orthogonality \eqref{eq:Gog} implies that
\beq\label{eq:Gog2}
\big\langle \Ptwo (u-u_h), v_h\big\rangle_{\Hilbertzero^*\times\Hilbertzero} = 0 \quad \tfa v_h \in \Hilbert _h \cap \Ker \Pone .
\eeq

\subsection{The quantity $\gamma_{\rm dv}(\operator)$}

Let 
\begin{align}\nonumber
&\gamma_{\rm dv}(\operator):=\\
&\quad \sup \bigg\{
\frac{\N{\Pi_0 w_h}_{\Hilbertzero }}{\N{w_h}_{\Hilbert }} \, :\, w_h\in \Hilbert_h \text{ satisfies } \big\langle \Ptwo  w_h, v_h \big\rangle_{\Hilbertzero^*\times\Hilbertzero}=0 \,\,\tfa v_h \in \Hilbert _h \cap \Ker \Pone 
\bigg\}.
\label{eq:gamma_dv}
\end{align}
We write $\gamma_{\rm dv}(\operator)$, since we consider below both $\gamma_{\rm dv}(\operator)$ and $\gamma_{\rm dv}(\operator^*)$ (and we highlight again that $\Pi_0$ depends on $\operator$).

Note that if $v_h\in \Hilbert_h \cap \Ker \Pone$ in the definition of $\gamma_{\rm dv}(\operator)$ \eqref{eq:gamma_dv} is changed to $v\in \Hilbert\cap \Ker\Pone$, then $\gamma_{\rm dv}(\operator)=0$. Indeed, if $w$ satisfies $( \iota^{-1}\Ptwo  w, v)_{\Hilbertzero}=0$ for all $v\in \Ker \Pone $, 
then $\iota^{-1}\Ptwo w= \Pi_1^{\Hilbertzero } z$ for some $z\in \Hilbertzero$ (since $\Hilbertzero = (\Ker \Pone)\oplus (\Ker \Pone)^\perp$);
i.e. $w=(\iota^{-1}\Ptwo)^{-1} \Pi_1^{\Hilbertzero } z$ for some $z\in \Hilbertzero$.
 Then, by \eqref{eq:key3} below,  $\Pi_0 w=0$. 

Comparing \eqref{eq:Gog2} and \eqref{eq:gamma_dv}, we see that, since the Galerkin error $u-u_h \not \in \Hilbert_h$, $u-u_h$ is not contained 
in the set of $w_h$ considered in \eqref{eq:gamma_dv}. Nevertheless, controlling $\gamma_{\rm dv}(\operator)$ gives us a way to control $\Pi_0(u-u_h)$,
with Lemma \ref{lem:gamma_dv} below showing that, for a certain projection $\Pi_h^+$,  $\Pi_0 \Pi^+_h(u-u_h)$ is controlled by $\gamma_{\rm dv}(\operator)$, and then 
Lemma \ref{lem:Pi0error} controlling $\Pi_0(u-u_h)$.

\bre[$\gamma_{\rm dv}(\operator)$ is the divergence conformity factor]\label{rem:relation2}
In the Maxwell setting of Lemma \ref{lem:Maxwell},   
the $w_h$ considered in 
\eqref{eq:gamma_dv} are discretely $\epsilon$-divergence free. 
By Remark \ref{rem:relation1}, 
if $\Pi_1 w_h=w_h$ (i.e., $\Pi_0w_h=0$) then $w_h$ is $\epsilon$-divergence free. The quantity $\gamma_{\rm dv}(\operator)$
 is therefore the familiar \emph{divergence conformity factor}, 
measuring how much a finite-element function that is discretely $\epsilon$-divergence-free is not pointwise $\epsilon$-divergence-free, with 
this mismatch 
central to the analysis of the Maxwell FEM using N\'ed\'elec elements, as discussed in \S\ref{sec:context}; see also 
\cite[Lemma 5.2]{ArFaWi:00}, \cite[Lemma 4.5]{Hi:02}, \cite[Lemma 7.6]{Mo:03}, and 
\cite[Lemma 5.2]{ChEr:23} (with the notation $\gamma_{\rm dv}$ taken from \cite[Equation 3.9]{ChEr:23}).
\ere

\subsection{The main abstract theorem}\label{sec:abs_thm}

\begin{theorem}[The main result in abstract form]\label{thm:abs1}
Fix the constants in Assumptions \ref{ass:1} and \ref{ass:2} (i.e., $C_{{\Ptwo }}', C_1, C_2$ in Assumption \ref{ass:1} and $\Creg$ in Assumption \ref{ass:2}), fix $\newell \in\mathbb{Z}^+$ and spaces $\cZ^j, j=1,\ldots,\newell+1$, and let $\Coscil, C_{\Pi_0}>0$. 
Then there exist $c,C>0$ such that for all $\operator$ such that 
\bit
\item
Assumption \ref{ass:1} holds for $\operator$ (with $C_{{\Ptwo }}', C_1, C_2$ fixed above), 
\item
Assumption \ref{ass:2} holds for both $\operator$ and $\operator^*$ (with $C_{{\Ptwo }}', C_1, C_2, \Creg$, $\newell$, and $\cZ^j, j=1,\ldots,\newell+1$ fixed above) and, 
\item given $f\in \Hilbert^*$, the solution to 
the equation $Pu =f$ 
exists and is unique,
\eit
if
\beq\label{eq:sufficiently_small}
\gamma_{\rm dv}(\operator) + \gamma_{\rm dv}(\operator^*)
+\Big(\N{I-\Pi_h}_{\cZ^{\newell+1}\to \Hilbert }\Big)^2 \big( 1 + \N{\Rs\Pi_1^*}_{\Hilbertzero^* \to \Hilbertzero }\big)\leq c ,
\eeq
then $u_h$ defined by \eqref{eq:Gog} exists, is unique, and satisfies
\beq\label{eq:qo_abs}
\N{u-u_h}_{\Hilbert } \leq 
C\Big( 1 + \N{I-\Pi_h}_{\cZ^{\newell+1}\to \Hilbert } \big( 1 + \N{\Rs\Pi_1^*}_{\Hilbertzero^* \to \Hilbertzero }\big)\Big)
\N{(I-\Pi_h)u}_{\Hilbert }.
\eeq

In addition, 
if 
$\|\iota^{-1}f \|_{\cZ^{\newell+1}} \leq \Coscil \|f\|_{\Hilbert^*}$ 
and \eqref{eq:sufficiently_small} holds, then 
\beq\label{eq:relative_error_abs}
\frac{\N{u-u_h}_{\Hilbert }
}{
\N{u}_{\Hilbert}
}
 \leq 
C\Big( 1 + \N{I-\Pi_h}_{\cZ^{\newell+1}\to \Hilbert } \big( 1 + \N{\Rs\Pi_1^*}_{\Hilbertzero^* \to \Hilbertzero }\big)\Big)
\N{I-\Pi_h}_{\cZ^{\newell+1}\to \Hilbert }.
\eeq
Furthermore, if $\Pi_0^*f=\widetilde\Pi_0^*f$ for some $\widetilde\Pi_0:\Hilbertzero\to \Hilbertzero$ satisfying $\Pi_0 \widetilde\Pi_0 = \widetilde\Pi_0$ (i.e., $\widetilde\Pi_0$ maps into $\Ker \Pone$) and 
\beq\label{eq:tired2}
\|\iota^{-1}\widetilde\Pi_0^*\iota\|_{\cZ^{\newell-1}
\to \cZ^{\newell+1}}\leq C_{\Pi_0},
\eeq
then 
the assumption 
$\|\iota^{-1}f \|_{\cZ^{\newell+1}} \leq \Coscil \|f\|_{\Hilbert^*}$ can be relaxed to $\|\iota^{-1} f \|_{\cZ^{\newell-1}} \leq \Coscil \|f\|_{\Hilbert^*}.$
\end{theorem}

Theorem \ref{thm:abs1} is proved in \S\ref{sec:abs_proof} below. We make three remarks:

(i) By the order of quantifiers in Theorem \ref{thm:abs1}, $c$ and $C$
in the theorem depend on  
$C_{\Ptwo}'$, $C_1$, and $C_2$ in Assumption \ref{ass:1}, $\Creg$ in Assumption \ref{ass:2}, and the fixed constants $\Coscil$ and $C_{\Pi_0}$. 
In the proof of Theorem \ref{thm:abs1}, $C_{\Ptwo}'$, $C_1$, $C_2$, $\Creg$, $\Coscil$ and $C_{\Pi_0}$ are fixed at the outset, but, for brevity, this is not stated explicitly in each intermediate result. In the proofs, the letter $C$ denotes a constant that, in principle, depends on $C_{\Ptwo}'$, $C_1$, $C_2$, $\Creg$, $\Coscil$ and $C_{\Pi_0}$, but nothing else.

(ii) The additional projection $\widetilde\Pi_0^*$ in the last part of the theorem caters for the fact that, in the Maxwell setting, the kernel of the curl does not only consist of gradients when $\partial\Omega$ has more than one connected component (see \eqref{eq:KN} below), but the condition that $\dive f= \dive(\epsilon E)=0$ (i.e., $f$ is orthogonal to gradients) is nevertheless enough for $E$ to gain regularity with respect to $f$ (see Theorem \ref{thm:Weber} below). 

(iii) The relative-error bound \eqref{eq:relative_error_abs} follows from the preasymptotic error bound \eqref{eq:qo_abs} and 
the following regularity result (proved in \S\ref{sec:regularity} below). 

\ble[$\wn $-oscillatory data implies $\wn $-oscillatory solution]\label{lem:oscil}
Suppose that 
Assumption \ref{ass:2} holds  for some $\newell \in\mathbb{Z}^+$ and spaces $\cZ^j, j=1,\ldots,\newell+1$. 
Given $\Cosc>0$ 
there exists $C'>0$ such that the following is true. If $\operator u=f$ with $f\in \Hilbertzero^*$ satisfying
\beq\label{eq:oscil}
\big\|\iota^{-1} f \big\|_{\cZ^{\newell+1}} \leq \Cosc \N{f}_{\Hilbert^*}, \quad\text{ then } \quad 
\N{u}_{\cZ^{\newell+1}} \leq C' \N{u}_{\Hilbert}.
\eeq
Furthermore, if $\Pi_0^*f=\widetilde\Pi_0^*f$ for some $\widetilde\Pi_0:\Hilbertzero\to \Hilbertzero$ satisfying $\Pi_0 \widetilde\Pi_0 = \widetilde\Pi_0$ and 
\eqref{eq:tired2},
then
the assumption 
$\|\iota^{-1} f \|_{\cZ^{\newell+1}} \leq \Cosc \|f\|_{\Hilbert^*}$ can be relaxed to $\|\iota^{-1}f \|_{\cZ^{\newell-1}} \leq C \|f\|_{\Hilbert^*}.$
\ele

\section{Properties of $\Pi_0$ and $\Pi_1$}\label{sec:projections}

By its definition \eqref{eq:Pi0}, $\Pi_0: \Hilbertzero \to \Ker\Pone$. Since $\Ker \Pone =\Ker \Pone^*$, 
\beq\label{eq:key1}
\Pone =\Pone  \Pi_1=\Pi_1^* \Pone =\Pi_1^* \Pone  \Pi_1.
\eeq

\ble[Properties and equivalent definitions of $\Pi_0$]
\label{lem:Pi0}
The following are equivalent

(a) 
$\Pi_0 := (\Ptwo ^{00})^{-1} \Pi_0^{\Hilbertzero } (\iota^{-1}\Ptwo) $;
i.e., $\Pi_0$ is given by \eqref{eq:Pi0matrix}. 

(b) $\Pi_0: \Hilbertzero  \to \Ker \Pone $ is a projection satisfying 
\beq\label{eq:key2}
\Pi_0^* \Ptwo  \Pi_1=0.
\eeq

(c) $\Pi_0: \Hilbertzero  \to \Ker \Pone $ is a projection satisfying 
\beq\label{eq:key3}
\Pi_0 (\iota^{-1}\Ptwo)^{-1}\Pi_1^{\Hilbertzero}=0.
\eeq

(d) $\Pi_0: \Hilbertzero  \to \Ker \Pone $ is a projection satisfying 
\beq\label{eq:key4}
\Pi_0^{\Hilbertzero} (\iota^{-1}\Ptwo) \Pi_1=0.
\eeq
\ele

We highlight that the property \eqref{eq:key2} is essential in the duality arguments below (since it means 
that the matrix representation of $\operator$ as a map $(\Pi_0 \Hilbert , \Pi_1 \Hilbert )\to(\Pi_0^* \Hilbert ^*, \Pi_1^* \Hilbert ^*)$ is lower triangular; see \eqref{eq:matrices} below). The property \eqref{eq:key3} is essential for $\gamma_{\rm dv}(\operator)$ to $\to 0$ as $\Hilbert_h\to \Hilbert$ (as explained in the text after \eqref{eq:gamma_dv}).
Finally, recall from Remark \ref{rem:relation1} that \eqref{eq:key4} implies, in the Maxwell setting,  that
$\epsilon \Pi_1 $ is $L^2$ orthogonal to $\Ker \Pone= \Ker \curl$ and thus $\Pi_1$ projects, in particular, to functions that are $\epsilon$-divergence free.

\bpf[Proof of Lemma \ref{lem:Pi0}]
We prove that (a) implies (b), (c), and (d), and then that each of (b), (c), and (d) imply (a). 
(a) immediately implies (c) since $\Pi_0^{\Hilbertzero}\Pi_1^{\Hilbertzero}=0$.
To see that (a) implies (d), 
observe that 
\beq\label{eq:coffee1}
\iota^{-1}\Ptwo  \Pi_1 = 
\begin{pmatrix}
\Ptwo ^{00} & \Ptwo ^{01}\\
\Ptwo ^{10} & \Ptwo ^{11}
\end{pmatrix}
\begin{pmatrix}
0& -(\Ptwo ^{00})^{-1}\Ptwo ^{01}\\
0 & I
\end{pmatrix}
=
\begin{pmatrix}
0&0\\
0 & -\Ptwo ^{10} (\Ptwo ^{00})^{-1} \Ptwo ^{01} + \Ptwo ^{11}
\end{pmatrix},
\eeq
so that (d) holds. To see that (a) implies (b), 
we first claim that 
\beq\label{eq:iota_matrix}
\iota = \begin{pmatrix}
\iota & 0\\
0 &\iota \\
\end{pmatrix}
\eeq
as a map from $\Hilbertzero = \big(\Pi^{\Hilbertzero}_0 \Hilbertzero, \Pi_1^{\Hilbertzero}\Hilbertzero\big)$ to $\Hilbertzero^* = \big((\Pi^{\Hilbertzero}_0)^* \Hilbertzero^*, (\Pi_1^{\Hilbertzero})^*\Hilbertzero^*\big)$.
Indeed, since $\Pi_0^{\Hilbertzero}$ and $\Pi_1^{\Hilbertzero}$ are 
are $\Hilbertzero$-orthogonal projections, they are self-adjoint in $(\cdot,\cdot)_{\Hilbertzero}$. Therefore, since also $\Pi_0^{\Hilbertzero} \Pi_1^{\Hilbertzero}=0$,
\begin{align*}
\langle \iota u,v \rangle_{\Hilbertzero^*\times \Hilbertzero} = (u,v)_{\Hilbertzero} &= (\Pi_0^{\Hilbertzero} u, \Pi_0^{\Hilbertzero} v)_{\Hilbertzero} + (\Pi_1^{\Hilbertzero} u, \Pi_1^{\Hilbertzero} v)_{\Hilbertzero}\\
&= \langle\iota \Pi_0^{\Hilbertzero} u, \Pi_0^{\Hilbertzero} v\rangle_{\Hilbertzero^*\times\Hilbertzero} + \langle\iota \Pi_1^{\Hilbertzero} u, \Pi_1^{\Hilbertzero} v\rangle_{\Hilbertzero^*\times\Hilbertzero},
\end{align*}
which implies \eqref{eq:iota_matrix}.
The combination of \eqref{eq:coffee1} and \eqref{eq:iota_matrix} implies that 
\begin{align*}
\Pi_1^* \Ptwo  \Pi_1 &= 
\begin{pmatrix}
0& 0\\
-(\Ptwo ^{01})^*((\Ptwo ^{00})^{-1})^* & I
\end{pmatrix}
\begin{pmatrix}
\iota & 0\\
0 &\iota \\
\end{pmatrix}
\begin{pmatrix}
0&0\\
0 & -\Ptwo ^{10} (\Ptwo ^{00})^{-1} \Ptwo ^{01} + \Ptwo ^{11}
\end{pmatrix}\\
&=
\begin{pmatrix}
0&0\\
0 & \iota(-\Ptwo ^{10} (\Ptwo ^{00})^{-1} \Ptwo ^{01} + \Ptwo ^{11})
\end{pmatrix}
=\Ptwo  \Pi_1,
\end{align*}
i.e., (a) implies (b). 

For (b) implies (a):~since $\Pi_0^2=\Pi_0$ and $\Pi_0 :\Hilbertzero  \to \Ker \Pone $, the matrix representation of $\Pi_0$ 
as a map from $\Hilbertzero = \big(\Pi^{\Hilbertzero}_0 \Hilbertzero, \Pi_1^{\Hilbertzero}\Hilbertzero\big)$ to itself 
is 
\beq\label{eq:generalPi0}
\Pi_0 = 
\begin{pmatrix}
I & A\\
0 & 0
\end{pmatrix}
\eeq
for some $A$. Then, by a similar calculation to that in \eqref{eq:coffee1},
\beqs
\Pi_0^* \Ptwo  \Pi_1 = 
\begin{pmatrix}
0 & \iota(-\Ptwo ^{00} A + \Ptwo ^{01}) \\
0 & A^*\iota(-\Ptwo ^{00} A + \Ptwo ^{01}) 
\end{pmatrix},
\eeqs
so that (b) implies that $A= (\Ptwo ^{00})^{-1}\Ptwo ^{01}$ (i.e., (a) holds).

Similarly, for (d) implies (a):~if $\Pi_0$ is given by \eqref{eq:generalPi0}, then
\beqs
\Pi_0^{\Hilbertzero} (\iota^{-1}\Ptwo) \Pi_1 = \begin{pmatrix}
0 & -\Ptwo ^{00} A + \Ptwo ^{01} \\
0 & 0
\end{pmatrix},
\eeqs
so that, again, $A= (\Ptwo ^{00})^{-1}\Ptwo ^{01}$.

For (c) implies (a):~$\Pi_0 (\iota^{-1}\Ptwo)^{-1} \Pi_1^{\Hilbertzero }=0$ implies that $\Pi_0 (\iota^{-1}\Ptwo) ^{-1} = \Pi_0 (\iota^{-1}\Ptwo)^{-1} \Pi_0^{\Hilbertzero }$, so that 
$\Pi_0 = \Pi_0(\iota^{-1} \Ptwo)^{-1} \Pi_0^{\Hilbertzero }(\iota^{-1}\Ptwo)$; i.e., 
\beq\label{eq:wine1}
\Pi_0 = B \Pi_0^{\Hilbertzero }(\iota^{-1}\Ptwo) \quad \text{ for some $B:\Hilbertzero\to \Hilbertzero$.}
\eeq
We'll show that $B= (\Ptwo ^{00})^{-1}$ to complete the proof. Since $\Pi_0, \Pi_0^{\Hilbertzero}: \Hilbertzero  \to \Ker \Pone $, \eqref{eq:wine1} implies that $B: \Ker \Pone \to \Ker \Pone $. Furthermore, by \eqref{eq:generalPi0}, 
\beq\label{eq:wine2}
\Pi_0^{\Hilbertzero }=\Pi_0 \Pi_0^{\Hilbertzero }.
\eeq
The combination of \eqref{eq:wine1} and \eqref{eq:wine2} implies that 
\beqs
\Pi^{\Hilbertzero }_0 =B  \Pi_0^{\Hilbertzero }(\iota^{-1}\Ptwo)  \Pi_0^{\Hilbertzero } = B \Ptwo ^{00}
\eeqs
(by the definition of $\Ptwo^{00}$ in \eqref{eq:matrixP2}); 
i.e., on $\Ker \Pone $, $B \Ptwo ^{00}$ is the identity. Therefore, $B = (\Ptwo ^{00})^{-1}$ as an operator $\Ker \Pone  \to \Ker \Pone $ and the proof is complete.
\epf

\ble[$\Pi_0$ and $\Pi_1$ preserve regularity]\label{lem:Pireg}
If Assumption \ref{ass:2} holds then there exists $C>0$ such that for $j=0,\ldots, \newell+1$, 
\beq\label{eq:sun2}
\N{\Pi_0}_{\cZ^{j}\to \cZ^{j}} 
\leq C 
\eeq
and
\beq\label{eq:sun1}
\big\|\iota^{-1}\Pi_0^*\iota\big\|_{\cZ^{j}\to \cZ^{j}}
=\big\|\iota \Pi_0 \iota^{-1}\big\|_{(\cZ^{j})^*\to (\cZ^{j})^*} 
\leq C,
\eeq
with analogous bounds holding for $\Pi_1$ since $\Pi_0= I-\Pi_1$.
\ele
\bpf
By the definitions of $\Pi_0$ \eqref{eq:Pi0}, $\Ptwo^{00}$ \eqref{eq:matrixP2}, and $\Ptwo^{01}$, to prove \eqref{eq:sun2} it is sufficient to prove that
\beq\label{eq:CG3}
\big( \Pi^{\Hilbertzero}_0 (\iota^{-1}\Ptwo) \Pi^{\Hilbertzero}_0\big)^{-1} \big( \Pi^{\Hilbertzero}_0 (\iota^{-1}\Ptwo) \Pi^{\Hilbertzero}_1\big) : \Pi^{\Hilbertzero}_1 \cZ^j \to \Pi^{\Hilbertzero}_0\cZ^j.
\eeq
By \eqref{eq:proj_reg} and \eqref{eq:P2reg} (with $\mathsf{E}=\iota^{-1}\Ptwo$), 
\beqs
\big\| \Pi^{\Hilbertzero}_0(\iota^{-1}\Ptwo) \Pi^{\Hilbertzero}_1\big\|_{\cZ^j \to \cZ^j} \leq C.
\eeqs
Therefore, to prove \eqref{eq:sun2} it is sufficient to prove that 
\beq\label{eq:CG2}
\big\| \big( \Pi^{\Hilbertzero}_0(\iota^{-1}\Ptwo) \Pi^{\Hilbertzero}_0\big)^{-1}\big\|_{\Pi^{\Hilbertzero}_0\cZ^j \to \Pi^{\Hilbertzero}_0\cZ^j} \leq C.
\eeq
However, \eqref{eq:CG2} follows from \eqref{eq:CG1} with $\mathsf{E}= \iota^{-1}\Ptwo$ and the fact that $(\Pi^{\Hilbertzero}_0(\iota^{-1}\Ptwo)\Pi^{\Hilbertzero}_0)^{-1} : \Pi^{\Hilbertzero}_0\Hilbertzero\to \Pi^{\Hilbertzero}_0\Hilbertzero$ is bounded by \eqref{eq:P2}.

For \eqref{eq:sun1}, by \eqref{eq:iota_matrix} and the definition of $\Pi_0$,
\beqs
\iota^{-1}\Pi^*_0 \iota = 
\begin{pmatrix}
I & 0 \\
\iota^{-1} (\Ptwo^{01})^* \big( (\Ptwo^{00})^*\big)^{-1}\iota &0
\end{pmatrix}.
\eeqs
Now, by \eqref{eq:Riesz}, $\iota^{*}= \iota$ and $(\Pi^{\Hilbertzero}_0)^*=\iota \Pi^{\Hilbertzero}_0 \iota^{-1}$. Therefore, by \eqref{eq:matrixP2},
\beqs
\iota^{-1} (\Ptwo^{01})^* \big( (\Ptwo^{00})^*\big)^{-1}\iota = \Pi^{\Hilbertzero}_1( \iota^{-1}\Ptwo^*) \Pi^{\Hilbertzero}_0 \big( \Pi^{\Hilbertzero}_0(\iota^{-1}\Ptwo^*) \Pi^{\Hilbertzero}_0\big)^{-1}.
\eeqs
Comparing this last expression to \eqref{eq:CG3}, we see that the bound
\eqref{eq:sun1} follows in a similar way to \eqref{eq:sun2}, now using \eqref{eq:CG1} with $\mathsf{E}= \iota^{-1}\Ptwo^*$. 
\epf

\section{Matrix representation of $\operator$, regularity shift of $\Rs\Pi_1^*$, and proof of Lemma \ref{lem:oscil}}\label{sec:regularity}

The combination of $\Pi_0+\Pi_1=I$ and either $\Pi_0 \Pi_1=0$ or $\Pi_j^2 =\Pi_j, j=1,2$, implies that, for all $f\in \Hilbert ^*$ and $v\in\Hilbert $,
\beqs
\big\langle f , v \big\rangle_{\Hilbert ^* \times\Hilbert } = \big\langle \Pi_0^* f, \Pi_0 v\big\rangle_{\Hilbert ^* \times\Hilbert } + \big\langle \Pi_1^* f, \Pi_1 v\big\rangle_{\Hilbert ^* \times\Hilbert }.
\eeqs
Thus, given $A:\Hilbert \to \Hilbert ^*$, for all $u,v\in \Hilbert $,
\beqs
\big\langle Au, v \big\rangle_{\Hilbert ^* \times\Hilbert } = \big\langle \Pi_0^* A (\Pi_0 u + \Pi_1u), \Pi_0 v\big\rangle_{\Hilbert ^* \times\Hilbert } + \big\langle \Pi_1^* A (\Pi_0 u + \Pi_1u), \Pi_1 v\big\rangle_{\Hilbert ^* \times\Hilbert }.
\eeqs
Therefore, given $A:\Hilbert \to\Hilbert ^*$, its matrix representation as a map $(\Pi_0 \Hilbert , \Pi_1 \Hilbert )\to(\Pi_0^* \Hilbert ^*, \Pi_1^* \Hilbert ^*)$ is
\beq\label{eq:matrixnotation}
\begin{pmatrix}
\Pi_0^* A \Pi_0 & \Pi_0^* A\Pi_1 \\
\Pi_1^* A\Pi_0  & \Pi_1^* A\Pi_1 
\end{pmatrix}
=:
\begin{pmatrix}
A_{00} & A_{01}\\
A_{10} & A_{11}
\end{pmatrix}
\eeq
With this notation, by \eqref{eq:key1} and \eqref{eq:key2},
\beq\label{eq:matrices}
\Pone = 
\begin{pmatrix}
0& 0\\
0& \Pone_{11}
\end{pmatrix},
\quad
\Ptwo  = 
\begin{pmatrix}
\Ptwo_{00} & 0\\
\Ptwo_{10} & \Ptwo_{11}
\end{pmatrix}
\quad\text{ and thus }\quad 
\operator = 
\begin{pmatrix}
-\Ptwo_{00} & 0\\
-\Ptwo_{10} & \Pone_{11} -\Ptwo_{11}
\end{pmatrix}.
\eeq

The main result of this section is the following regularity shift for $\Rs\Pi_1^*\iota$.

\ble[Regularity of $\Rs\Pi_1^*\iota$]\label{lem:regRs}
If $\Rs$ exists and Assumption \ref{ass:2} holds, then there exists $C>0$ such that 
\beqs
\N{ \Rs \Pi_1^* \iota}_{\cZ^{j-2}\to \cZ^j} \leq C \big(1 + \N{\Rs \Pi_1^*}_{\Hilbertzero^* \to \Hilbertzero }\big)\quad\tfor j=2,\ldots,\newell+1.
\eeqs
\ele

To prove Lemma \ref{lem:regRs}, we first show that Part (iv) of Assumption \ref{ass:2} (i.e., \eqref{eq:CG1})
and the coercivity of $\Ptwo$ \eqref{eq:P2} imply the following result.

\ble\label{lem:Berlin1}
If  Assumption \ref{ass:2} holds then there exists $C>0$ such that, for $j=0,\ldots,m+1$, 
\beq\label{eq:tired1}
\N{\Pi_0 u}_{\cZ^j} \leq C\min\big\{ \N{ \iota^{-1}\Ptwo_{00} \Pi_0 u}_{\cZ^j}, \N{ \iota^{-1}\Ptwo_{00}^* \Pi_0 u}_{\cZ^j}\big\}
\quad \tfa u \in \Hilbertzero.
\eeq
\ele

\bpf
We first prove the bound in \eqref{eq:tired1} involving $\Ptwo_{00}$.
By \eqref{eq:Pi0matrix}, $\Pi_0^{\Hilbertzero} \Pi_0=\Pi_0$. Therefore, by \eqref{eq:CG1} with $\mathsf{E}= \iota^{-1}\Ptwo$, 
\begin{align}\nonumber
\N{\Pi_0 u }_{\cZ^j} = \big\|\Pi_0^{\Hilbertzero} \Pi_0 u \big\|_{\cZ^j} &\leq C \Big( \big\|\Pi_0^{\Hilbertzero} (\iota^{-1}\Ptwo) \Pi_0^{\Hilbertzero} \Pi_0 u \big\|_{\cZ^j} + \big\| \Pi_0^{\Hilbertzero} \Pi_0 u\big\|_{\Hilbertzero}\Big) \\
&= C \Big( \big\|\Pi_0^{\Hilbertzero} (\iota^{-1}\Ptwo) \Pi_0 u \big\|_{\cZ^j} + \big\| \Pi_0 u\big\|_{\Hilbertzero}\Big).
\label{eq:CG4}
\end{align}
Now, by \eqref{eq:Pi0matrix} and \eqref{eq:iota_matrix}, $ \Pi_0^{\Hilbertzero} \iota^{-1} \Pi_0^* =  \Pi_0^{\Hilbertzero}\iota^{-1}$. By this and the definition $\Ptwo_{00}:= \Pi_0^* \Ptwo \Pi_0$ \eqref{eq:matrixnotation}, 
\beqs
 \Pi_0^{\Hilbertzero} \iota^{-1} \Ptwo_{00} \Pi_0 u =  \Pi_0^{\Hilbertzero}\iota^{-1} \Pi_0^* \Ptwo\Pi_0u =  
 \Pi_0^{\Hilbertzero} (\iota^{-1}\Ptwo) \Pi_0 u,
\eeqs
By the last displayed equation, \eqref{eq:CG4}, and \eqref{eq:proj_reg}, 
\beq\label{eq:CG5}
\N{\Pi_0 u }_{\cZ^j} \leq C\Big( \big\|\iota^{-1}\Ptwo_{00} \Pi_0 u\big\|_{\cZ^j} +  \big\| \Pi_0 u\big\|_{\Hilbertzero}\Big).
\eeq
To remove the second term on the right-hand side of \eqref{eq:CG5} and obtain the bound in \eqref{eq:tired1}  involving $\Ptwo_{00}$, we use 
\eqref{eq:P2} to obtain that
\beq\label{eq:CG6}
C_{\Ptwo} \N{\Pi_0 u}^2_{\Hilbertzero} \leq 
\big| \big\langle \Ptwo_{00} \Pi_0 u ,\Pi_0 u \big\rangle_{\Hilbertzero^*\times\Hilbertzero}\big|
\leq \big\|\iota^{-1}\Ptwo_{00} \Pi_0 u\big\|_{\Hilbertzero} \N{ \Pi_0 u}_{\Hilbertzero};
\eeq
the result then follows by combining \eqref{eq:CG5} and \eqref{eq:CG6}. The bound in \eqref{eq:tired1} involving $\Ptwo_{00}^*$ follows in an analogous way, now using \eqref{eq:CG1} with $\mathsf{E} = \iota^{-1}\Ptwo^*$. 
\epf

\bpf[Proof of Lemma \ref{lem:regRs}]
Let $\newf\in \cZ^{j-2}\subset \Hilbertzero$ and let $u= \Rs \Pi_1^* \iota \newf\in \Hilbert$ so that $\operator^* u = \Pi_1^*\iota \newf\in\Hilbert^*$.
The idea of the proof is to use \eqref{eq:er1} to obtain a regularity-shift-like bound on $\Pi_1 u$ (\eqref{eq:birthday2} below) and then show that $\Pi_0 u$ inherits the regularity of $\Pi_1 u$ via  the equation $\Pi_0^*\operator^* u =0$ and \eqref{eq:tired1}
 (see \eqref{eq:birthday4} and \eqref{eq:birthday2a} below).
 
In preparation for applying \eqref{eq:er1} with $\mathsf{D}= \Re\Pone^*$, we observe that,
by \eqref{eq:key1} and \eqref{eq:Riesz}, for $v\in \Hilbert\subset\Hilbertzero$, 
\begin{align}\nonumber
&\big| \big\langle \Pone ^* \Pi_1 u ,\Pi_1 v\big\rangle_{\Hilbert^*\times \Hilbert}\big| 
=\big| \big\langle \Pone ^*  u , v\big\rangle_{\Hilbert^*\times \Hilbert}\big| 
= \big| \big\langle \Pi_1^*\iota  \newf + \Ptwo ^* u, v\big\rangle_{\Hilbertzero^*\times \Hilbertzero}\big|\\ \nonumber
&\qquad\leq\big| \big\langle   \newf , \iota\Pi_1 v\big\rangle_{\Hilbertzero\times \Hilbertzero^*}\big|
+ \big| \big\langle  u, \Ptwo v\big\rangle_{\Hilbertzero\times \Hilbertzero^*}\big|
\\ \nonumber
&\qquad\leq  \Big(\N{\newf}_{\cZ^{j-2}} \N{\iota \Pi_1 \iota^{-1}}_{(\cZ^{j-2})^*\to (\cZ^{j-2})^*} +
 \N{\Ptwo\iota^{-1}}_{(\cZ^{j-2})^*\to (\cZ^{j-2})^*}  \N{u}_{\cZ^{j-2}}
\Big)\N{ \iota v}_{(\cZ^{j-2})^*}.
\end{align}
Therefore, by \eqref{eq:sun1} and \eqref{eq:P2reg},
\begin{align*}
\big| \big\langle \Pone ^* \Pi_1 u ,\Pi_1 v\big\rangle_{\Hilbert^*\times \Hilbert}\big| \leq C \Big(
\N{\newf}_{\cZ^{j-2}} + \N{u}_{\cZ^{j-2}}
\Big) 
\N{\iota v}_{(\cZ^{j-2})^*},
\end{align*}
so that, by \eqref{eq:er1} with $\mathsf{D}= \Pone ^*$,
\begin{align}
\N{\Pi_1 u }_{\cZ^j} &\leq C \Big( \N{\Pi_1 u}_{\Hilbertzero } +
 \N{\newf}_{\cZ^{j-2}} + \N{u}_{\cZ^{j-2}}\Big)
 \leq C' \Big( 
 \N{\newf}_{\cZ^{j-2}} + \N{u}_{\cZ^{j-2}}\Big). 
\label{eq:birthday2}
\end{align}
Now, from the matrix form of $\operator$ \eqref{eq:matrices} and the equation $\operator^* u = \Pi_1^*\iota \newf$,
\beq\label{eq:birthday4}
\Ptwo_{00}^* \Pi_0 u + \Ptwo_{10}^* \Pi_1 u= 0 \quad\tin \Pi_0^*\Hilbert^*. 
\eeq
By the combination of \eqref{eq:tired1}, \eqref{eq:birthday4}, the definition $\Ptwo_{10}^* := (\Pi_1^* \Ptwo \Pi_0)^* = \Pi_0^* \Ptwo^*\Pi_1$ (from \eqref{eq:matrixnotation}), \eqref{eq:sun1}, and \eqref{eq:P2reg} with $\mathsf{E}= \iota^{-1}\Ptwo^*$, for $j=2,\ldots, \newell+1$, 
\begin{align}\nonumber
\N{\Pi_0 u}_{\cZ^j}& \leq C\N{ \iota^{-1}\Ptwo_{00}^* \Pi_0 u}_{\cZ^j}
\\ \nonumber
&
=C \N{ \iota^{-1}\Ptwo_{10}^* \Pi_1 u}_{\cZ^j}\\
&\leq C\big\| \iota^{-1} \Pi_0^* \Ptwo^* \Pi_1 u\big\|_{\cZ^j}
= C\big\| \iota^{-1} \Pi_0^* \iota (\iota^{-1}\Ptwo^*) \Pi_1 u\big\|_{\cZ^j} 
\leq C' \big\| \Pi_1 u\big\|_{\cZ^j}. 
\label{eq:birthday2a}
\end{align}
Combining this with \eqref{eq:birthday2} we obtain that 
\beq\label{eq:birthday3}
\N{u}_{\cZ^j}\leq 
 C \Big(  \N{\newf}_{\cZ^{j-2}} + \N{u}_{\cZ^{j-2}}\Big).
\eeq
When $j=2$, \eqref{eq:birthday3} implies that
\beq\label{eq:birthday3a}
\N{u}_{\cZ^2} \leq C\Big(1+\N{\Rs \Pi_1^*}_{\Hilbertzero^* \to \Hilbertzero }\Big)\N{\newf}_{\Hilbertzero }.
\eeq
The result then follows by the combination of \eqref{eq:birthday3}, \eqref{eq:birthday3a}, and induction.
\epf

We now also prove Lemma \ref{lem:oscil}, since its proof is similar to that of Lemma 
\ref{lem:regRs}.

\bpf[Proof of Lemma \ref{lem:oscil}]
From the matrix form of $\operator$ \eqref{eq:matrices}, $-\Ptwo_{00}\Pi_0 u =\Pi_0^*f$.
By \eqref{eq:tired1}, 
 \eqref{eq:sun1},  and the bound on $f$ in \eqref{eq:oscil},
\begin{align}\nonumber
\N{\Pi_0 u}_{\cZ^{\newell+1}} \leq C \big\|\iota^{-1}\Pi_0^* f\big\|_{\cZ^{\newell+1}}\leq C \big\|\iota^{-1} f\big\|_{\cZ^{\newell+1}}
&\leq  C \Cosc \N{f}_{\Hilbert^*}\\ \nonumber
&= C\Cosc \sup_{v\in \Hilbert}\frac{\big| \langle \operator u, v\rangle_{\Hilbert^*\times \Hilbert}\big|}{\N{v}_{\Hilbert}} \\ 
&\leq C \Cosc C'\N{u}_{\Hilbert}.\label{eq:coffee2}
\end{align}
We now argue as in the proof of Lemma \ref{lem:regRs} -- but now with $\operator^*$ replaced by $\operator$
and \eqref{eq:er1} applied with $\mathsf{D}= \Pone$ -- to obtain that \eqref{eq:birthday2} holds for $j=2,\ldots,\newell+1$ 
and $g= \iota^{-1}f$ (recall that in Lemma \ref{lem:regRs} we started with $g\in \cZ^{j-2}\subset \Hilbertzero$, and here we started with $f\in \Hilbert^*$).
When $j=2$, this bound implies that
\beq\label{eq:coffee3}
\N{\Pi_1 u}_{\cZ^2} \leq C \Big( \big\|\iota^{-1} f \big\|_{\Hilbertzero} + \N{u}_{\Hilbertzero}\Big) 
\leq C \Big( \big\|\iota^{-1} f \big\|_{\cZ^{\newell+1}} + \N{u}_{\Hilbert}\Big) 
\leq C' \N{u}_{\Hilbert},
\eeq
where in the last inequality we have argued as in \eqref{eq:coffee2}. 
The combination of \eqref{eq:coffee2} and \eqref{eq:coffee3} implies that $\| u\|_{\cZ^2} \leq C \|u \|_{\Hilbert}$. 
The result then follows from iterating the argument involving \eqref{eq:birthday2} for increasing $j$, up to $\newell+1$.

For the second assertion, now $-\Ptwo_{00}\Pi_0 u = \widetilde\Pi_0^*f$.
By \eqref{eq:tired1} and the assumption \eqref{eq:tired2}, 
\beq\label{eq:coffee4}
\N{\Pi_0 u}_{\cZ^{\newell+1}} \leq C \big\|\iota^{-1}\widetilde\Pi_0^* f\big\|_{\cZ^{\newell+1}}\leq C \N{ \iota^{-1} f}_{\cZ^{\newell-1}}.
\eeq
Therefore, since $\Pi_1 u$ gains two derivatives over $f$ via \eqref{eq:birthday2}, 
\begin{align*}
\N{u}_{\cZ^{\newell+1}} \leq \N{\Pi_0 u }_{\cZ^{\newell+1}} + \N{\Pi_1 u }_{\cZ^{\newell+1}} \leq C \Big(\big\|\iota^{-1} f\big\|_{\cZ^{\newell-1}} + \N{u}_{\cZ^{\newell-1}}\Big).
\end{align*}
Repeatedly applying \eqref{eq:birthday2} and using \eqref{eq:coffee4} (similar to in the first part of the proof) then gives that 
\begin{align*}
\N{u}_{\cZ^{\newell+1}} \leq C \Big(\big\|\iota^{-1} f\big\|_{\cZ^{\newell-1}} + \N{u}_{\Hilbert}\Big);
\end{align*}
the result then follows in an analogous way to \eqref{eq:coffee2}. 
\epf

\section{Bounding $\|\Pi_0(u-u_h)\|_{\Hilbert }$ using $\gamma_{\rm dv}(\operator)$}\label{sec:gammadv}

To prove the bound \eqref{eq:qo_abs} in Theorem \ref{thm:abs1}, we claim that it is sufficient to prove this bound under the assumption that $u_h$ exists.
To justify this claim, observe that 
since $u_h$ is the solution of a finite-dimensional linear system, 
the statements ``under the assumption that $u_h$ exists, $u_h$ is unique" 
and ``$u_h$ exists and is unique" are equivalent. 
Once the bound \eqref{eq:qo_abs} is established under the assumption that $u_h$ exists, setting $f=0$ and using that then (by one of the assumptions in Theorem \ref{thm:abs1}) $u=0$, we find that $u_h=0$; i.e., under the assumption that $u_h$ exists, $u_h$ is unique, and hence $u_h$ exists and is unique by the above equivalence.
From now on, therefore, we assume that $u_h$ exists and seek to prove \eqref{eq:qo_abs}.

The main result of this section is the following.

\ble\label{lem:Pi0error}
Given $\operator$ satisfying Assumption \ref{ass:1}, define $\Pi_0$ by \eqref{eq:Pi0}, $\gamma_{\rm dv}(\operator)$ by \eqref{eq:gamma_dv}. Given $u\in \Hilbert$, assume that the solution 
$u_h\in \Hilbert_h$ of \eqref{eq:Gog} exists.
Then 
\beq\label{eq:Pi0error}
\N{\Pi_0(u-u_h)}_{\Hilbert } \leq C \Big( 
\N{(I-\Pi_h)u }_{\Hilbert } + \gamma_{\rm dv}(\operator) \N{u-u_h}_{\Hilbert }
\Big).
\eeq
\ele

To prove Lemma \ref{lem:Pi0error}, we introduce the sesquilinear form
\beq\label{eq:b+}
b^+(u,v) := \langle \Pone  u,v\rangle_{\Hilbert ^*\times \Hilbert } + (C_{\Ptwo })^{-1} \langle\Ptwo  u,v\rangle_{\Hilbertzero^*\times\Hilbertzero}.
\eeq

\begin{lemma}\label{lem:b+}
$b^+$ is continuous and coercive on $\Hilbert $.
\end{lemma}

\bpf
Continuity is immediate. For coercivity, by \eqref{eq:P2} and \eqref{eq:Garding},
\beqs
\Re b^+(v,v) = \Re\langle \Pone v,v\rangle + (C_{\Ptwo })^{-1} \Re\langle \Ptwo  u,v\rangle_{\Hilbertzero^*\times\Hilbertzero}. = \N{v}^2_{\Hilbert } - \N{v}_{\Hilbertzero }^2 + \N{v}_{\Hilbertzero }^2 = \N{v}^2_{\Hilbert }.
\eeqs
\epf

\begin{corollary}[Definition and boundedness of $\Pi^+_h$]\label{cor:b+}
Given $u\in \Hilbert$, define $\Pi_h^+ u \in \Hilbert_h$ as the solution of 
\beqs
b^+\big( \Pi_h^+u, v_h\big) =  
b^+\big( u, v_h\big) 
\quad\tfa v_h\in \Hilbert _h;
\eeqs
i.e., 
\beq\label{eq:Gog+}
b^+\big( (I-\Pi_h^+)u, v_h\big) = 0 \quad\tfa v_h\in \Hilbert _h
\eeq
Then $\Pi_h^+ :\Hilbert \to \Hilbert _h$ is well-defined, bounded, satisfies $\Pi_h^+ w_h =w_h$ for all $w_h \in \Hilbert_h$, and 
satisfies
\beq\label{eq:Cea1}
\N{I-\Pi_h^+}_{\Hilbert \to \Hilbert } \leq C \N{I-\Pi_h}_{\Hilbert \to \Hilbert }.
\eeq
\end{corollary} 

\bpf
The fact that $\Pi_h^+$ is well defined and bounded follows from Lemma \ref{lem:b+} combined with the Lax--Milgram lemma \cite{LaMi:54}, \cite[Lemma 2.32]{Mc:00}, and the bound \eqref{eq:Cea1} then follows from 
C\'ea's lemma \cite{Ce:64}, \cite[Theorem 13.1]{Ci:91}. The fact that $\Pi_h^+w_h=w_h$ for all $w_h\in\Hilbert_h$ follows from the facts that (i) $\Pi_h^+$ is well-defined, and (ii) if $w_h\in \Hilbert_h$, then 
$\Pi_h^+w_h=w_h$ is a solution of \eqref{eq:Gog+}.
\epf

By the definitions of $b^+$ \eqref{eq:b+} and $\Pi_h^+$ \eqref{eq:Gog+},
\beq\label{eq:Pi+P2}
\big\langle\Ptwo  (I-\Pi_h^+) u, v_h\big\rangle_{\Hilbertzero^*\times\Hilbertzero } =0 \quad\tfa v_h\in \Hilbert _h \cap \Ker \Pone.
\eeq
In the terminology of Remark \ref{rem:relation2}, 
a consequence of \eqref{eq:Pi+P2} is that if $u$ is discretely $\epsilon$-divergence free, then so is $\Pi_h^+ u$.

\bre[Link to the notation of \cite{ChEr:23} and \cite{BrPa:08}]
In the Maxwell setting of Lemma \ref{lem:Maxwell}, and when $\epsilon$ and $\mu$ are real, $\Pi_h^+$ is denoted by $\mathcal{B}^{\rm c}_{h0}$ in \cite[\S4.1]{ChEr:23};
the property \eqref{eq:Pi+P2} is then \cite[Equation 4.3b]{ChEr:23}. 
The operator defined by \eqref{eq:Gog} with the arguments of $b^+(\cdot,\cdot)$ swapped is denoted by $\hat{P}_h$ in \cite[Equation 3.12]{BrPa:08}.
\ere

\begin{lemma}
\label{lem:gamma_dv}
Given $\operator$ satisfying Assumption \ref{ass:1}, define $\Pi_0$ by \eqref{eq:Pi0}, $\gamma_{\rm dv}(\operator)$ by \eqref{eq:gamma_dv}, and $\Pi_h^+$ by \eqref{eq:Gog+}. 
If $w$ satisfies 
\beqs
 \big\langle\Ptwo  w, v_h \big\rangle_{\Hilbertzero^*\times\Hilbertzero}=0 \quad\tfa v_h \in \Hilbert _h \cap \Ker \Pone,
 \eeqs
then
\beqs
\N{\Pi_0 \Pi_h^+ w}_{\Hilbertzero } \leq C \gamma_{\rm dv}(\operator)\N{w}_{\Hilbert }.
\eeqs
\end{lemma}

\bpf
By \eqref{eq:Pi+P2},
\beqs
\big\langle \Ptwo  \Pi^+_h w, v_h \big\rangle_{\Hilbertzero^*\times\Hilbertzero}=0 \quad\tfa v_h \in \Hilbert _h \cap \Ker \Pone .
\eeqs
Therefore, 
by the definition of $\gamma_{\rm dv}$ \eqref{eq:gamma_dv} and Corollary \ref{cor:b+},
\beqs
\N{\Pi_0 \Pi_h^+ w}_{\Hilbertzero } 
 \leq\gamma_{\rm dv}(\operator)\N{\Pi_h^+ w}_{\Hilbert }\leq C \gamma_{\rm dv}(\operator)\N{w}_{\Hilbert }.
\eeqs
\epf

\bpf[Proof of Lemma \ref{lem:Pi0error}]
Since $\Pi_h^+ u_h=u_h$ (by Corollary \ref{cor:b+}), $\Pi_0:\Hilbertzero \to \Hilbertzero $ is bounded, and $I-\Pi_h^+$ satisfies \eqref{eq:Cea1},
\begin{align*}
\N{\Pi_0(u-u_h)}^2_{\Hilbertzero }
&= \Big( \Pi_0(u-u_h) , \Pi_0\Big( (I-\Pi_h^+)u + \Pi_h^+ (u-u_h)\Big)\Big)_{\Hilbertzero }\\
&\leq \N{\Pi_0(u-u_h)}_{\Hilbertzero } \Big( C \N{(I-\Pi_h)u}_{\Hilbert} + \N{\Pi_0 \Pi_h^+ (u-u_h) }_{\Hilbertzero }\Big).
\end{align*}
By \eqref{eq:Gog2} and Lemma \ref{lem:gamma_dv},
\beqs
\N{\Pi_0 \Pi_h^+ (u-u_h)}_{\Hilbertzero } \leq C \gamma_{\rm dv}(\operator) \N{u-u_h}_{\Hilbert },
\eeqs
and the result then follows since $\N{\Pi_0(u-u_h)}_{\Hilbert }=\N{\Pi_0(u-u_h)}_{\Hilbertzero }$ by \eqref{eq:simple_norm2}.
\epf

\section{Asymptotic quasi-optimality}

As mentioned in \S\ref{sec:context}, the following result, Lemma \ref{lem:basic}, is morally equivalent to that obtained by the classic duality argument introduced in \cite{Mo:92,Mo:03a} (see also \cite[Theorem 4.6]{ChEr:23} for a recent variant of this result, which uses notation similar to that below). 
The main abstract result of this paper (in the form of Lemma \ref{thm:abs2} below) provides a stronger result than Lemma \ref{lem:basic}, 
but the proof of Lemma \ref{thm:abs2} uses Lemma \ref{lem:basic} applied to an auxiliary operator, $\Phash$; see Lemma \ref{lem:Pihash} below.

\ble[Asymptotic quasi-optimality]\label{lem:basic}
If $\operator$ satisfies Assumption \ref{ass:1} and $(P^*)^{-1}$ exists, 
then there exist $C_1,C_2,C_3>0$ such that if 
\beq\label{eq:asymptotic1}
\gamma_{\rm dv}(\operator) \leq C_1 \quad\tand\quad 
\N{(I-\Pi_h)\Rs\Pi_1^*}_{\Hilbertzero^* \to \Hilbert }\leq C_2 ,
\eeq
then $u_h$ exists, is unique, and satisfies
\beqs
\N{u-u_h}_{\Hilbert } \leq 
C_3
\N{(I-\Pi_h)u}_{\Hilbert }.
\eeqs
\ele

 Lemma \ref{lem:basic} combined with Lemma \ref{lem:regRs} with $\newell=1$ gives the following corollary.  

\begin{corollary}[Asymptotic quasi-optimality under low regularity]\label{cor:basic}
If $\operator$ satisfies Assumptions \ref{ass:1} and \ref{ass:2}, the latter with $\newell=1$,  and $(P^*)^{-1}$ exists, then there exist $C_1,C_2,C_3>0$ such that if 
\beqs
\gamma_{\rm dv}(\operator) \leq C_1 \quad\tand\quad 
\N{I-\Pi_h}_{\cZ^{2}\to \Hilbert } \big( 1 + \N{\Rs\Pi_1^*}_{\Hilbertzero^* \to \Hilbertzero }\big)\leq C_2 ,
\eeqs
then $u_h$ exists, is unique, and satisfies
\beqs
\N{u-u_h}_{\Hilbert } \leq 
C_3
\N{(I-\Pi_h)u}_{\Hilbert }.
\eeqs
\end{corollary}

Lemma \ref{lem:basic} is an immediate consequence of the following two results. 

\ble[Quasi-optimality of the Galerkin solution, modulo $\|\Pi_1(u-u_h)\|_{\Hilbert}$]\label{lem:kernel_basic}
Suppose that $\operator$ satisfies Assumption \ref{ass:1}. Given $u\in \Hilbert$, assume that the solution 
$u_h\in \Hilbert_h$ of \eqref{eq:Gog} exists. Then  there exists $C_1, C_2>0$ such that 
\beq\label{eq:Garding_error_basic}
\big(1 - C_1 \gamma_{\rm dv}(\operator)\big) \N{ u-u_h}_{\Hilbert } \leq C_2 \Big(\N{(I-\Pi_h)u}_{\Hilbert }+ \N{\Pi_1(u-u_h)}_{\Hilbertzero }\Big).
\eeq
\ele

\ble[Aubin-Nitsche-type argument analogous to \eqref{eq:AN}] \label{lem:dualityL2_basic}
Suppose that $\operator$ satisfies Assumption \ref{ass:1}. 
Given $u\in \Hilbert$, assume that the solution 
$u_h\in \Hilbert_h$ of \eqref{eq:Gog} exists. Then  there exists $C>0$ such that 
\begin{align*}
\N{\Pi_1(u-u_h)}_{\Hilbertzero }
\leq 
C\big\|(I-\Pi_h) \Rs \Pi_1^*\|_{\Hilbertzero^* \to \Hilbert } 
\N{u-u_h}_{\Hilbert }.
\end{align*}
\ele

The proof of Lemma \ref{lem:dualityL2_basic} is short, and so we give it first.

\bpf[Proof of Lemma \ref{lem:dualityL2_basic}]
By the definition of $\iota$ \eqref{eq:Riesz}, the definition of $\Rs:\Hilbert^*\to\Hilbert$, Galerkin orthogonality \eqref{eq:Gog}, and boundedness of $\operator:\Hilbert\to \Hilbert^*$, 
\begin{align*}
\N{\Pi_1(u-u_h)}^2_{\Hilbertzero } &= 
\big\langle \Pi_1(u-u_h), \iota \Pi_1 (u-u_h)\big\rangle_{\Hilbertzero\times \Hilbertzero^*},\\
&=\big\langle u-u_h, \Pi_1^* \iota\Pi_1 (u-u_h)\big\rangle_{\Hilbertzero\times \Hilbertzero^*},\\
&= \big\langle \operator(u-u_h) , \Rs \Pi_1^* \iota\Pi_1 (u-u_h)\big\rangle_{\Hilbert^*\times\Hilbert},\\
&= \big\langle \operator(u-u_h) , (I-\Pi_h)\Rs \Pi_1^* \iota\Pi_1 (u-u_h)\big\rangle_{\Hilbert^*\times\Hilbert},\\
&\leq  C
 \N{u-u_h}_\Hilbert  \big\|(I-\Pi_h)\Rs \Pi_1^*\big\|_{\Hilbertzero^* \to \Hilbert }\N{ \Pi_1 (u-u_h)}_{\Hilbertzero },
\end{align*}
and the result follows.
\epf

\bpf[Proof of Lemma \ref{lem:kernel_basic}]
By the triangle inequality, \eqref{eq:simple_norm2}, and \eqref{eq:Pi0error}, 
\begin{align*}
\N{u-u_h}_\Hilbert  &\leq  \N{\Pi_0(u-u_h)}_{\Hilbertzero }+\N{\Pi_1(u-u_h)}_\Hilbert \\
&\leq  C \Big( \N{(I-\Pi_h)u}_\Hilbert  + \gamma_{\rm dv}(\operator) \N{u-u_h}_{\Hilbert }\Big)+\N{\Pi_1(u-u_h)}_\Hilbert ;
\end{align*}
i.e., 
\beq\label{eq:coffee5}
\big(1 - C \gamma_{\rm dv}(\operator)\big) \N{ u-u_h}_{\Hilbert } \leq  C \N{(I-\Pi_h)u}_\Hilbert +\N{\Pi_1(u-u_h)}_\Hilbert .
\eeq
We claim that it is now sufficient to prove that, for all $\varepsilon>0$, 
\beq\label{eq:STP1_basic}
\N{\Pi_1(u-u_h)}_\Hilbert \leq \varepsilon \N{u-u_h}_\Hilbert  +C \varepsilon^{-1} \Big(
\N{(I-\Pi_h)u}_\Hilbert + \N{\Pi_1(u-u_h)}_{\Hilbertzero }
+\N{\Pi_0(u-u_h)}_{\Hilbert }
\Big).
\eeq
Indeed, inputting \eqref{eq:STP1_basic} into \eqref{eq:coffee5} and using again \eqref{eq:Pi0error}, we find \eqref{eq:Garding_error_basic}.

We now prove \eqref{eq:STP1_basic}. 
By the G\aa rding inequality \eqref{eq:Garding2},
\beq\label{eq:Friday2_basic}
\N{\Pi_1(u-u_h)}^2_\Hilbert  \leq \Re
 \big\langle \operator \Pi_1 (u-u_h), \Pi_1 (u-u_h)
\big\rangle_{\Hilbert^*\times \Hilbert}
+ (1 + \N{\Ptwo}_{\Hilbertzero\to\Hilbertzero^*})\N{\Pi_1(u-u_h)}^2_{\Hilbertzero}.
\eeq
Now, since $\Pi_0:\Hilbert\to \Ker \Pone$ 
and $\Ker \Pone  = \Ker \Pone^*$, for all $v\in\Hilbert$,
\begin{align*}
&\Re\big\langle \operator\Pi_1 v, \Pi_1 v\big\rangle_{\Hilbert^*\times \Hilbert} 
\\
&\qquad= \langle \operator v,v\rangle_{\Hilbert^*\times \Hilbert} - \langle \operator\Pi_0v,v\rangle_{\Hilbert^*\times \Hilbert} - \langle \operator v, \Pi_0 v\rangle_{\Hilbert^*\times \Hilbert} + \langle \operator\Pi_0 v,\Pi_0v\rangle_{\Hilbert^*\times \Hilbert}\\
&\qquad=\langle \operator v,v\rangle_{\Hilbert^*\times \Hilbert} + \langle \Ptwo \Pi_0v,v\rangle_{\Hilbertzero^*\times \Hilbertzero} + \langle \Ptwo v, \Pi_0 v\rangle_{\Hilbertzero^*\times \Hilbertzero} - \langle \Ptwo \Pi_0 v,\Pi_0v\rangle_{\Hilbertzero^*\times \Hilbertzero}.
\end{align*}
Therefore, by the boundedness of $\Ptwo:\Hilbertzero\to\Hilbertzero^*$ and
the inequality
\beq\label{eq:PeterPaul}
2ab\leq \varepsilon a^2 + \varepsilon^{-1}b^2, \quad\tfor a,b,\varepsilon>0,
\eeq
\beqs
\Re\big\langle \operator\Pi_1 v, \Pi_1 v
\big\rangle_{\Hilbert^*\times \Hilbert} 
\leq\Re \big\langle \operator v,v\big\rangle_{\Hilbert^*\times \Hilbert} + C\big( \varepsilon^{-1} \N{\Pi_0 v}^2_{\Hilbertzero } + \varepsilon \N{v}^2_{\Hilbertzero }\big).
\eeqs
Applying this last inequality with $v=u-u_h$, combining with \eqref{eq:Friday2_basic}, and then using Galerkin orthogonality \eqref{eq:Gog}, we find that 
\begin{align*}
\N{\Pi_1(u-u_h)}^2_{\Hilbert } &\leq \Re \big\langle 
\operator(u-u_h), (I-\Pi_h) u 
\big\rangle \\
&\quad + C \Big( \varepsilon^{-1} \N{\Pi_0 (u-u_h)}^2_{\Hilbertzero } + \varepsilon \N{u-u_h}^2_{\Hilbertzero } + 
\N{\Pi_1(u-u_h)}_{\Hilbertzero }^2
\Big).
\end{align*}
Therefore, by 
the boundedness of $\operator:\Hilbert \to \Hilbert ^*$ and 
\eqref{eq:PeterPaul} (with $a= \|u-u_h\|_{\cH}$ and $b = (1/2)\| P\|_{\cH\to \cH^*} \| (I- \Pi_h) u\|_{\cH}$),
\begin{align}\nonumber
\N{\Pi_1(u-u_h)}^2_{\Hilbert } &\leq 
\varepsilon \N{u-u_h}_{\Hilbert }^2
\\
&\quad+ C \Big(
\varepsilon^{-1} \N{(I-\Pi_h)u}^2_{\Hilbert } 
+ \varepsilon^{-1} \N{\Pi_0 (u-u_h)}^2_{\Hilbertzero }+ 
\N{\Pi_1(u-u_h)}_{\Hilbertzero }^2
\Big).
\label{eq:plant1}
\end{align}
By \eqref{eq:simple_norm2}, this last inequality implies \eqref{eq:STP1_basic} and the proof is complete.
\epf

\section{Definition of the operator $\Phash$ and associated results}\label{sec:Phash}

\subsection{Identification of $\Pi_1 \Hilbertzero$ with $\Pi_1^*\Hilbertzero^*$ and $(\Pi_1 \Hilbertzero)^*$}\label{sec:identify}

Since $\Pi_1\Hilbertzero $ is the kernel of the bounded operator $\Pi_0 :\Hilbertzero \to\Hilbertzero $, $\Pi_1\Hilbertzero $ is closed in $\Hilbertzero $, and thus $\Pi_1\Hilbertzero$ is a Hilbert space. We define
\beq\label{eq:ipPi1Hilbert}
(u,v)_{\Pi_1\Hilbertzero}:= (\Pi_1 u, \Pi_1 v)_{\Hilbertzero} \quad\tfor u,v \in \Pi_1 \Hilbertzero.
\eeq

We now define the maps identifying $\Pi_1 \Hilbertzero$ with $\Pi_1^*\Hilbertzero^*$ and $(\Pi_1 \Hilbertzero)^*$ and then prove that these maps are bijective (see Corollary \ref{cor:bijective} below). In particular, the rest of \S\ref{sec:Phash} crucially uses the fact that 
 the identification of $\Pi_1 \Hilbertzero$ with $\Pi_1^*\Hilbertzero^*$, denoted by $\eta$, is invertible.

\paragraph{Identification of $\Pi_1 \Hilbertzero$ with $\Pi_1^*\Hilbertzero^*$.}

Let $\eta:\Pi_1 \Hilbertzero \to \Pi_1^* \Hilbertzero^*$ be the identification of $\Pi_1 \Hilbertzero$ with $\Pi_1^*\Hilbertzero^*$ defined by
\beq\label{eq:eta}
\langle \eta u, v \rangle_{\Pi_1^* \Hilbertzero^* \times \Pi_1 \Hilbertzero} = (u,v)_{\Pi_1\Hilbertzero} \quad\tfor u,v \in \Pi_1 \Hilbertzero.
\eeq

\ble[Properties of $\eta$]\label{lem:eta}

\

(i) $\eta: \Pi_1 \Hilbertzero\to \Pi_1^*\Hilbertzero^*$ is injective.

(ii) $\eta = \Pi_1^* \iota\Pi_1$, and this formula extends $\eta$ to a map $\Hilbertzero\to \Hilbertzero^*$. 
\ele

\bpf
(i) Suppose $\eta u =0$ for some $u\in \Pi_1 \Hilbertzero$. Then, by \eqref{eq:eta}, $(u,v)_{\Pi_1\Hilbertzero}=0$ for all $v\in \Pi_1\Hilbertzero$, so that $u=0$.

(ii) By \eqref{eq:eta}, \eqref{eq:ipPi1Hilbert}, and the definition of the Riesz map $\iota :\Hilbertzero \to \Hilbertzero^*$ \eqref{eq:Riesz}, for $u,v\in \Pi_1 \Hilbertzero$, 
\begin{align}\nonumber
\langle \eta u, v \rangle_{\Pi_1^* \Hilbertzero^* \times \Pi_1 \Hilbertzero} =  (\Pi_1 u, \Pi_1 v)_{\Hilbertzero}
= \big\langle \iota \Pi_1 u, \Pi_1 v\big\rangle_{\Hilbertzero^*\times \Hilbertzero} 
&= \big\langle \Pi_1^*\iota \Pi_1 u, v\big\rangle_{\Hilbertzero^*\times \Hilbertzero}\\
&= \big\langle \Pi_1^*\iota \Pi_1 u, v\big\rangle_{\Pi_1^*\Hilbertzero^*\times \Pi_1\Hilbertzero}
\label{eq:eta_calc1}
\end{align}
(where in the last step we treat $\Pi_1 \Hilbertzero$ and $\Pi_1^*\Hilbertzero^*$ as subsets of $\Hilbertzero$ and $\Hilbertzero^*$, respectively).
Therefore \eqref{eq:eta_calc1} shows that 
$\eta = \Pi_1^* \iota\Pi_1$ as a map $\Pi_1 \Hilbertzero\to \Pi_1^*\Hilbertzero^*$. 
Since $\eta=\Pi^*_1 \eta \Pi_1$, 
the formula $\eta = \Pi_1^* \iota\Pi_1$ extends $\eta$ to a map $\Hilbertzero\to \Hilbertzero^*$.
\epf

\paragraph{Identification of $\Pi_1 \Hilbertzero$ with $(\Pi_1\Hilbertzero)^*$.}

Let $\widetilde{\eta} :\Pi_1 \Hilbertzero \to (\Pi_1\Hilbertzero)^*$ be defined by
\beq\label{eq:widetildeeta}
\langle \widetilde\eta u, v \rangle_{(\Pi_1 \Hilbertzero)^* \times \Pi_1 \Hilbertzero} = (u,v)_{\Pi_1\Hilbertzero} \quad\tfor u,v \in \Pi_1 \Hilbertzero
\eeq
(compare to \eqref{eq:eta}). 
By the Riesz representation theorem, $\widetilde \eta$ is bijective $\Pi_1 \Hilbertzero \to (\Pi_1\Hilbertzero)^*$.

\paragraph{Identification of $\Pi_1^* \Hilbertzero^*$ with $(\Pi_1\Hilbertzero)^*$.}

Let $\rho :\Pi_1^* \Hilbertzero^* \to (\Pi_1\Hilbertzero)^*$ be defined by:~given $\phi \in \Pi_1^*\Hilbertzero^*$ (so that $\phi= \Pi_1^*\phi$), 
\beq\label{eq:rho}
\langle \rho \phi , v \rangle_{(\Pi_1 \Hilbertzero)^* \times \Pi_1 \Hilbertzero} = \langle \Pi_1^* \phi, v\rangle_{\Hilbertzero^*\times\Hilbertzero}
\quad\tfa v \in \Pi_1 \Hilbertzero.
\eeq

\ble
$\rho :\Pi_1^* \Hilbertzero^* \to (\Pi_1\Hilbertzero)^*$  is injective. 
\ele

\bpf
If $\rho \phi=0$, where $\phi = \Pi^*_1\phi$, then, by definition, $\langle \Pi_1^* \phi, v\rangle_{\Hilbertzero^*\times\Hilbertzero}=0$ for all $v\in \Pi_1 \Hilbertzero$. Since $\Pi_0 \Pi_1=0$, this last equality holds in fact for all $v\in \Hilbertzero$, so that $\phi = \Pi^*_1 \phi=0$ as an element of $\Hilbertzero^*$, and hence also as an element of $\Pi_1^* \Hilbertzero^*$.
\epf 

\ble\label{lem:comp}
$\rho \eta = \widetilde\eta$ as maps $\Pi_1 \Hilbertzero\to (\Pi_1 \Hilbertzero)^*$.
\ele

\begin{corollary}\label{cor:bijective}
$\eta: \Pi_1 \Hilbertzero\to \Pi_1^*\Hilbertzero^*$, $\widetilde{\eta} :\Pi_1 \Hilbertzero \to (\Pi_1\Hilbertzero)^*$, and 
$\rho :\Pi_1^* \Hilbertzero^* \to (\Pi_1\Hilbertzero)^*$ are all bijective. 
\end{corollary}

\bpf[Proof of Corollary \ref{cor:bijective}]
The bijectivity of $\widetilde{\eta} :\Pi_1 \Hilbertzero \to (\Pi_1\Hilbertzero)^*$ is a consequence of the Riesz representation theorem (as noted above). 
Since $\rho$ and $\eta$ are both injective and $\widetilde\eta$ is bijective, Lemma \ref{lem:comp} implies that $\rho$ and $\eta$ are bijective. 
\epf

\bpf[Proof of Lemma \ref{lem:comp}]
By the definition of $\rho$ \eqref{eq:rho}, Part (ii) of Lemma \ref{lem:eta}, the definition of $\iota$ \eqref{eq:Riesz}, \eqref{eq:ipPi1Hilbert}, and \eqref{eq:widetildeeta},
 for all $u,v\in \Pi_1 \Hilbertzero$, 
\begin{align*}
\langle \rho \eta u, v \rangle_{(\Pi_1 \Hilbertzero)^* \times \Pi_1 \Hilbertzero} = \langle \Pi_1^*\Pi_1^* \iota \Pi_1u, v\rangle_{\Hilbertzero^*\times\Hilbertzero}
&=\langle  \iota \Pi_1u, \Pi_1 v\rangle_{\Hilbertzero^*\times\Hilbertzero}\\
&= (\Pi_1 u,\Pi_1 v)_{\Hilbertzero}\\
&= (u,v)_{\Pi_1 \Hilbertzero} = 
\langle \widetilde\eta u, v \rangle_{(\Pi_1 \Hilbertzero)^* \times \Pi_1 \Hilbertzero}.
\end{align*}
\epf

Having proved that $\eta^{-1}$ exists, we now prove that $\eta^{-1}$ and $\eta$ are both self adjoint.

\ble[$\eta$ and $\eta^{-1}$ are both self-adjoint]\label{lem:eta_inverse}
With $\eta:\Pi_1 \Hilbertzero \to \Pi_1^* \Hilbertzero^*$ defined by \eqref{eq:eta},
\beqs
\langle \phi, v\rangle_{\Pi_1^* \Hilbertzero^*\times \Pi_1 \Hilbertzero} = \langle \eta^{-1}\phi, \eta v\rangle_{\Pi_1 \Hilbertzero\times \Pi_1^* \Hilbertzero^*} \quad\tfa \phi \in \Pi_1^*\Hilbertzero^* \tand v \in \Pi_1 \Hilbertzero.
\eeqs
\ele

\bpf
Let $u:= \eta^{-1}\phi\in \Pi_1 \Hilbertzero$ (which exists by Corollary \ref{cor:bijective}).
Then, by two applications of \eqref{eq:eta}, 
\begin{align*}
\langle \phi, v\rangle_{\Pi_1^* \Hilbertzero^*\times \Pi_1 \Hilbertzero} = 
\langle \eta u, v\rangle_{\Pi_1^* \Hilbertzero^*\times \Pi_1 \Hilbertzero} = 
(u,v)_{\Pi_1 \Hilbertzero}
= 
\langle  u, \eta v\rangle_{\Pi_1 \Hilbertzero\times \Pi_1^* \Hilbertzero^*},
\end{align*} 
and the result follows.
\epf

\subsection{Definition of $\Phash$ and $\Rhash$}

Recalling the matrix form of $\operator$ \eqref{eq:matrices}, 
we define $\cP: 
 \Pi_1\Hilbert \to \Pi_1^* \Hilbert ^*$ by
\beq\label{eq:cP}
\mathcal{P} := \Re \big( \Pone_{11} - \Ptwo_{11}\big).
\eeq
By definition, if $v\in \Pi_1\Hilbert$ then $\langle 
\cP
v,v\rangle_{\Pi_1^*\Hilbert ^*\times \Pi_1\Hilbert } 
=\Re\langle 
\operator
 v, v\rangle_{\Hilbert ^*\times \Hilbert }$. Therefore, 
by \eqref{eq:Garding2}, 
\beq\label{eq:GardingP111}
\big\langle 
\cP
v,v\big\rangle_{\Pi_1^*\Hilbert ^*\times \Pi_1\Hilbert } 
\geq \N{ v}^2_{\Hilbert } - (1+ \|\Ptwo\|_{\Hilbertzero\to\Hilbertzero^*})\N{ v}^2_{\Hilbertzero }\quad\tfa v\in \Pi_1 \Hilbert .
\eeq

\begin{theorem}[Friedrichs extension theorem]\label{thm:Friedrichs}
Suppose that $\Hilbertzeroalt$ is a  Hilbert space and $\Hilbertalt$ is dense in $\Hilbertzeroalt$. 
Suppose that $Q:\Hilbertalt\times\Hilbertalt\to \mathbb{C}$ is a sesquilinear form such that 
(i) $Q(u,v) = \overline{Q(v,u)}$ for all $u,v\in \Hilbertalt$,
(ii) there exists $C>0$ such that 
\beqs
Q(v,v) \geq 
-C \N{v}^2_{\Hilbertzeroalt}  \quad\tfa v\in \Hilbertalt,
\eeqs
and (iii) $\Hilbertalt$ is complete under the norm
\beqs
\vertiii{v}:= \sqrt{ Q(v,v) + (1+C) \N{v}_{\Hilbertzeroalt}^2}.
\eeqs

Then there exists a densely-defined, self-adjoint operator $\Palt:\Hilbertzeroalt\to \Hilbertzeroalt^*$  such that 
\beqs
Q(u,v) = \langle \Palt u,v\rangle_{\Hilbertzeroalt^*\times \Hilbertzeroalt} \quad \tfa u \in {\rm Dom}(\Palt) \tand v \in \Hilbertzeroalt,
\eeqs
where the domain of $\Palt$, ${\rm Dom}(\Palt)$, is defined by 
\beqs
{\rm Dom}(\Palt) : = \Big\{ u \in\Hilbertalt\,:\, \text{there exists $C_u>0$ such that }\,|Q(u,v)| \leq C_u \N{v}_{\Hilbertzeroalt} \, \tfa v\in \Hilbertalt\Big\}.
\eeqs
\end{theorem}

\bpf[References for the proof]
See, e.g., 
\cite[Theorem VIII.15, Page 278]{ReSi:80},  \cite[Theorem 12.24, Page 360]{Gr:08} (with  \cite{Fr:34} the original paper). 
\epf

\begin{corollary}\label{cor:Friedrichs}
$\cP$ defined by \eqref{eq:cP} extends to a densely-defined, self-adjoint operator
$\Pi_1\Hilbertzero \to \Pi_1^* \Hilbertzero^*$ 
with 
\beq\label{eq:Qsesqui}
\big\langle \cP u,v \big\rangle_{\Pi_1^* \Hilbert^*\times \Pi_1\Hilbert}= \big\langle \cP u,v\big\rangle_{\Pi_1^*\Hilbertzero^*\times \Pi_1^*\Hilbertzero} \quad \tfa u \in {\rm Dom}(\cP)\subset \Pi_1\Hilbert \tand v \in \Pi_1 \Hilbertzero.
\eeq
Furthermore $\eta^{-1} \cP$ is a densely-defined, self-adjoint operator $\Pi_1 \Hilbertzero \to \Pi_1 \Hilbertzero$, 
with its spectrum  
bounded below by $- \|\Ptwo\|_{\Hilbertzero\to\Hilbertzero^*}$.
\end{corollary} 

\bre
We work with the Friedrichs extension of $\cP$, since the spectral theorem (used below to construct the elliptic projection operator $P^\sharp$ \eqref{eq:Phash}) is most-naturally formulated for a, possibly unbounded, self-adjoint operator from a Hilbert space to itself.
\ere

\bpf[Proof of Corollary \ref{cor:Friedrichs}]
We apply Theorem \ref{thm:Friedrichs} with $\Hilbertalt= \Pi_1 \Hilbert$, $\Hilbertzeroalt=\Pi_1 \Hilbertzero$, and $Q(u,v)= \langle \cP u,v\rangle_{\Hilbert^*\times\Hilbert}$. We also identify  
$(\Pi_1\Hilbert)^*$ and $\Pi_1^* \Hilbert ^*$ (so that $\cP: \Pi_1 \Hilbert \to (\Pi_1 \Hilbert)^*$, i.e., $\cP: H\to H^*$); this identification is analogous to the identification between 
$(\Pi_1 \Hilbertzero)^*$ and $\Pi_1^*\Hilbertzero^*$ described in \S\ref{sec:identify}. However, 
 we do not introduce any notation for it 
since this identification is only used inside this proof and inside the proof of Lemma \ref{lem:S1} (when using the the Lax--Milgram lemma).
The proof of Lemma \ref{lem:Phash_inverse} uses the Lax--Milgram lemma with the analogous identification of $\Pi_0^*\Hilbert^*$ and $(\Pi_0 \Hilbert)^*$. 

We now check the assumptions of Theorem \ref{thm:Friedrichs}. Since $\Hilbert$ is dense in $\Hilbertzero$, and $\Pi_1:\Hilbertzero\to\Hilbertzero$ is bounded, $\Hilbertalt=\Pi_1\Hilbert$ is dense in $\Hilbertzeroalt = \Pi_1\Hilbertzero$. 
By its definition \eqref{eq:cP}, $\cP: \Pi_1\Hilbert\to \Pi_1^*\Hilbert^*$ is self adjoint; i.e., Assumption (i) of Theorem \ref{thm:Friedrichs} is satisifed. The G\aa rding inequality \eqref{eq:GardingP111} then implies that Assumptions (ii) and (iii) of Theorem \ref{thm:Friedrichs} are satisfied with $C= C_{\Ptwo}$.

We denote the extension $\Palt$ given by Theorem \ref{thm:Friedrichs} also by $\cP$, so that we extend $\cP$ to 
a densely-defined, self-adjoint operator
$\Pi_1\Hilbertzero \to \Pi_1^* \Hilbertzero^*$.

By \eqref{eq:eta}, the self-adjointness of $\cP: \Pi_1\Hilbertzero \to \Pi_1^* \Hilbertzero^*$, and \eqref{eq:eta} again, for all $u\in  {\rm Dom}(\eta^{-1}\cP)$, 
\begin{align*}
\big( \eta^{-1} \cP u,v\big)_{\Pi_1 \Hilbertzero} = 
\big\langle \cP u,v\big\rangle_{\Pi_1^*\Hilbertzero^*\times\Pi_1 \Hilbertzero} = 
\big\langle u,\cP v\big\rangle_{\Pi_1\Hilbertzero\times\Pi_1^* \Hilbertzero^*} = 
\big( u, \eta^{-1} \cP v\big)_{\Pi_1 \Hilbertzero};
\end{align*}
thus $\eta^{-1} \cP: \Pi_1 \Hilbertzero \to \Pi_1 \Hilbertzero$ is a densely-defined  self-adjoint operator. 
Finally, by \eqref{eq:GardingP111}, for all $v\in {\rm Dom}(\eta^{-1}\cP)\subset \Pi_1\Hilbertzero$, 
\beqs
\big(\eta^{-1} \cP v,v\big)_{\Pi_1 \Hilbertzero} \geq \N{v}^2_\Hilbert - (1+ \|\Ptwo\|_{\Hilbertzero\to\Hilbertzero^*})\N{ v}^2_{\Hilbertzero }
\geq  - \|\Ptwo\|_{\Hilbertzero\to\Hilbertzero^*}\N{ v}^2_{\Hilbertzero}.
\eeqs
For all $\varepsilon>0$, $\eta^{-1}\cP +  \|\Ptwo\|_{\Hilbertzero\to\Hilbertzero^*} + \varepsilon
:{\rm Dom}(\eta^{-1}\cP)\to\Pi_1\Hilbertzero$ is then invertible by a variant of the Lax--Milgram lemma for densely-defined operators; see, e.g., \cite[Theorem 12.18]{Gr:08} or the proof of \cite[Theorem VIII.15]{ReSi:80}. 
Thus the spectrum of $\eta^{-1}\cP$ (i.e., the set of $\lambda$ such that $\eta^{-1}\cP - \lambda:{\rm Dom}(\eta^{-1}\cP)\to\Pi_1\Hilbertzero$ is not invertible) is bounded below by $- \|\Ptwo\|_{\Hilbertzero\to\Hilbertzero^*}$. 
\epf

We now use the functional calculus for 
$\eta^{-1}\cP:\Pi_1 \Hilbertzero \to \Pi_1 \Hilbertzero$ to define
\beq\label{eq:S}
\smoother:= \psi\big(\eta^{-1}\cP\big),
\eeq 
where $\psi\in C^{\infty}_{\rm comp}(\Rea; [0,\infty))$ is such that 
\beq\label{eq:psi_inequality}
x+ \psi^2(x)\geq 1 \quad \tfor \quad x\geq - \|\Ptwo\|_{\Hilbertzero\to\Hilbertzero^*}.
\eeq

We recap the following results about the functional calculus.

\begin{theorem}[Functional-calculus results]\label{thm:func_calc}
Let $\cL$ be a densely-defined, self-adjoint operator on a Hilbert space $\Hilbertzeroalt$, and let $\sigma(\cL)$ denote its spectrum.

(i) If $\psi \in L^\infty(\Rea;\Rea)$ then $\psi(\cL):\Hilbertzeroalt\to\Hilbertzeroalt$ is self-adjoint, in the sense that $(\psi(\cL)u,v)_{\Hilbertzeroalt}=(u,\psi(L)v)_{\Hilbertzeroalt}$ for all $u,v\in\Hilbertzeroalt$.

(ii) If $\psi \in L^\infty(\Rea;\Com)$ then $\|\psi(\cL)\|_{\Hilbertzeroalt\to\Hilbertzeroalt} \leq \sup_{\lambda \in \sigma(\cL)} |\psi(\lambda)|$. 

(iii) If $\psi \in L^\infty(\Rea;\Rea)$ is such that $\psi\geq c>0$ on $\sigma(\cL)$, then 
\beqs
\big( \psi(\cL) v, v \big)_{\Hilbertzeroalt} \geq c \N{v}^2_{\Hilbertzeroalt} \quad\tfa v \in \Hilbertzeroalt.
\eeqs
\end{theorem}

\bpf[References for the proof]
See, e.g., \cite[\S VIIII.3, Page 259]{ReSi:80}.
\epf

\ble[Properties of $\smoother$ inherited from the functional calculus]\label{lem:SPpowers}
If 
$\smoother:=\psi(\eta^{-1} \cP  )$ then 

(a) $\smoother: \Pi_1 \Hilbertzero  \to \Pi_1 \Hilbertzero $.

(b) $\eta \smoother = \smoother^*\eta$, where $\smoother^*: \Pi_1^*\Hilbertzero^*\to \Pi_1\Hilbertzero$.

(c) Given $m\in \mathbb{Z}^+$, there exists $C>0$ such that 
\beqs
\N{\Pi_1(\eta^{-1} \cP )^m \psi(\eta^{-1} \cP  )\Pi_1 }_{ \Hilbertzero  \to  \Hilbertzero }\leq C.
\eeqs
\ele

\bpf
Part (a) follows immediately from Part (i) of Theorem \ref{thm:func_calc}.
Since $\psi$ is real valued, 
by Part (i) of Theorem \ref{thm:func_calc}, $\smoother: \Pi_1 \Hilbertzero  \to \Pi_1 \Hilbertzero$ satisfies 
\beqs
(\smoother u,v)_{\Pi_1 \Hilbertzero}=(u,\smoother v)_{\Pi_1 \Hilbertzero} \quad\tfa u,v \in \Pi_1\Hilbertzero. 
\eeqs
Therefore, by the definition of $\eta$ \eqref{eq:eta}, 
\beqs
\langle \eta \smoother u, v\rangle_{\Pi_1^*\Hilbertzero^*\times \Pi_1 \Hilbertzero} 
=\langle  \eta u,  \smoother v\rangle_{\Pi_1^*\Hilbertzero^*\times \Pi_1 \Hilbertzero},
\eeqs
so that Part (b) follows.

Finally, since $\psi$ has compact support, the function $t\mapsto t^m \psi(t)$ is bounded for all $m\geq 0$; Part (c) then follows by Part (ii) of Theorem \ref{thm:func_calc}.
\epf

\ble\label{lem:S1}
With $\smoother$ defined by \eqref{eq:S}, $\Pone_{11} -\Ptwo_{11}+\eta\smoother^2: \Pi_1\Hilbert \to \Pi_1^* \Hilbert ^*$ is continuous, 
\beq\label{eq:Phashcoercivity1}
\Re\big\langle \big(\Pone_{11} -\Ptwo_{11}+\eta\smoother^2\big) v, v\big\rangle_{\Hilbert ^*\times \Hilbert } \geq C \N{v}^2_{\Hilbert } \quad\tfa v\in \Pi_1\Hilbert,
\eeq
and thus $\Pone_{11} -\Ptwo_{11}+\eta\smoother^2: \Pi_1\Hilbert \to \Pi_1^* \Hilbert ^*$ is invertible.
\ele

\bpf
Since $\smoother^2: \Pi_1\Hilbert \subset \Pi_1 \Hilbertzero\to \Pi_1\Hilbertzero$ 
and $\eta: \Pi_1 \Hilbertzero\to \Pi_1^*\Hilbertzero^* \subset \Pi_1^*\Hilbert^*$ 
are continuous, 
$\eta\smoother^2: \Pi_1\Hilbert \to \Pi_1^* \Hilbert ^*$ is continuous.
Since $\Pone_{11}$ and $\Ptwo_{11}$ are continuous $\Pi_1\Hilbert \to \Pi_1^* \Hilbert ^*$ by assumption, the continuity result follows. 

For the coercivity, by the definition of $\smoother$ \eqref{eq:S},
\begin{align}\nonumber
\Re\big\langle \big(\Pone_{11} -\Ptwo_{11}+\eta\smoother^2\big) v, v\big\rangle_{\Hilbert ^*\times \Hilbert } 
&=\big\langle \big( \cP
+\eta \psi^2(\eta^{-1} \cP ) \big) v, v\big\rangle_{\Hilbert ^*\times \Hilbert }.
\end{align}
By \eqref{eq:Qsesqui} and the fact that $\eta ^{-1}$ is the identification map $\Pi_1^* \Hilbertzero^* \to \Pi_1\Hilbertzero$, 
\beqs
\big\langle \cP v, v\big\rangle_{\Hilbert^*\times \Hilbert} 
=\big\langle \cP v, v\big\rangle_{\Hilbertzero^*\times \Hilbertzero} 
= \big( \eta ^{-1} \cP v, v\big)_{\Pi_1\Hilbertzero}
 \quad\tfa v\in {\rm Dom}(\cP)\subset\Pi_1\Hilbertzero.
\eeqs
Therefore, by the inequality  \eqref{eq:psi_inequality} and Part (iii) of Theorem \ref{thm:func_calc}, for all $v\in  {\rm Dom}(\cP)$, 
\begin{align}
\Re\big\langle \big(\Pone_{11} -\Ptwo_{11}+\eta\smoother^2\big) v, v\big\rangle_{\Hilbert ^*\times \Hilbert } 
=\big( \big( \eta ^{-1}\cP + \psi^2(\eta ^{-1}\cP)\big)v,v\big)_{\Pi_1\Hilbertzero}
\geq \N{v}^2_{\Hilbertzero }.
\label{eq:coercivity_temp1}
\end{align}
Since ${\rm Dom}(\cP)$ is dense in $\Pi_1\Hilbertzero$ (since $\cP$ is densely-defined by Corollary \ref{cor:Friedrichs}), \eqref{eq:coercivity_temp1} holds for all $v\in \Pi_1\Hilbert$.

We now use the G\aa rding inequality \eqref{eq:GardingP111}  to replace the $\Hilbertzero$ norm on the right-hand side of \eqref{eq:coercivity_temp1} by a $\Hilbert$ norm and obtain \eqref{eq:Phashcoercivity1}. Since $\psi^2 \geq 0$, $\smoother^2\geq 0$. Using this, along with \eqref{eq:GardingP111} and \eqref{eq:coercivity_temp1}, we find that, for all $\varepsilon>0$ and $v\in \Pi_1 \Hilbert$,
\begin{align*}
&\Re\big\langle \big(\Pone_{11} -\Ptwo_{11}+\eta\smoother^2\big) v, v\big\rangle_{\Hilbert ^*\times \Hilbert } \\
&\quad \geq \varepsilon \Re\big\langle \big(\Pone_{11} -\Ptwo_{11}\big) v, v\big\rangle_{\Hilbert ^*\times \Hilbert } + (1-\varepsilon)\Re\big\langle \big(\Pone_{11} -\Ptwo_{11}+\eta\smoother^2\big) v, v\big\rangle_{\Hilbert ^*\times \Hilbert } \\
&\quad \geq \varepsilon 
\Big(\N{ v}^2_{\Hilbert } - (1+  \|\Ptwo\|_{\Hilbertzero\to\Hilbertzero^*})\N{ v}^2_{\Hilbertzero }\Big) 
+(1-\varepsilon) \N{v}^2_{\Hilbertzero }
\end{align*}
so that, choosing $0<\varepsilon\leq (2+ \|\Ptwo\|_{\Hilbertzero\to\Hilbertzero^*})^{-1}$, we see that $\Pone_{11} -\Ptwo_{11}+\eta\smoother^2$ is coercive $\Pi_1\Hilbert \to \Pi_1^*\Hilbert^*$.

Invertibility of $\Pone_{11} -\Ptwo_{11}+\eta\smoother^2: \Pi_1\Hilbert \to \Pi_1^* \Hilbert ^*$ then follows from the Lax--Milgram lemma (where, as in the application of 
Theorem \ref{thm:Friedrichs} we identify  
$(\Pi_1\Hilbert)^*$ and $\Pi_1^* \Hilbert ^*$).
\epf

Let 
\beq\label{eq:Phash}
\Phash = 
\begin{pmatrix}
-\Ptwo_{00} & 0\\
-\Ptwo_{10} & \Pone_{11} -\Ptwo_{11} +\eta\smoother^2
\end{pmatrix} 
= \operator + \Pi^*_1 \eta\smoother^2 \Pi_1.
\eeq
With this definition, we record for later that \eqref{eq:Phashcoercivity1} is equivalent to 
\beq\label{eq:Phashcoercivity2}
\Re\big\langle \Phash\Pi_1 v,\Pi_1 v\big\rangle_{\Hilbert ^*\times \Hilbert } \geq C \N{\Pi_1 v}^2_{\Hilbert } \quad\tfa v\in \Hilbert .
\eeq

\ble\label{lem:Phash_inverse}
$\Rhash:\Hilbert ^*\to \Hilbert $ is well defined with $\| \Rhash\|_{\Hilbert ^*\to \Hilbert }\leq C$.
\ele

\bpf
By the matrix form of $\Phash$ \eqref{eq:Phash} and the fact that $\Pone_{11} -\Ptwo_{11}+\eta\smoother^2: \Pi_1\Hilbert \to \Pi_1^* \Hilbert ^*$ is invertible, the result follows if $\Ptwo_{00}:\Pi_0\Hilbert \to \Pi_0^*\Hilbert ^*$ is invertible. We claim that $\Ptwo_{00}: \Pi_0 \Hilbert \to \Pi_0^* \Hilbert ^*$ satisfies
\beq\label{eq:P200coercivity}
\Re \big\langle \Ptwo_{00} v,v
\big\rangle_{\Hilbert ^*\times\Hilbert } \geq C_{\Ptwo } \N{v}^2_{\Hilbert } \quad\tfa v\in \Pi_0 \Hilbert ,
\eeq
from which the result follows by the Lax--Milgram lemma (where we identify $\Pi_0^*\Hilbert^*$ and $(\Pi_0 \Hilbert)^*$).

By (in this order) the definition $\Ptwo_{00}:=\Pi_0^*\Ptwo\Pi_0$ \eqref{eq:matrixnotation}, 
the inclusion $\Pi_0 \Hilbert \subset \Hilbertzero $ (by \eqref{eq:simple_norm2}),
the boundedness of $\Ptwo  :\Hilbertzero \to\Hilbertzero^*$, 
 the coercivity of $\Ptwo $  \eqref{eq:P2}, and \eqref{eq:simple_norm2}, for all $v\in \Hilbert$,
\begin{align*}
\Re \big\langle \Ptwo_{00} \Pi_0 v,\Pi_0 v
\big\rangle_{\Hilbert ^*\times\Hilbert } = 
\Re \big\langle \Ptwo  \Pi_0v,\Pi_0v
\big\rangle_{\Hilbert ^*\times\Hilbert } 
&= 
\Re \big\langle \Ptwo  \Pi_0v,\Pi_0v
\big\rangle_{\Hilbertzero ^*\times\Hilbertzero } \\
&\geq C_{\Ptwo } \N{\Pi_0 v}^2_{\Hilbertzero } = C_{\Ptwo } \N{\Pi_0 v}^2_{\Hilbert };
\end{align*}
i.e., \eqref{eq:P200coercivity} holds and the proof is complete.
\epf

\ble\label{lem:PhashGarding}
$\Phash: \Hilbert \to \Hilbert ^*$ satisfies a
G\aa rding inequality; i.e., there exists $C_1,C_2>0$ such that
\beq\label{eq:GardingPhash1}
\Re \big\langle \Phash v, v\big\rangle_{\Hilbert^*\times \Hilbert} \geq C_1 \N{v}^2_{\Hilbert }-C_2 \N{ v}^2_{\Hilbertzero } \quad\tfa v\in \Hilbert .
\eeq
\ele

\bpf
We first claim that it is sufficient to prove that there exist $C_1',C_2'>0$ such that, for all $v\in \Hilbert$,
\beq\label{eq:GardingPhash}
\Re \big\langle \Phash v, v\big\rangle_{\Hilbert^*\times \Hilbert} \geq C_1' \N{\Pi_1 v}^2_{\Hilbert }-C_2' \N{\Pi_0 v}^2_{\Hilbertzero } \quad\tfa v\in \Hilbert .
\eeq
Indeed, by \eqref{eq:simple_norm2} and \eqref{eq:PeterPaul},
\beqs
\N{v}_{\Hilbert}^2 \leq \big( \N{\Pi_1 u}_{\Hilbert} + \N{\Pi_0 u}_{\Hilbertzero}\big)^2 \leq (1 + \varepsilon)  \N{\Pi_1 u}_{\Hilbert}^2 + (1 + \varepsilon^{-1}) \N{\Pi_0 u}_{\Hilbertzero}^2,
\eeqs
for all $\varepsilon>0$, 
so that, if \eqref{eq:GardingPhash} holds, then
\beqs
\Re \big\langle \Phash v, v\big\rangle_{\Hilbert^*\times \Hilbert}
\geq C_1'(1+\varepsilon)^{-1} \N{v}^2_\Hilbert - \Big[ 
C_1'(1+\varepsilon^{-1})(1+ \varepsilon)^{-1}+C_2'
\Big]\N{\Pi_0 v}^2_{\Hilbertzero },
 \eeqs
 and \eqref{eq:GardingPhash1} follows since $\Pi_0:\Hilbertzero\to\Hilbertzero$ is bounded.

We therefore now prove \eqref{eq:GardingPhash}. By the coercivity of $\Phash$ on $\Pi_1 \Hilbert $ \eqref{eq:Phashcoercivity2}, the boundedness of $\Phash:\Hilbert\to \Hilbert^*$, and \eqref{eq:simple_norm2}, 
\begin{align*}
\Re\langle \Phash v,v\rangle_{\Hilbert^*\times \Hilbert} &= \Re \big\langle \Phash \Pi_1 v, \Pi_1 v\big\rangle_{\Hilbert^*\times \Hilbert} + \Re \big\langle \Phash \Pi_0 v,\Pi_1 v\big\rangle_{\Hilbert^*\times \Hilbert} \\
&\qquad \qquad+ \Re \big\langle \Phash \Pi_1 v, \Pi_0 v\big\rangle_{\Hilbert^*\times \Hilbert} + \Re \big\langle \Phash \Pi_0 v, \Pi_0 v\big\rangle_{\Hilbert^*\times \Hilbert},\\
&\geq C_3 \N{\Pi_1 v}^2_\Hilbert  - C_4 \N{\Pi_0 v}_\Hilbert  \N{\Pi_1 v}_\Hilbert  - C_5 \N{\Pi_0 v}^2_\Hilbert ,\\
&=C_3 \N{\Pi_1 v}^2_\Hilbert  - C_4 \N{\Pi_0 v}_{\Hilbertzero } \N{\Pi_1 v}_\Hilbert  - C_5 \N{\Pi_0 v}^2_{\Hilbertzero },
\end{align*}
and \eqref{eq:GardingPhash} follows from the inequality \eqref{eq:PeterPaul}.
\epf

\subsection{$\smoother=\psi(\eta ^{-1}\cP)$ increases regularity}

The main result of this subsection is the following.

\ble[$\smoother$ increases regularity]\label{lem:Ssmooth}
Suppose that Assumption \ref{ass:2} holds  for some $\newell \in\mathbb{Z}^+$ and spaces $\cZ^j, j=1,\ldots,\newell+1$. 
Then there exists $C>0$ such that 
\beqs
\N{\smoother}_{\Pi_1 \Hilbertzero  \to \Pi_1\cZ^j} \leq C \quad\tfor j=0,\ldots, \newell-1.
\eeqs
\ele

To prove Lemma \ref{lem:Ssmooth}, 
we first combine the regularity assumptions \eqref{eq:er1} and \eqref{eq:P2reg}.

\ble\label{lem:ercP}
Suppose that Assumption \ref{ass:2} holds  for some $\newell \in\mathbb{Z}^+$ and spaces $\cZ^j, j=1,\ldots,\newell+1$. 
Then there exists $C>0$ such that for $j=2,\ldots, \newell+1$, 
\beq\label{eq:report1}
\N{\Pi_1 u}_{\cZ^j} \leq C \Big( 
\N{\Pi_1 u}_{\cZ^{j-2}} + 
\N{\Pi_1\eta ^{-1}\cP \Pi_1 u}_{\cZ^{j-2}}\Big)
\quad\tfa u \in \Hilbert. 
\eeq
\ele

\bpf
In preparation for applying \eqref{eq:er1} with $\mathsf{D}= \Re\Pone$, observe that, by \eqref{eq:matrices},
\begin{align*}
 \big\langle\Re\Pone \Pi_1 u, 
v\big\rangle_{\Hilbert^*\times\Hilbert}
= \big\langle\Pi^*_1
 (\Re\Pone_{11}) \Pi_1 u, 
v\big\rangle_{\Hilbert^*\times\Hilbert}
&= \big\langle\Pi^*_1
 \big(\cP+ \Re\Ptwo_{11}\big) \Pi_1 u, 
v\big\rangle_{\Hilbert^*\times\Hilbert}
\end{align*}
We now claim that 
\beq\label{eq:claim_sun1} 
\sup_{v\in \Pi_1\Hilbert , \| \iota  v\|_{(\cZ^{j-2})^*}=1} \big| 
\big\langle\Pi^*_1
 \big(\cP+ \Re\Ptwo_{11}\big) \Pi_1 u, 
v\big\rangle_{\Hilbert^*\times\Hilbert}
\big|
\leq  \big\|\eta^{-1}\Pi^*_1 \big(\cP+ \Re\Ptwo_{11}\big) \Pi_1 u\big\|_{\cZ^{j-2}},
\eeq
so that 
\begin{align*}
 \sup_{v\in \Pi_1\Hilbert , \| \iota v\|_{(\cZ^{j-2})^*}=1} \big| \big\langle \Re\Pone \Pi_1 u, 
v\big\rangle_{\Hilbert^*\times\Hilbert}\big|\leq \big\|\eta^{-1}\Pi^*_1 \big(\cP+ \Re\Ptwo_{11}\big) \Pi_1 u\big\|_{\cZ^{j-2}},
\end{align*}
and thus, by  \eqref{eq:er1} with $\mathsf{D}= \Re\Pone$, 
\beq\label{eq:Sat1}
\N{\Pi_1 u}_{\cZ^j} \leq C \Big(
\N{\Pi_1 u}_{\cZ^{j-2}} + 
\big\|\eta^{-1}\Pi^*_1 \big(\cP+ \Re\Ptwo_{11}\big) \Pi_1 u\big\|_{\cZ^{j-2}}\Big).
\eeq
We continue with the proof of \eqref{eq:report1} (using \eqref{eq:Sat1}), and then prove \eqref{eq:claim_sun1} at the end.
Now
\begin{align}\nonumber
\N{\eta^{-1}\Pi_1^*\Re \Ptwo_{11} \Pi_1 u }_{\cZ^{j-2}} &\leq 
\N{\eta^{-1}\Pi_1^*\iota}_{\cZ^{j-2}\to \cZ^{j-2}} \big\|\iota^{-1}\Re \Ptwo_{11}\big\|_{\cZ^{j-2}\to \cZ^{j-2}}\N{ \Pi_1 u }_{\cZ^{j-2}}\\
&\leq C \N{\eta^{-1}\Pi_1^*\iota}_{\cZ^{j-2}\to \cZ^{j-2}}\N{ \Pi_1 u }_{\cZ^{j-2}}
\label{eq:Sat2a}
\end{align}
by \eqref{eq:P2reg} with $\mathsf{E} = \Re \Ptwo $.
By Part (ii) of Lemma \ref{lem:eta}, $\eta^{-1}: \Pi_1^*\Hilbertzero^*\to \Pi_1\Hilbertzero$ is given by  $\eta^{-1}=\Pi_1 \iota^{-1}\Pi_1^*$ (since the inverse of the inclusion map is the projection map and vice versa). Therefore, by \eqref{eq:sun2} and \eqref{eq:sun1},
\beqs
 \N{\eta^{-1}\Pi_1^*\iota}_{\cZ^{j-2}\to \cZ^{j-2}} =  \N{\Pi_1\iota^{-1}\Pi_1^*\iota}_{\cZ^{j-2}\to \cZ^{j-2}}\leq C,
 \eeqs
and combining this with \eqref{eq:Sat2a} we obtain that 
\beq\label{eq:Sat2}
\N{\eta^{-1}\Pi_1^*\Re \Ptwo_{11} \Pi_1 u }_{\cZ^{j-2}} \leq C\N{ \Pi_1 u }_{\cZ^{j-2}}.
\eeq
Now $\eta^{-1} = \Pi_1 \eta^{-1} = \Pi_1 \eta^{-1}\Pi_1^*$ (either by the formula $\eta^{-1}=\Pi_1 \iota^{-1}\Pi_1^*$, or just the fact that  $\eta^{-1}: \Pi_1^*\Hilbertzero^*\to \Pi_1 \Hilbertzero$),
so that 
\beq\label{eq:iotamoving}
\eta^{-1}\Pi_1^*\cP \Pi_1= \Pi_1 \eta^{-1} \Pi_1^*\cP\Pi_1 =\Pi_1 \eta^{-1} \cP\Pi_1;
\eeq
the result \eqref{eq:report1} then follows from combining \eqref{eq:iotamoving} with \eqref{eq:Sat1} and \eqref{eq:Sat2}. 

It therefore remains to prove \eqref{eq:claim_sun1}.
By Corollary \ref{cor:Friedrichs}, $\cP: \Pi_1 \Hilbertzero\to \Pi_1^*\Hilbertzero^*$ and, since $\Ptwo: \Hilbertzero\to \Hilbertzero^*$ (by Assumption \ref{ass:1}), $\Ptwo_{11}:=\Pi_1^* \Ptwo \Pi_1 : \Pi_1 \Hilbertzero\to \Pi_1^*\Hilbertzero^*$. Therefore, for $v\in \Pi_1\Hilbert \subset \Pi_1 \Hilbertzero$, 
\begin{align}\label{eq:claim_sun2}
\big\langle  \Pi_1^*\big(\cP+ \Re\Ptwo_{11}\big) \Pi_1 u, 
v\big\rangle_{\Hilbert^*\times\Hilbert}
=
\big\langle \Pi_1^* \big(\cP+ \Re\Ptwo_{11}\big)\Pi_1 u, 
v\big\rangle_{\Pi_1^*\Hilbertzero^*\times\Pi_1\Hilbertzero}.
\end{align}
Now, by (in this order)  Lemma \ref{lem:eta_inverse}, the formula $\eta = \Pi_1^*\iota \Pi_1$, and the fact that $\Pi_1 \eta^{-1} =\eta^{-1}$ (since $\eta^{-1}: \Pi_1^*\Hilbertzero^* \to \Pi_1 \Hilbertzero$), for $v\in \Pi_1\Hilbert \subset \Pi_1 \Hilbertzero$
\begin{align}\nonumber
\big\langle \Pi^*_1  \big(\cP+ \Re\Ptwo_{11}\big) \Pi_1 u, 
v\big\rangle_{\Pi_1^*\Hilbertzero^*\times\Pi_1\Hilbertzero}
&
=\big\langle \eta^{-1}\Pi^*_1  \big(\cP+ \Re\Ptwo_{11}\big) \Pi_1 u, 
\eta v\big\rangle_{\Pi_1^*\Hilbertzero^*\times\Pi_1\Hilbertzero}\\ \nonumber
&=\big\langle \eta^{-1}\Pi^*_1  \big(\cP+ \Re\Ptwo_{11}\big) \Pi_1 u, 
\Pi_1^*\iota \Pi_1 v\big\rangle_{\Pi_1^*\Hilbertzero^*\times\Pi_1\Hilbertzero}\\ \label{eq:claim_sun3}
&=\big\langle \eta^{-1}\Pi^*_1  \big(\cP+ \Re\Ptwo_{11}\big) \Pi_1 u, 
\iota v\big\rangle_{\Pi_1^*\Hilbertzero^*\times\Pi_1\Hilbertzero}.
\end{align}
The claimed bound \eqref{eq:claim_sun1} then follows from the combination of \eqref{eq:claim_sun2} and \eqref{eq:claim_sun3}.
\epf

The final result we need to prove Lemma \ref{lem:Ssmooth} is the following.

\ble\label{lem:Smapez}
$\smoother: \Pi_1\Hilbertzero\to \Hilbert$.
\ele

\bpf
By its definition \eqref{eq:S}, $\smoother:=\psi(\eta^{-1}\cP):\Pi_1 \Hilbertzero \to \Pi_1 \Hilbertzero$. 
Given $v\in\Pi_1 \Hilbertzero$, to bound $\|\smoother v\|_{\Hilbert}$ 
it is sufficient to prove that 
\beq\label{eq:Sat3}
\big| \big\langle \cP \psi(\eta^{-1}\cP)v, \psi(\eta^{-1}\cP)v\big\rangle_{\Hilbert^*\times \Hilbert}\big|
\leq C \N{v}_{\Hilbertzero}^2 \quad\tfa v \in \Pi_1\Hilbertzero
\eeq 
by the G\aa rding inequality \eqref{eq:GardingP111}.`
By \eqref{eq:Qsesqui} and the fact that $\eta^{-1}$ is the identification $\Pi_1^*\Hilbertzero^* \to \Pi_1\Hilbertzero$,
\begin{align*}
\big\langle \cP \psi(\eta^{-1}\cP)v, \psi(\eta^{-1}\cP)v\big\rangle_{\Hilbert^*\times \Hilbert}
&=
\big\langle \cP \psi(\eta^{-1}\cP)v, \psi(\eta^{-1}\cP)v\big\rangle_{\Hilbertzero^*\times \Hilbertzero}\\
&= \big(\eta^{-1} \cP \psi(\eta^{-1}\cP)v, \psi(\eta^{-1}\cP)v\big)_{\Hilbertzero}.
\end{align*}
The bound \eqref{eq:Sat3} then follows from Lemma \ref{lem:SPpowers}.
\epf

\bpf[Proof of Lemma \ref{lem:Ssmooth}]

We apply Lemma \ref{lem:ercP} with $u=\smoother \Pi_1 v= \psi(\eta^{-1} \cP ) \Pi_1 v$ for arbitrary $v\in \Hilbertzero$; observe that this is allowed since $u\in \Hilbert$ by Lemma \ref{lem:Smapez}. 
Since $\Pi_1 \psi(\eta^{-1} \cP  )= \psi(\eta^{-1} \cP  )$ (since $\psi(\eta^{-1} \cP  )$ is defined using the functional calculus on $\Pi_1 \Hilbertzero$), this application of Lemma \ref{lem:ercP} implies that 
\begin{align}
\N{\Pi_1 
\psi(\eta^{-1} \cP  )\Pi_1}_{\Hilbertzero \to
\cZ^j} 
&\leq C \Big( \N{\Pi_1
\psi(\eta^{-1} \cP  )\Pi_1}_{\Hilbertzero \to 
\cZ^{j-2}} + 
\N{\Pi_1\eta^{-1}
\cP \psi(\eta^{-1} \cP  )\Pi_1}_{\Hilbertzero  \to 
\cZ^{j-2}}\Big).
\label{eq:er_Pi1}
\end{align}
We now apply Lemma \ref{lem:ercP} with $u=
(\eta^{-1} \cP )^m\psi(\eta^{-1} \cP  )\Pi_1v$ for arbitrary $v\in \Hilbertzero$.
The proof that this $u\in \Hilbert$ is very similar to the proof of Lemma \ref{lem:Smapez}, using Lemma \ref{lem:SPpowers} -- the key points are that (i) any compactly supported function of $\eta^{-1} \cP $ is bounded $\Pi_1 \Hilbertzero\to \Pi_1 \Hilbertzero$, and (ii) the $\Hilbert$ norm essentially just adds another power of $\eta^{-1} \cP $ by the G\aa rding inequality \eqref{eq:GardingP111}). 

Lemma \ref{lem:ercP} and the fact that $\Pi_1 \eta^{-1} \cP = \eta^{-1} \cP $
(since $\eta^{-1}\cP : \Pi_1 \Hilbertzero\to  \Pi_1 \Hilbertzero$) 
 therefore imply that
\begin{align}\nonumber
&\N{
\Pi_1(\eta^{-1} \cP )^m\psi(\eta^{-1} \cP  )\Pi_1}_{\Hilbertzero \to \cZ^j}\\
&\quad \leq C \Big( \N{
 \Pi_1(\eta^{-1} \cP )^m\psi(\eta^{-1} \cP  )\Pi_1}_{\Hilbertzero \to  \cZ^{j-2}} + 
\N{\Pi_1
 (\eta^{-1} \cP )^{m+1} \psi(\eta^{-1} \cP  )\Pi_1}_{
\Hilbertzero  \to 
\cZ^{j-2}}\Big).\label{eq:er_Pi2}
\end{align}
The combination of \eqref{eq:er_Pi1} and \eqref{eq:er_Pi2} implies that \beqs
\N{\Pi_1
 \psi(\eta^{-1}\cP)\Pi_1}_{ \Hilbertzero \to \cZ^{\newell-1}} \leq C_\newell \sum_{j=0}^{\lceil (\newell-1)/2\rceil} 
\N{\Pi_1(\eta^{-1}\cP)^{j} \psi(\cP)\Pi_1}_{1\Hilbertzero  \to \Hilbertzero },
\eeqs
and the result then follows from Lemma \ref{lem:SPpowers}.
\epf

\subsection{Regularity of $\Rhash\Pi_1^*$}

\ble\label{lem:regRhash}
Suppose that Assumptions \ref{ass:1} and  \ref{ass:2} hold, the latter  for some $\newell \in\mathbb{Z}^+$ and spaces $\cZ^j, j=1,\ldots,\newell+1$. Then
\beqs
\big\| \Rhash \Pi_1^* \iota\big\|_{\cZ^{j-2}\to \cZ^j} \leq C \quad\tfor j=2,\ldots,\newell+1.
\eeqs
\ele

\bpf
The proof is similar to the proof of Lemma \ref{lem:regRs}, but it is simpler since it turns out that now $\Pi_0u=0$. 
Given $f\in \cZ^{j-2}$, let $u= \Rhash \Pi_1^* \iota f$ so that $\Phash u = \Pi_1^*\iota f$.
By the definition of $\Phash$ \eqref{eq:Phash}, $\Ptwo_{00} \Pi_0 u =0$; i.e., $\Pi^*_0 \Ptwo  \Pi_0 u=0$ by \eqref{eq:matrixnotation}, and thus $\Pi_0 u =0$ by \eqref{eq:P2}.
Therefore, for $f\in \cZ^{j-2}$, by \eqref{eq:key1}, the definition of $\Phash$ \eqref{eq:Phash}, and the fact that $\eta = \Pi_1^*\iota\Pi_1$ (by Part (ii) of Lemma \ref{lem:eta}), 
\begin{align*}
\big| \big\langle \Pone  \Pi_1 u ,\Pi_1 v\big\rangle_{\Hilbert\times\Hilbert^*}\big| 
&= \big| \big\langle \Pi_1^* \iota f +\Ptwo  u - \Pi_1^* \eta\smoother^2 \Pi_1 u , \Pi_1 v\big\rangle_{\Hilbertzero\times\Hilbertzero^*}\big|\\
&= \big| \big\langle \iota^{-1}\big(\Pi_1^*\iota  f +\Ptwo  u - \Pi_1^* \eta\smoother^2 \Pi_1 u\big), \iota \Pi_1 v\big\rangle_{\Hilbertzero\times \Hilbertzero^*}\big|\\
&\leq C \Big[ \big\|\iota^{-1}\Pi_1^*\iota\big\|_{\cZ^{j-2}\to \cZ^{j-2}}\N{f}_{\cZ^{j-2}} + \big\|\iota^{-1}\Ptwo\big\|_{\cZ^{j-2}\to \cZ^{j-2}} \N{u}_{\cZ^{j-2}}
\\
&\hspace{2cm}
+ \big\|\iota^{-1}\Pi_1^*\iota\big\|_{\cZ^{j-2}\to\cZ^{j-2}}\N{\Pi_1 S^2 \Pi_1}_{\Hilbertzero\to \cZ^{j-2}}\N{\Pi_1 u}_{\Hilbertzero}
\Big]\\
&\hspace{4cm}
\N{\iota \Pi_1 \iota^{-1}}_{(\cZ^{j-2})^*\to (\cZ^{j-2})^*} 
\N{\iota v}_{(\cZ^{j-2})^*}.
\end{align*}
Thus, by \eqref{eq:sun1}, \eqref{eq:P2reg}, and Lemma \ref{lem:Ssmooth},
\beqs
\big| \big\langle \Pone  \Pi_1 u ,\Pi_1 v\big\rangle_{\Hilbert\times\Hilbert^*}\big|\leq C \Big(
\N{f}_{\cZ^{j-2}} + \N{u}_{\cZ^{j-2}} + \N{\Pi_1 u}_{\Hilbertzero }
\Big) 
\N{\iota v}_{(\cZ^{j-2})^*}
\eeqs
 for $j=2,\ldots,\newell+1$. 
 Inputting this last inequality into \eqref{eq:er1} with $\mathsf{D}= \Pone $ and recalling that $u=\Pi_1u$, we see that
\begin{align}\label{eq:birthday1}
\N{ u }_{\cZ^j}=\N{\Pi_1 u }_{\cZ^j} &\leq C \Big( \N{\Pi_1 u}_{\Hilbertzero } + \N{f}_{\cZ^{j-2}} + \N{u}_{\cZ^{j-2}}\Big)\leq C \Big( \N{f}_{\cZ^{j-2}} + \N{u}_{\cZ^{j-2}} \Big)
\end{align}
 for $j=2,\ldots,\newell+1$. 
Now $\| \Rhash\|_{\Hilbertzero^* \to\Hilbertzero }\leq C$ by Lemma \ref{lem:Phash_inverse} and the fact that $\Hilbert \subset \Hilbertzero$ and $\Hilbertzero^* \subset \Hilbert ^*$. Therefore, by \eqref{eq:birthday1} with $j=2$, $\N{u }_{\cZ^2} \leq C \N{f }_{\Hilbertzero^*}$; the result then follows by combining this with \eqref{eq:birthday1}.
\epf

\subsection{Quasi-optimality of $\Pihash$}

Our final task in \S\ref{sec:Phash} is to prove quasi-optimality of the projection $\Pihash:\Hilbert \to \Hilbert _h$ defined by
\beq\label{eq:GogPhash}
\big\langle \Phash v_h , (I-\Pihash)w\big\rangle_{\Hilbert^*\times\Hilbert} 
=0 \quad \tfa v_h\in \Hilbert _h;
\eeq
i.e., 
\beq\label{eq:GogPhash2}
\big\langle (\Phash)^*(I-\Pihash)w, v_h\big\rangle_{\Hilbert^*\times\Hilbert} 
=0 \quad \tfa v_h\in \Hilbert _h.
\eeq

\ble[Quasi-optimality of $\Pihash$]\label{lem:Pihash}
If $\operator$ satisfies Assumptions \ref{ass:1} and \ref{ass:2}, the latter with $\newell=1$, then there exist $C_1,C_2>0$ such that if 
\beq\label{eq:Pihashqo}
\gamma_{\rm dv}(\operator^*) \leq C_1
\quad \text{ then } \quad
\big\| (I-\Pihash)v\big\|_{\Hilbert } \leq C_2 \big\|(I-\Pi_h)v\big\|_{\Hilbert } \quad\tfa v\in \Hilbert.
\eeq
\ele

\bpf
The idea is to apply Corollary \ref{cor:basic} with $\operator$ replaced by $(\Phash)^*$ (so that $\operator^*$ is replaced by $\Phash$). 
We now need to check that the assumptions of Corollary \ref{cor:basic} are satisfied with this replacement. 

We first claim that 
\beq\label{eq:Phash*}
(\Phash)^* = \Pone^* - \Ptwo^* + \Pi^*_1 \eta\smoother^2 \Pi_1.
\eeq
Indeed, by the definition of $\Phash$ \eqref{eq:Phash}, \eqref{eq:Phash*} holds if $\Pi^*_1 \eta\smoother^2 \Pi_1$ is self-adjoint, and this holds 
by Part (b) of Lemma \ref{lem:SPpowers} and the fact that $\eta$ is self-adjoint (by Lemma \ref{lem:eta_inverse}).

Now, since  $\Phash$ satisfies the G\aa rding inequality \eqref{eq:GardingPhash1},  
$(\Phash)^*$ satisfies Assumption \ref{ass:1} with $\Pone$ set to $\Pone^* + \Pi^*_1 \eta\smoother^2 \Pi_1$ (which has the same kernel as $\Pone$) and $\Ptwo$ set to $\Ptwo^*$. 
Because of the regularity property of $\smoother$ in Lemma \ref{lem:Ssmooth}, if $\operator$ satisfies Assumption \ref{ass:2} with $\newell=1$, then so does $(\Phash)^*$; i.e., the assumptions of Corollary \ref{cor:basic} are satisfied with $\operator$ replaced by $(\Phash)^*$.

The result then follows if we can show that (i) $\gamma_{\rm dv}((\Phash)^*)= \gamma_{\rm dv}(\operator^*)$, and (ii) 
$\| (\Phash)^{-1}\Pi_1^*\|_{\Hilbertzero\to \Hilbertzero}\leq C$. 
Point (ii) is satisfied by Lemma \ref{lem:Phash_inverse} since $\Hilbert\subset \Hilbertzero\subset \Hilbert^*$. To show Point (i), observe that the projections $\Pi_0$ and $\Pi_1$ are now defined with $\Ptwo$ replaced by $\Ptwo^*$, and the analogue of \eqref{eq:Pi+P2} is now 
\beqs
\big\langle \Ptwo ^*(I-\Pihash)w, v_h\big\rangle =0 \quad\tfa v_h \in \Hilbert _h \cap \Ker \Pone
\eeqs
(this follows from \eqref{eq:GogPhash2} since $(\Phash)^* v_h =-\Ptwo^* v_h$ for $v_h\in \Ker \Pone $ by \eqref{eq:Phash*}). 
By \eqref{eq:gamma_dv}, $\gamma_{\rm dv}((\Phash)^*)= \gamma_{\rm dv}(\operator^*)$ and the proof is complete.
\epf

\section{Proof of Theorem \ref{thm:abs1} (the main abstract theorem)}\label{sec:abs_proof}

As noted below the statement of Theorem \ref{thm:abs1} the relative-error bound \eqref{eq:relative_error_abs} follows from the error bound \eqref{eq:qo_abs} and 
the regularity result of Lemma \ref{lem:oscil}. 

We now use a duality argument involving $\Phash$ to prove the error bound \eqref{eq:qo_abs}.

\subsection{Reducing bounding the Galerkin error to bounding $\|\smoother\Pi_1 (u-u_h)\|_{\Hilbertzero }$}

The following lemma is an improved version of Lemma \ref{lem:kernel_basic} (due to the presence of $\smoother$ on the right-hand side).

\ble[Galerkin quasi-optimality, modulo a norm of the error involving $\smoother$]\label{lem:plant1}
Suppose that $\operator$ satisfies Assumption \ref{ass:1}. Given $u\in \Hilbert$, assume that the solution 
$u_h\in \Hilbert_h$ of \eqref{eq:Gog} exists. Then  there exists $C_1, C_2>0$ such that 
\beq\label{eq:dark1}
\big(1 - C_1 \gamma_{\rm dv}(\operator)\big) \N{ u-u_h}_{\Hilbert } \leq C_2 \Big(\N{(I-\Pi_h)u}_{\Hilbert }+ \N{\smoother\Pi_1(u-u_h)}_{\Hilbertzero }\Big) \quad\tfa v\in \Hilbert .
\eeq
\ele

\bpf
We first argue exactly as at the start of Lemma \ref{lem:kernel_basic}.
By the triangle inequality, \eqref{eq:simple_norm2}, and \eqref{eq:Pi0error}, 
\begin{align*}
\N{u-u_h}_\Hilbert  &\leq  \N{\Pi_0(u-u_h)}_{\Hilbertzero }+\N{\Pi_1(u-u_h)}_\Hilbert \\
&\leq  C \Big( \N{(I-\Pi_h)u}_\Hilbert  + \gamma_{\rm dv}(\operator) \N{u-u_h}_{\Hilbert }\Big)+\N{\Pi_1(u-u_h)}_\Hilbert ;
\end{align*}
i.e., 
\beq\label{eq:coffee5v2}
\big(1 - C \gamma_{\rm dv}(\operator)\big) \N{ u-u_h}_{\Hilbert } \leq  C \N{(I-\Pi_h)u}_\Hilbert +\N{\Pi_1(u-u_h)}_\Hilbert .
\eeq
We now claim that it is
sufficient to prove the bound 
\begin{align}\nonumber
&\N{\Pi_1(u-u_h)}_\Hilbert \\
&\qquad\leq \varepsilon \N{u-u_h}_\Hilbert  +C \varepsilon^{-1} \Big(
\N{(I-\Pi_h)u}_\Hilbert + \N{\smoother\Pi_1(u-u_h)}_{\Hilbertzero }
+\N{\Pi_0(u-u_h)}_{\Hilbert }
\Big)
\label{eq:STP1}
\end{align}
(note that \eqref{eq:STP1} is identical to \eqref{eq:STP1_basic} apart from the $\smoother$ multiplying $\Pi_1(u-u_h)$ on the right-hand side). 
Indeed, inputting \eqref{eq:STP1} into \eqref{eq:coffee5v2} and using again \eqref{eq:Pi0error}, we find \eqref{eq:dark1}.

Now, by coercivity of $\Phash= \operator + \Pi^*_1 S^2 \Pi_1$ on $\Pi_1 \Hilbert $ \eqref{eq:Phashcoercivity2}, 
\beq\label{eq:Friday2}
\N{\Pi_1(u-u_h)}^2_\Hilbert  \leq \Re
 \big\langle \operator \Pi_1(u-u_h), \Pi_1 (u-u_h)
\big\rangle
+ \N{ S \Pi_1(u-u_h)}^2_{\Hilbertzero}
\eeq
(compare to \eqref{eq:Friday2_basic}). 
The arguments after \eqref{eq:Friday2_basic} then show that 
\begin{align*}
\N{\Pi_1(u-u_h)}^2_{\Hilbert } &\leq 
\varepsilon \N{u-u_h}_{\Hilbert }^2
\\
&\quad+ C \Big(
\varepsilon^{-1} \N{(I-\Pi_h)u}^2_{\Hilbert } 
+ \varepsilon^{-1} \N{\Pi_0 (u-u_h)}^2_{\Hilbertzero }+ 
\N{\smoother\Pi_1(u-u_h)}_{\Hilbertzero }^2
\Big)
\end{align*}
(compare to \eqref{eq:plant1});  
this implies \eqref{eq:STP1} and the proof is complete.
\epf

\subsection{Duality argument using $\Phash$ to bound 
$\|\smoother\Pi_1 (u-u_h)\|_{\Hilbertzero }$}\label{sec:final_duality}

\ble\label{lem:dualityL2}
Suppose that $\operator$ satisfies Assumption \ref{ass:1}. Given $u\in \Hilbert$, assume that the solution 
$u_h\in \Hilbert_h$ of \eqref{eq:Gog} exists. 
Suppose further that the projection $\Pihash$ \eqref{eq:GogPhash} is well-defined.
Then  there exists $C_1, C_2>0$ such that 
\begin{align*}\nonumber
&\bigg(
1- C_1 \big\|(I-\Pi_h) \Rhash \Pi_1^* \eta\smoother \Pi_1\big\|_{\Hilbertzero \to \Hilbert } \big\|(I-\Pihash) \Rs\Pi_1^* \eta\smoother \Pi_1\big\|_{\Hilbertzero \to \Hilbert }
\bigg)
\N{\smoother\Pi_1(u-u_h)}_{\Hilbertzero } \\
&\qquad\leq C_2 \big\|(I-\Pi_h)\Rs\Pi_1^* \eta\smoother \Pi_1\big\|_{\Hilbertzero \to \Hilbert } \N{(I-\Pi_h)u}_{\Hilbert }.
\end{align*}
\ele

Combining Lemmas \ref{lem:plant1} and \ref{lem:dualityL2} 
 immediately gives the 
following result. 

\begin{lemma}[The main abstract result without using regularity of $\Rs$ or $\Rhash$]\label{thm:abs2}

\

\noi Suppose that $\operator$ satisfies Assumption \ref{ass:1}. Given $u\in \Hilbert$, assume that the solution 
$u_h\in \Hilbert_h$ of \eqref{eq:Gog} exists. 
Suppose further that the projection $\Pihash$ \eqref{eq:GogPhash} is well-defined.
Then  there exists $C_1, C_2, C_3, C_4>0$ such that 
\begin{align*}
&\big(1- C_1 \gamma_{\rm dv}(\operator)\big)\times\\
&\quad
\bigg(
1- C_2 \big\|(I-\Pi_h) \Rhash \Pi_1^* \eta\smoother \Pi_1\big\|_{\Hilbertzero \to \Hilbert } \big\|(I-\Pihash) \Rs\Pi_1^* \eta\smoother \Pi_1\big\|_{\Hilbertzero \to \Hilbert }
\bigg)
\N{u-u_h}_{\Hilbert } \\
&\qquad 
\leq C_3\bigg(
1- C_2 \big\|(I-\Pi_h) \Rhash \Pi_1^* \eta\smoother \Pi_1 \big\|_{\Hilbertzero \to \Hilbert } \big\|(I-\Pihash) \Rs\Pi_1^* \eta\smoother \Pi_1\big\|_{\Hilbertzero \to \Hilbert } 
\\
&\hspace{6cm}+ C_4  \big\|(I-\Pi_h)\Rs\Pi_1^* \eta\smoother \Pi_1\big\|_{\Hilbertzero \to \Hilbert }\bigg)
\N{(I-\Pi_h)u}_{\Hilbert }.
\end{align*}
That is, if
\beq\label{eq:fishchips1}
\gamma_{\rm dv}(\operator) \quad\tand\quad 
\big\|(I-\Pi_h) \Rhash \Pi_1^* \eta\smoother \Pi_1 \big\|_{\Hilbertzero \to \Hilbert } \big\|(I-\Pihash) \Rs\Pi_1^* \eta\smoother \Pi_1\big\|_{\Hilbertzero \to \Hilbert }
\eeq
are both sufficiently small, then 
\beq\label{eq:fishchips2}
\N{u-u_h}_{\Hilbert } \leq C\Big( 1 + \big\|(I-\Pi_h)\Rs\Pi_1^* \eta\smoother \Pi_1\big\|_{\Hilbertzero \to \Hilbert }\Big)
\N{(I-\Pi_h)u}_{\Hilbert }.
\eeq
\end{lemma}

\bpf[Proof of Lemma \ref{lem:dualityL2}]
By the definition of $\eta$ \eqref{eq:eta}, the definition of $\Rs:\Hilbert^*\to \Hilbert$, 
Part (b) of Lemma \ref{lem:SPpowers}, 
Galerkin orthogonality \eqref{eq:Gog}, 
the definition of $\Phash$ \eqref{eq:Phash}, and Galerkin orthogonality for $\Phash$ \eqref{eq:GogPhash},
\begin{align}\nonumber
\N{\smoother \Pi_1(u-u_h)}^2_{\Hilbertzero } 
&=\big\langle \smoother\Pi_1(u-u_h), \eta \smoother\Pi_1 (u-u_h)\big\rangle_{\Hilbertzero\times \Hilbertzero^*},\\ \nonumber
&= \big\langle u-u_h , \Pi_1^* \smoother^*\eta \smoother\Pi_1 (u-u_h)\big\rangle_{\Hilbertzero\times\Hilbertzero^*},\\ \nonumber
&= \big\langle \operator(u-u_h) , \Rs \Pi_1^* \eta\smoother^2\Pi_1 (u-u_h)\big\rangle_{\Hilbert^*\times\Hilbert},\\ \nonumber
&= \big\langle \operator(u-u_h) , (I-\Pihash)\Rs \Pi_1^* \eta\smoother^2 \Pi_1 (u-u_h)\big\rangle_{\Hilbert^*\times\Hilbert},\\ \nonumber
&= \big\langle \Phash(I-\Pi_h)u , (I-\Pihash)\Rs \Pi_1^* \eta\smoother^2\Pi_1 (u-u_h)\big\rangle_{\Hilbert^*\times\Hilbert}, \\ \nonumber
&\qquad- \big\langle\Pi_1^* \eta\smoother^2 \Pi_1(u-u_h), ( I- \Pihash) \Rs \Pi_1^*\eta\smoother^2\Pi_1 (u-u_h)
\big\rangle_{\Hilbert^*\times\Hilbert},\\ \nonumber
&\leq C \N{(I-\Pi_h)u}_\Hilbert  \big\|(I-\Pihash)\Rs \Pi_1^*\eta \smoother\Pi_1\big\|_{\Hilbertzero \to \Hilbert }\N{ \smoother\Pi_1 (u-u_h)}_{\Hilbertzero } 
\\&\qquad+ \big|\big\langle
\Pi_1^* \eta\smoother^2 \Pi_1(u-u_h), ( I- \Pihash) \Rs \Pi_1^*\eta\smoother^2\Pi_1 (u-u_h)
\big\rangle\big|.\label{eq:hot1}
\end{align}

We now use a duality argument involving $\Phash$ to bound the final term. 
By \eqref{eq:GogPhash}, for $\phi \in \Hilbertzero^*$ and $w\in \Hilbert$,
\begin{align*}\nonumber
\big\langle \Pi_1^* \phi, (I-\Pihash)w\big\rangle_{\Hilbert^*\times\Hilbert}
&= \big\langle
\Phash \Rhash \Pi_1^*\phi, (I-\Pihash) w
\big\rangle_{\Hilbert^*\times\Hilbert}\\
&
 =\big\langle
\Phash (I-\Pi_h)\Rhash \Pi_1^*\phi, (I-\Pihash) w
\big\rangle_{\Hilbert^*\times\Hilbert},
\end{align*}
so that 
\beq\label{eq:hot2}
\big|\big\langle \Pi_1^* \phi, (I-\Pihash)w\big\rangle \big|
\leq C\big\| (I-\Pi_h)\Rhash \Pi_1^*\phi\big\|_\Hilbert \big\|(I-\Pihash) w\big\|_\Hilbert .
\eeq
We apply \eqref{eq:hot2} with $\phi=\eta\smoother^2 \Pi_1 (u-u_h)\in \Hilbertzero^*$ and $w= \Rs \Pi_1^*\eta\smoother^2 \Pi_1 (u-u_h)\in \Hilbert$ and combine it with \eqref{eq:hot1} to obtain
\begin{align*}
&\N{\smoother\Pi_1(u-u_h)}^2_{\Hilbertzero } \\
&\quad\leq C \N{(I-\Pi_h)u }_{\Hilbert } \big\| (I-\Pihash)\Rs \Pi_1^* \eta\smoother \Pi_1\big\|_{\Hilbertzero^* \to \Hilbert } \N{\smoother\Pi_1(u-u_h)}_{\Hilbertzero }\\
&\qquad+
C \big\|(I-\Pi_h) \Rhash \Pi_1^* \eta\smoother \Pi_1\big\|_{\Hilbertzero \to \Hilbert } \N{\smoother\Pi_1(u-u_h)}^2_{\Hilbertzero } \big\|(I-\Pihash) \Rs \Pi_1^*\eta\smoother \Pi_1\big\|_{\Hilbertzero \to \Hilbert },
\end{align*}
and the result follows.
\epf

\subsection{Proof of the error bound \eqref{eq:qo_abs}}

We now use Lemma \ref{thm:abs2} to prove the error bound \eqref{eq:qo_abs} under the condition that  the quantities in \eqref{eq:sufficiently_small} are sufficiently small.

By Lemma \ref{lem:Pihash} the projection $\Pihash$ is well-defined and satisfies \eqref{eq:Pihashqo}  if $\gamma_{\rm dv}(\operator^*)$ is sufficiently small. Therefore, the instances 
of $(I-\Pihash)$ in Lemma \ref{thm:abs2} can be replaced (up to constants) by $(I-\Pi_h)$. 

The result  \eqref{eq:qo_abs} then follows if we can show that
\beqs
\big\|\Rhash \Pi_1^* \eta\smoother \Pi_1 \big\|_{\Hilbertzero \to \cZ^{\newell+1}} \leq C 
\eeqs
and
\beqs
\big\|\Rs \Pi_1^* \eta\smoother \Pi_1 \big\|_{\Hilbertzero \to \cZ^{\newell+1}} \leq C \big(1 + \N{\Rs \Pi_1^* }_{\Hilbertzero^* \to \Hilbertzero}\big).
\eeqs
By the regularity property of $\smoother$ in Lemma \ref{lem:Ssmooth}, it is sufficient to prove that 
\beq\label{eq:scooter1a}
\big\|\Rhash \Pi_1^* \eta \Pi_1\big\|_{\cZ^{\newell-1} \to \cZ^{\newell+1}} \leq C 
\eeq
and
\beq\label{eq:scooter1b}
\N{\Rs \Pi_1^* \eta \Pi_1 }_{\cZ^{\newell-1}\to \cZ^{\newell+1}} \leq C \big(1 + \N{\Rs \Pi_1^* }_{\Hilbertzero^* \to \Hilbertzero}\big).
\eeq
By Part (b) of Lemma \ref{lem:eta},
$\Rhash \Pi_1^* \eta\Pi_1 =\Rhash \Pi_1^* \iota\Pi_1$ 
and 
$\Rs \Pi_1^* \eta\Pi_1 =\Rs \Pi_1^* \iota\Pi_1$. The bounds in \eqref{eq:scooter1a} and \eqref{eq:scooter1b} then follow from Lemmas \ref{lem:regRhash} and \ref{lem:regRs}, respectively, combined with \eqref{eq:sun2} (with $\Pi_0$ replaced by $\Pi_1$).

\section{Recap of the regularity result of Weber \cite{We:81}}\label{sec:Weber}

The following result is \cite[Theorem 2.2]{We:81}, where we observe that this result -- originally proved for real-valued coefficients -- immediately generalises to complex-valued coefficients.
 Recall the definitions of the piecewise spaces $\Hpw{j}(\Omega)$ \eqref{eq:pw_space} and the associated norm $\|\cdot\|_{\Hpwo{j}(\Omega)}$ \eqref{eq:pw_norm}.

\begin{theorem}[Regularity result for $\curl$ and $\dive$]\label{thm:Weber}
Suppose that $\zeta$ is a complex matrix-valued function on $\Omega$ satisfying $\Re \zeta \geq c>0$ (in the sense of quadratic forms).
Suppose further that, for some integer $\kappa\geq 1$, 
$\Omega$ is $C^{\kappa+1}$ with respect to the partition $\{\Omega_i\}_{i=1}^n$ (in the sense of Definition \ref{def:Crpartition}) and $\zeta \in C^\kappa(\Omega_j)$ for all $j=1,\ldots,n$. 

Then there exists $C>0$ such that, for all $0\leq \ell \leq \kappa-1$, if \emph{either} $u \times n= 0$ \emph{or} $(\zeta u)\cdot n=0$ on $\partial\Omega$ then
\beq\label{eq:Weber1}
\N{u}_{\Hpwo{\ell+1}(\Omega)} \leq C \Big( \N{u}_{L^2(\Omega)} + \N{\wn ^{-1}\curl u}_{\Hpwo{\ell}(\Omega)} + \N{ \wn ^{-1}\dive (\zeta u)}_{\Hpwo{\ell}(\Omega)}\Big).
\eeq
\end{theorem}

\bpf
The result for $\zeta$ real-valued and symmetric positive-definite and with norms not weighted with $\wn $ is \cite[Theorem 2.2]{We:81}. Repeating the proof but now weighting each derivative by $\wn ^{-1}$ gives the bound \eqref{eq:Weber1}. 

We now outline why the result holds for complex-valued $\zeta$ with $\Re \zeta \geq c>0$.
The proof of \cite[Theorem 2.2]{We:81}  begins by localising and mapping the boundary to a half-plane (using \cite[Lemma 3.1]{We:81}) -- this is unaffected by the change in assumptions on $\zeta$. 
The parts of the proof that depend on $\zeta$ then involve
\ben
\item difference-quotient arguments, and
\item the decomposition of an arbitrary $L^2$ vector field $F$ into $F_1+ F_2$, where $F_1 = \nabla f$ and $\dive (\zeta F_2)=0$, and \emph{either} $f \in H^1_0(\Omega)$ and $\zeta F_2 \in H(\dive;\Omega)$ \emph{or} $f\in H^1(\Omega)$ and $\zeta F_2 \in H(\dive;\Omega)$ with $(\zeta F_2)\cdot n=0$ on $\partial \Omega$ \cite[Lemmas 3.4 and 3.5]{We:81}.
\een
The arguments in Point 1 go through verbatim (noting that $\zeta$ is still invertible). The results in Point 2 are quoted from \cite[Lemmas 3.8 and 3.9]{We:80}, where they are proved using projections in the $L^2(\Omega)$ inner product weighted with $\zeta$. 
When $\zeta$ is complex valued, 
the results in Point 2 can be proved
via the following.
For the first result (when $f\in H^1_0(\Omega)$), given $F$, let $f\in H^1_0(\Omega)$ be the solution of the variational problem
\beq\label{eq:lastLM1}
\big( \zeta\nabla f,\nabla w\big)_{L^2(\Omega)}=\big( \zeta F,\nabla w\big)_{L^2(\Omega)}\quad\tfa w\in H^1_0(\Omega).
\eeq
When $\Re \zeta \geq c>0$, the solution of \eqref{eq:lastLM1} is unique by the Lax--Milgram lemma. Let $F_2: =F - \nabla f$, so that \eqref{eq:lastLM1} is the statement that $\dive(\zeta F_2)=0$.
For the second result (when $f\in H^1(\Omega)$), 
given $F$, let $f\in H^1(\Omega)$ be the solution of the variational problem
\beq\label{eq:lastLM2}
\big( \zeta\nabla f,\nabla w\big)_{L^2(\Omega)}=\big( \zeta F,\nabla w\big)_{L^2(\Omega)}\quad\tfa w\in H^1(\Omega).
\eeq
When $\Re \zeta \geq c>0$, the solution of \eqref{eq:lastLM2} (a Laplace Neumann problem) is unique up to constants, so that $F_2: =F - \nabla f$ is uniquely defined. Now \eqref{eq:lastLM2} for  $w\in H^1_0(\Omega)$ is the statement that $\dive(\zeta F_2)=0$ and \eqref{eq:lastLM2} for $w\in H^1(\Omega)$ implies that $n\cdot (\zeta F)=0$ by, e.g., \cite[Equation 3.33]{Mo:03}.
\epf

\section{Definition and properties of N\'ed\'elec finite elements}\label{sec:FEM}

\subsection{Curved tetrahedral mesh}\label{sec:mesh}

We consider a partition of $\Omega$ into a conforming mesh $\cT_h$ of
(curved) tetrahedral elements $K$ as in, e.g., \cite[Assumption 3.1]{MeSa:21}.
For $K \in \cT_h$ we denote by $\LF_K: \widehat K \to K$ the mapping
between the reference tetrahedron $\widehat K$ and the element $K$. We further
assume that the mesh $\cT_h$ is conforming with the partition $\{\Omega_i\}_{i=1}^n$ of $\Omega$ from Assumption \ref{ass:regularity}. This 
means that for each $K \in \cT_h$, there is a unique $i\in \{ 1,\ldots, n\}$ such that
$K \subset \overline{\Omega_i}$. 

\subsection{N\'ed\'elec finite element space}\label{sec:Nedelec}
Fix a polynomial degree $p \geq 1$.
Then, following \cite{Ne:80} (see also, e.g., \cite[Chapter 15]{ErGu:21}), we introduce the N\'ed\'elec
polynomial space
\begin{equation*}
\BN_p(\widehat K) = \BP_{p-1}(\widehat K) + \bx \times \BP_{p-1}(\widehat K),
\end{equation*}
where $\BP_s(\widehat K)$ consists of functions such that each component is a 
polynomial of degree $\leq s$ defined over $\widehat K$. 
(Note that in \cite[Chapter 15]{ErGu:21} the lowest-order elements correspond to $p=0$, whereas here they correspond to $p=1$.)
The associated approximation space is obtained by mapping the N\'ed\'elec polynomial space to the
mesh cells through a Piola mapping (see \eqref{eq:Piola_curl} below), leading to
\begin{equation*}
\Hilbert_h \eq
\left \{
\bv_h \in \BH_0(\ccurl,\Omega)
\,:\,
\big(D\LF_K\big)^T \left (\bv_h|_K \circ \LF_K \right ) \in \BN_p(\widehat K)
\quad
\tfa K \in \cT_h
\right \},
\end{equation*}
where $D\LF_K$ is the Jacobian matrix of $\LF_K$.

\begin{assumption}[Curved finite-element mesh]
\label{assumption_curved_fem}
The maps $\LF_K$ satisfy
\begin{equation}
\label{eq_mapK}
\|\partial^\alpha\LF_K\|_{L^\infty(\widehat{K})}
\leq
C 
L\left (\frac{h_K}{L}\right )^{|\alpha|}
\quad\tand\quad
\|(D\LF_K)^{-1}\|_{L^\infty(K)}
\leq
C h_K^{-1},
\end{equation}
for $1\leq |\alpha| \leq p+1$,  
where $h_K$ is the diameter of $K$.
\end{assumption}

Note that the bound \eqref{eq_mapK} with $|\alpha|=1$ corresponds to the mesh elements $K\in \CT_h$ being shape regular (as in, e.g., \cite[Equation 3.3]{MeSa:21}).

\subsection{High-order interpolation}\label{sec:interpolation}

\begin{theorem}[Interpolation results in $\Hilbert_h$]\label{thm:interpolation}
Given $\Omega$ (with diameter $L$) and $\cH_h$ 
there exists an interpolation operator $\CJ_h: 
Z^2 \to \Hilbert_h$ (with $Z^j$ defined by \eqref{eq:Zj})
and a constant $C$ 
such that for all $\newellfour\in \{1,\ldots,p\}$, 
 for all $K\in \cT_h$, and for all $v\in Z^{\newellfour+1}$, 
\begin{equation}\label{eq:interpolation1}
\|\bv - \CJ_h \bv\|_{L^2(K)}
\leq
C \left (\frac{\hK}{L} \right )^\newellfour
\sum_{j=1}^\newellfour
L^j \Big(|v|_{H^j(K)}+\hK|\curl v|_{H^j(K)}\Big)
\end{equation}
and
\begin{equation}\label{eq:interpolation2}
\|\curl(\bv - \CJ_h \bv)\|_{L^2(K)}
\leq
C \left (\frac{\hK}{L} \right )^\newellfour
\sum_{j=1}^\newellfour
L^j |\curl v|_{H^j(K)}.
\end{equation}
\end{theorem}

Theorem \ref{thm:interpolation} is proved using Assumption \ref{assumption_curved_fem} and standard scaling arguments for curved elements in Appendix \ref{app:interpolation}.

Observe that \eqref{eq:interpolation2} implies that 
\beq\label{eq:commuting}
\text{ if } \,\,\curl \bv=0\,\, \text{ then }
\,\, \curl (\CJ_h \bv)= 0.
\eeq

\begin{corollary}[Best approximation result in $\Hilbert_h$]
Given $\Omega$,  $p\in \mathbb{N}$, and $k_0>0$, there exists $C>0$ such that the following is true. With $\Pi_h$ the orthogonal projection $\Hilbert\to\Hilbert_h$, for  all $\newellfour \in\{1,\ldots, p\}$ and for all $v\in \newZ^{\newellfour+1}$,
\begin{align}
\N{ (I- \Pi_h) v}_{H_\wn (\curl,\Omega)} 
&\leq C (\wn h )^\newellfour
\N{v}_{\newZ^{\newellfour+1}_\wn}.
\label{eq:interpolation}
\end{align}
\end{corollary}

\bpf
By summing \eqref{eq:interpolation1} and \eqref{eq:interpolation2} over $K\in \cT_h$, 
recalling the assumption that $\cT_h$
is conforming with the partition $\{\Omega_i\}_{i=1}^n$ of $\Omega$ from Assumption \ref{ass:regularity}, and using that $\hK\leq h$, 
\beqs
\|\bv - \CJ_h \bv\|_{L^2(\Omega)}
\leq
C \bigg(\frac{h}{L}\bigg)^{\newellfour}
\sum_{i=1}^n \sum_{j=0}^\newellfour
L^j \Big(|v|_{H^j(\Omega_i)}+(\wn\hK)|\wn^{-1}\curl v|_{H^j(\Omega_i)}\Big)
\eeqs
and
\beqs
\|\wn^{-1}\curl(\bv - \CJ_h \bv)\|_{L^2(\Omega)}
\leq
C \bigg(\frac{h}{L}\bigg)^{\newellfour}
\sum_{i=1}^n 
\sum_{j=1}^\newellfour
L^j |\curl v|_{H^j(\Omega_i)}.
\eeqs
The result then follows by 
using the definitions of the norms 
$\|\cdot\|_{H_\wn(\curl,\Omega)}$ \eqref{eq:1knorm}, 
$\|\cdot\|_{\Hpwo{\newellfour}(\Omega)}$ \eqref{eq:pw_norm}, 
and 
$\|\cdot\|_{\newZ_\wn^{\newellfour+1}}$ \eqref{eq:Zjnorm}, along with the fact that $\wn L \geq k_0 L$ (to absorb factors of $(\wn L)^{-1}$ into the constant $C$).
\epf

\section{Proof of Theorem \ref{thm:intro}}\label{sec:Maxwell_proof}

To show that  Theorem \ref{thm:intro} follows from the abstract result Theorem \ref{thm:abs1}, we need to 
\ben
\item prove Lemma \ref{lem:Maxwell} (i.e., show that Maxwell fits into the abstract framework),
\item show that the assumptions of the second part of Lemma \ref{lem:oscil} hold when $\dive f=0$,
\item bound $\gamma_{\rm dv}(\operator)$ and $\gamma_{\rm dv}(\operator^*)$ and show that the condition on these in \eqref{eq:sufficiently_small} is weaker (when $\wn L\gg 1$) than the condition \eqref{eq:sufficiently_small} involving $\Rs$ (since only the latter appears in \eqref{eq:threshold}), and 
\item show that, given $\newell\in\mathbb{N}$ and $p\leq \newell$, there exists $C>0$ such that 
\begin{align}\label{eq:interpolation_new}
\N{I- \Pi_h}_{\newZ_\wn^{\newell+1}\to H_\wn (\curl,\Omega)} 
&\leq C (\wn h )^p.
\end{align}
\een 

Indeed, Theorem \ref{thm:intro} follows from Theorem \ref{thm:abs1} using these points, as well as the fact that the $L^2\to L^2$ norm of the adjoint solution operator equals the 
$L^2\to L^2$ norm of the solution operator (just from standard properties of adjoint operators) so that $\|(\operator^*)^{-1}\|_{\Hilbertzero^*\to\Hilbertzero} = \Csol$.

Points 1, 2, and 3 are proved in \S\ref{sec:lemMaxwellProof}, \S\ref{sec:10.2}, and \S\ref{sec:gamma_dv_main}, respectively. 
Regarding Point 3:~we show in Lemma \ref{lem:gamma_dv_main} below that $\max\{\gamma_{\rm dv}(\operator), \gamma_{\rm dv}(\operator^*)\}\leq C \wn h (1+\wn h)$, i.e., 
for the first inequality in \eqref{eq:sufficiently_small} to hold, $\wn h$ must be sufficiently small. This condition is indeed weaker when $\wn L\gg 1$ than the condition ``$(\wn h)^{2p} \Csol$ is sufficiently small'' arising from the second inequality in \eqref{eq:sufficiently_small}, since $\Csol \geq C\wn L$ (as recalled in Remark \ref{rem:Csol}).

Point 4 follows from \eqref{eq:interpolation}; indeed, 
given $\newell\in\mathbb{N}$ and $p\leq \newell$, the bound \eqref{eq:interpolation} with $\newellfour=p$ implies that 
there exists $C>0$ such that
\begin{align*}
\N{ (I- \Pi_h) v}_{H_\wn (\curl,\Omega)} 
&\leq C (\wn h )^p
\N{v}_{\newZ^{p+1}_\wn}\leq C(\wn h)^p \N{v}_{\newZ^{\newell+1}_\wn},
\end{align*}
so that \eqref{eq:interpolation_new} follows.

\subsection{Proof of Lemma \ref{lem:Maxwell}}\label{sec:lemMaxwellProof}

\subsubsection{Proof of Part (a).}

Since $H_0(\curl;\Omega)\supset C^\infty_0(\Omega)$ (by, e.g., \cite[Equation 3.42]{Mo:03}) and $C^\infty_0(\Omega)$ is dense in $L^2(\Omega)$, $\Hilbert=H_0(\curl;\Omega)$ is dense in $\Hilbertzero= L^2(\Omega)$.

\ble\label{lem:kernel}
If $\Pone  := \wn ^{-2} \curl \mu^{-1}\curl$ then, in $H_0(\curl,\Omega)$, $\Ker \Pone=\Ker\Pone^*= \Ker (\curl)$. 
\ele

\bpf
With $\Hilbert= H_0(\curl,\Omega)$, by, e.g., \cite[Theorem 3.31]{Mo:03},
\beqs
\langle \Pone u,v\rangle_{\Hilbert^*\times\Hilbert} = \langle u, \Pone^* v\rangle_{\Hilbert^*\times\Hilbert} = \wn ^{-2}\big(\mu^{-1}\curl u , \curl v\big)_{L^2(\Omega)}
\eeqs
for all $u,v \in H_0(\curl,\Omega)$. Therefore, if $\Pone u=0$, then 
\beqs
0 = \wn ^{-2}\big(\mu^{-1}\curl u , \curl u\big)_{L^2(\Omega)},
\eeqs
and thus $\curl u=0$ by \eqref{eq:coefficients_sign}. Identical arguments show that if $\Pone^*v =0$, then $\curl v=0$.
Clearly if $\curl u=0$ then $u\in \Ker \Pone\cap \Ker \Pone^*$, and the result follows.
\epf

The rest of 
Part (i) of Lemma \ref{lem:Maxwell} follows immediately from the definitions.

\subsubsection{Proof of Part (b).}\label{sec:lemMaxwellProofb}

\paragraph{Proof that Parts (ii) and (iii) of Assumption \ref{ass:2} hold.}

By,  e.g., \cite[Theorem 1.4.1.1, page 21]{Gr:85}, the bound \eqref{eq:P2reg} (i.e., Part (iii) of Assumption \ref{ass:2}) holds 
with $\cZ^j$ given by $Z^j$ \eqref{eq:Zj} if $\epsilon$ is piecewise $C^{\newell,1}$ with respect to the partition $\{\Omega_i\}_{i=1}^n$.
We now use Theorem \ref{thm:Weber}  to prove that 
the bound \eqref{eq:er1} (i.e., Part (ii) of Assumption \ref{ass:2})
holds when $\Omega$ is $C^{\newell+1}$ with respect to the partition $\{\Omega_j\}_{j=1}^n$ (in the sense of Definition \ref{def:Crpartition}) and $\mu,\epsilon \in C^{\newell}(\overline{\Omega_j})$ for all $j=1,\ldots, n$
(recall from Remark \ref{rem:regularity} that the combination of these regularity requirements is then Assumption \ref{ass:regularity}).

\ble\label{lem:Maxwell_reg}
Suppose that $\Omega$ is $C^{\newell+1}$ with respect to the partition $\{\Omega_i\}_{i=1}^n$ (in the sense of Definition \ref{def:Crpartition}).
Suppose that $\zeta_1, \zeta_2$ are complex, matrix-valued functions on $\Omega$ satisfying $\Re \zeta_j \geq c>0$, $j=1,2$, (in the sense of quadratic forms) and 
$\zeta_1, \zeta_2 \in C^{\newell}(\Omega_j)$ for all $j=1,\ldots,n$.

Then 
there exists $C>0$ such that the following is true for $j=2,\ldots, \newell+1$. 
Given $f\in \newZ^{j-2}$ (defined by \eqref{eq:Zj}), 
if $v\in H_0(\curl,\Omega)$ is such that 
\beq\label{eq:Maxwell_reg0}
\wn ^{-2}\curl (\zeta_1 \curl  v ) = f \in \newZ^{j-2} 
\quad\tand\quad \dive(\zeta_2 v) =0  \quad\tin \Omega, 
\eeq
then 
\beq\label{eq:Maxwell_reg}
\N{v}_{\newZ^j} 
\leq C \Big( \N{v}_{L^2} +
\N{f}_{\newZ^{j-2}}\Big).
\eeq
\ele

\bpf 
First observe that it is sufficient to prove the bound
\beq\label{eq:Maxwell_reg2}
\N{v}_{\newZ^j} \leq C \Big( \N{v}_{L^2} + \N{\wn ^{-1}\curl v}_{L^2} + \N{f}_{\newZ^{j-2}}\Big).
\eeq
Indeed, the weak form of the PDE \eqref{eq:Maxwell_reg0} and the fact that 
$\Re \zeta_1 \geq c>0$ imply that 
\beqs
\N{\wn ^{-1}\curl v}^2_{L^2(\Omega)} \leq c^{-1} \N{f}_{L^2(\Omega)} \N{v}_{L^2(\Omega)}; 
\eeqs
the term involving $\curl v$ on the right-hand side of \eqref{eq:Maxwell_reg2} can therefore be removed (since $j\geq 2$), with \eqref{eq:Maxwell_reg} the result.

Let $w := \wn ^{-1} \zeta_1 \curl v$ so that $\wn ^{-1}\curl w=f$. Observe that 
$\dive (\zeta_1^{-1}w)=0$ and $\dive f=0$. 
With ${\rm div}_T$ the surface divergence on $\partial\Omega$, by, e.g., \cite[Equation 3.52]{Mo:03}, 
on $\partial\Omega$, 
\beqs
n \cdot \curl v = {\rm div}_T (v \times n).
\eeqs
Since $v\in H_0(\curl,\Omega)$, 
$v\times n=0$ on $\partial\Omega$, and thus $n \cdot (\zeta_1^{-1} w)=0$ on $\partial \Omega$.

The regularity assumptions on $\Omega$, $\zeta_1,$ and $\zeta_2$ imply that we can apply Theorem \ref{thm:Weber} with $\kappa=\newell$ and $\zeta$ equal one of $\zeta_1$, $\zeta_2$, or their inverses. 
Therefore, Theorem \ref{thm:Weber} applied with $u=w$, $\zeta= \zeta_1^{-1}$, $\kappa=\newell$, and $\ell=j-2$, $j=2,\ldots,\newell+1$ (so that $\ell\leq \kappa-1$ as required by Theorem \ref{thm:Weber}), implies that
\beq\label{eq:rain1}
\N{w}_{\Hpwo{j-1}(\Omega)}\leq C\Big( \N{w}_{L^2}+ \N{\wn ^{-1}\curl w}_{\Hpwo{j-2}(\Omega)}\Big)
\leq C\Big( \N{\wn^{-1}\curl v}_{L^2}+ \N{f}_{\Hpw{j-2}(\Omega)}\Big).
 \eeq
Similarly, Theorem \ref{thm:Weber} applied with $u=f$, $\zeta=I$, and $\kappa=\newell$, and $\ell=j-3$ ($j=3,\ldots,\newell+1$, so that $\ell\leq \kappa-1$) implies that 
\beq\label{eq:rain2}
\N{f}_{\Hpwo{j-2}(\Omega)}\leq C\Big( \N{f}_{L^2(\Omega)}+ \N{\wn^{-1}\curl f}_{\Hpwo{j-3}(\Omega)}\Big) \quad\tfor j=3,\ldots,\newell+1.
\eeq
Since $\wn ^{-1}\curl v= \zeta_1^{-1} w$ and $\zeta_1^{-1}$ is piecewise $C^\newell$,
\beq\label{eq:scooter2}
\N{\wn ^{-1}\curl v}_{\Hpwo{j-1}(\Omega)}
=\big\|\zeta_1^{-1} w \big\|_{\Hpwo{j-1}(\Omega)} \leq C \big\|w \big\|_{\Hpwo{j-1}(\Omega)} \quad\tfor j=2, \ldots, \newell+1
\eeq
by, e.g., \cite[Theorem 1.4.1.1, page 21]{Gr:85}.

The combination of \eqref{eq:rain1}, \eqref{eq:rain2}, and \eqref{eq:scooter2} imply that 
\beq\label{eq:rain3}
\N{\wn ^{-1}\curl v}_{\Hpwo{j-1}(\Omega)}
\leq C\Big( \N{\wn^{-1}\curl v}_{L^2}+\N{f}_{L^2(\Omega)}+ \N{\wn^{-1}\curl f}_{\Hpwo{j-3}(\Omega)}\Big) \,\tfor j=3,\ldots,\newell+1.
\eeq

Theorem \ref{thm:Weber} applied with $u=v$, $\zeta= \zeta_2$, $\kappa=\newell$, and $\ell=j-2$, $j=2,\ldots,\newell+1$ (so that again $\ell\leq \kappa-1$ as required by Theorem \ref{thm:Weber}), implies that
\beq\label{eq:rain4}
\N{v}_{\Hpwo{j-1}(\Omega)}\leq C\Big( \N{v}_{L^2}+ \N{\wn ^{-1}\curl v}_{\Hpwo{j-2}(\Omega)}\Big).
\eeq
The combination of \eqref{eq:rain3} and \eqref{eq:rain4} implies that the bound \eqref{eq:Maxwell_reg2} holds for $j=3,\ldots,\newell+1$. 
The bound \eqref{eq:Maxwell_reg2} when $j=2$ then follows from combining \eqref{eq:rain1} and \eqref{eq:rain4}, both with $j=2$, and the proof is complete.
\epf

To prove that Part (ii) of Assumption \ref{ass:2} holds, 
we seek to apply Lemma \ref{lem:Maxwell_reg} with $v =\Pi_1u$, $\zeta_1= \mu^{-1},$ and $\zeta_2 = \Ptwo= \epsilon$.
Since $C^\infty(\overline{D})$ is dense in $H^s(D)$ for all $s\in \Rea$ and $D\subset \Rea^3$ open (see, e.g., \cite[Page 77]{Mc:00}), $L^2(\Omega)$ is dense in $(\Hpw{j}(\Omega))^*$ for $j\geq 1$,
so the assumption that $\Hilbertzero^{*}$ is dense in $(\cZ^j)^*=(\newZ^j)^*$ for $j\geq 1$ is satisfied. 
The bound \eqref{eq:er1} then follows from \eqref{eq:Maxwell_reg} if $\Pi_1 u$ is in $H_0(\curl,\Omega)$ and satisfies 
$\dive (\epsilon \Pi_1u)=0$. 
Recall from \eqref{eq:simple_norm2} that $\Pi_0:\Hilbert\to\Hilbert$, and thus also $\Pi_1:\Hilbert\to \Hilbert$. 
Thus, $u\in H_0(\curl,\Omega)$ implies that $\Pi_1 u \in H_0(\curl,\Omega)$. For the zero-divergence condition, 
observe that, since gradients are always inside the kernel of $\curl$, 
\beqs
\big\langle \grad \phi,\Pi_1^* v \big\rangle_{L^2(\Omega)} 
=\big\langle \Pi_1 (\grad \phi), v \big\rangle_{L^2(\Omega)} =0 \quad \tfa \phi \in H^1_0(\Omega) \tand v \in L^2(\Omega).
\eeqs
Thus, $\dive \Pi_1^* v=0$ for all $v\in L^2(\Omega)$ (by, e.g., \cite[Equation 3.33]{Mo:03}). 
By \eqref{eq:key2}, $\dive (\epsilon \Pi_1 u) =  \dive (\Ptwo  \Pi_1 u)=\dive(\Pi_1^* \Ptwo \Pi_1 u)$ so that $\dive (\epsilon \Pi_1 u)=0$ as required. 

\paragraph{Proof that Part (i) of Assumption \ref{ass:2} holds.}
By, e.g., \cite[Theorem 3.41]{Mo:03}, the kernel of the curl operator in $H_0(\curl,\Omega)$ equals $(\nabla H^1_0(\Omega))\oplus K_N(\Omega)$, where
\beq\label{eq:KN}
K_N(\Omega):= \Big\{ u \in H_0(\curl,\Omega)\, :\, \curl u= 0 \tand \dive u = 0 \tin \Omega\Big\}
\eeq
(the \emph{normal cohomology space}); note that the dimension of $K_N(\Omega)$ equals the number of connected components of $\partial \Omega$ minus one; see, e.g., 
\cite[Theorem 3.42]{Mo:03}.
Let 
$\Pi^{\Hilbertzero}_{\nabla H^1_0(\Omega)}$ and 
$\Pi^{\Hilbertzero}_{K_N(\Omega)}$ be the $\Hilbertzero$-orthogonal projections onto $\nabla H^1_0(\Omega)$ and $K_N(\Omega)$, respectively, so that 
\beq\label{eq:Pi0split}
\Pi^{\Hilbertzero}_0= \Pi^{\Hilbertzero}_{K_N(\Omega)} +\Pi^{\Hilbertzero}_{\nabla H^1_0(\Omega)}.
\eeq

\begin{lemma}\label{lem:CoDaNi}
If $\partial \Omega\in C^{\newell+1}$ then $K_N(\Omega) \subset Z^{m+1}$ (where $Z^{m+1}$ is defined by  \eqref{eq:Zj}).
\end{lemma}

That is, $\Pi^{\Hilbertzero}_{K_N(\Omega)}$ smooths to the maximal extent possible given the spaces $\{\cZ^j\}_{j=0}^{\newell+1}$, and, in particular, preserves regularity, as required for \eqref{eq:proj_reg}.  

\bpf[Proof of Lemma \ref{lem:CoDaNi}]
The definition of $K_N(\Omega)$ implies that elements of $K_N(\Omega)$ are solutions of the equation 
\beqs
\curl \curl u - \grad(\dive u) =0 \quad\tin \Omega.
\eeqs
By \cite[\S4.5]{CoDaNi:10}, 
this PDE is strongly elliptic in the sense of \cite[Definition 3.2.2]{CoDaNi:10}, and the PDE plus the boundary condition $u\times n=0$ on $\partial \Omega$ are then elliptic in the sense of \cite[Definition 2.2.31]{CoDaNi:10}; see \cite[Theorem 3.2.6]{CoDaNi:10}. Since $\partial \Omega\in C^{\newell+1}$ (by Assumption \ref{ass:regularity}), the elliptic-regularity result \cite[Theorem 3.4.5]{CoDaNi:10} implies that $K_N(\Omega)\subset H^{\newell+1}(\Omega)$. 
Since every $u\in K_N(\Omega)$ has $\curl u=0$ by definition, $K_N(\Omega)$ is therefore contained in $\newZ^{\newell+1}$ \eqref{eq:Zj}. 
\epf

Given $f\in L^2(\Omega)$, $\Pi^{\Hilbertzero}_{\nabla H^1_0(\Omega)} f =\nabla \phi$, where $\phi\in H_0^1(\Omega)$ is the unique solution of the variational problem
\beq\label{eq:Laplace}
(\nabla\phi, \nabla v)_{L^2(\Omega)}= (f, \nabla v)_{L^2(\Omega)} \quad\tfa v\in H^1_0(\Omega).
\eeq
Observe that this is the weak form of the PDE $\Delta \phi = \dive f$.

We now apply Theorem \ref{thm:Weber} with $u= \nabla\phi$, $\zeta= I$, and $\kappa=\newell$ (note that this is allowed because $\Omega$ is $C^{\newell+1}$ with respect to the partition). Since $\phi \in H^1_0(\Omega)$ implies that $\nabla \phi \in H_0(\curl;\Omega)$ (see, e.g., \cite[Equation 3.60/\S B.3]{Mo:03}), the boundary condition on $u$ in Theorem \ref{thm:Weber} is satisfied and, for $\ell=0,\ldots,\newell-1$, 
\beq\label{eq:Weber2a}
\N{\nabla \phi }_{\Hpwo{\ell+1}(\Omega)} \leq C \Big( \N{\nabla \phi}_{L^2(\Omega)} + \N{ \wn ^{-1}\dive f}_{\Hpwo{\ell}(\Omega)}\Big).
\eeq
By \eqref{eq:Laplace} with $v=\phi$, $\|\nabla \phi\|_{L^2(\Omega)}\leq  \N{f}_{L^2(\Omega)}$. Therefore, by  \eqref{eq:Weber2a} 
applied with $\ell+1=j-1$ (so that $\ell=0,\ldots,k-1$ corresponds to $j=2,\ldots, \newell+1$)
and the definition of $\|\cdot\|_{\Hpwo{j}(\Omega)}$ \eqref{eq:Zjnorm},
\beqs
\big\|\Pi^{\Hilbertzero}_{\nabla H^1_0(\Omega)} f\big\|_{\newZ^{j}_\wn}=\|\nabla\phi\|_{\newZ^{j}_\wn}
=\|\nabla\phi\|_{\Hpwo{j-1}(\Omega)}
\leq C 
\| f\|_{\Hpwo{j-1}(\Omega)}
\leq C \| f\|_{\newZ^{j}_\wn}
\eeqs
for $j=2,\ldots,\newell+1$. The splitting \eqref{eq:Pi0split} therefore implies that \eqref{eq:proj_reg} holds for $j=2,\ldots,\newell+1$. Since $\cZ^1=\Hilbert$, and $\Pi^{\Hilbertzero}_0:\Hilbert\to\Hilbert$ is bounded by \eqref{eq:simple_norm2}, \eqref{eq:proj_reg} holds for $j=1,\ldots,\newell+1$ and thus Part (i) of Assumption \ref{ass:2} holds, as required.

\paragraph{Proof that Part (iv) of Assumption \ref{ass:2} holds.}

We first show that the bound \eqref{eq:CG1} follows if we can show that there exists $C>0$ such that, for $j=1,\ldots,\newell+1$, 
\beq\label{eq:CG1alt}
\big\|\Pi^{\Hilbertzero}_{\nabla H^1_0(\Omega)} u \big\|_{\cZ^j} \leq C \Big( 
\big\| \Pi^{\Hilbertzero}_{\nabla H^1_0(\Omega)}\mathsf{E} \Pi^{\Hilbertzero}_{\nabla H^1_0(\Omega)}u\big\|_{\cZ^j} + \big\|\Pi^{\Hilbertzero}_{\nabla H^1_0(\Omega)}u \big\|_{\Hilbertzero}
\Big)\quad\tfa u\in\Hilbertzero.
\eeq
Indeed, by the smoothing property of $\Pi^{\Hilbertzero}_{K_N(\Omega)}$ in Lemma \ref{lem:CoDaNi} and the 
regularity preserving properties of $\Pi^{\Hilbertzero}_0$ \eqref{eq:proj_reg} and $\mathsf{E}$ 
 \eqref{eq:P2reg}, for all $u\in \Hilbertzero$,
\beqs
 \big\|\Pi^{\Hilbertzero}_{\nabla H^1_0(\Omega)} \mathsf{E} \Pi^{\Hilbertzero }_{\nabla H^1_0(\Omega)}u\big\|_{\cZ^j} 
\leq
\big\|\Pi^{\Hilbertzero}_0 \mathsf{E} \Pi^{\Hilbertzero }_0u\big\|_{\cZ^j} 
+ C \Big( \big\| \Pi^{\Hilbertzero}_{K_N(\Omega)} u\big\|_{\Hilbertzero} + \big\| \Pi^{\Hilbertzero}_{\nabla H^1_0(\Omega)} u\big\|_{\Hilbertzero}\Big).
\eeqs
Therefore, by \eqref{eq:Pi0split}, \eqref{eq:CG1alt}, and the fact that $\Pi^{\Hilbertzero}_{K_N(\Omega)}$ is smoothing, for all $u\in\Hilbertzero$, 
\begin{align*}
\big\|\Pi^{\Hilbertzero}_{0} u \big\|_{\cZ^j} 
&\leq \big\|\Pi^{\Hilbertzero}_{\nabla H^1_0(\Omega)} u \big\|_{\cZ^j} + \big\| \Pi^{\Hilbertzero}_{K_N(\Omega)} u\big\|_{\Hilbertzero} \\
&\leq C\Big(
\big\|\Pi^{\Hilbertzero}_0 \mathsf{E} \Pi^{\Hilbertzero }_0u\big\|_{\cZ^j} 
+  \big\| \Pi^{\Hilbertzero}_{K_N(\Omega)} u\big\|_{\Hilbertzero} + \big\| \Pi^{\Hilbertzero}_{\nabla H^1_0(\Omega)} u\big\|_{\Hilbertzero}\Big),
\end{align*}
and the result \eqref{eq:CG1} follows since $\| \Pi^{\Hilbertzero}_{K_N(\Omega)} u\|_{\Hilbertzero} 
+\| \Pi^{\Hilbertzero}_{\nabla H^1_0(\Omega)} u\|_{\Hilbertzero}= 
\| \Pi^{\Hilbertzero}_{0} u\big\|_{\Hilbertzero}$ (since $\Ker \Pone=(\nabla H^1_0(\Omega))\oplus K_N(\Omega)$ in $\Hilbertzero$).

We now prove \eqref{eq:CG1alt} with $\mathsf{E} =\iota^{-1}\Ptwo=\epsilon$; the proof for $\mathsf{E}=\iota^{-1}\Ptwo^* = \epsilon^*$ is analogous. 
By definition, there exists $\phi\in H^1_0(\Omega)$ such that $ \Pi^{\Hilbertzero}_{\nabla H^1_0(\Omega)} u=\nabla\phi$. Then $ \Pi^{\Hilbertzero}_{\nabla H^1_0(\Omega)} (\iota^{-1}\Ptwo)  \Pi^{\Hilbertzero}_{\nabla H^1_0(\Omega)} u = \nabla w$ where
\beqs
\big( \nabla w,\nabla v \big)_{L^2(\Omega)} = \big( \epsilon \nabla\phi, \nabla v \big)_{L^2(\Omega)} \quad\tfa v\in H^1_0(\Omega);
\eeqs
i.e., $\Delta w = \dive (\epsilon \nabla \phi)$. 
When $j=1$, since $\cZ^1=\Hilbert=H_0(\curl,\Omega)$, 
\beqs
\big\|\Pi^{\Hilbertzero}_{\nabla H^1_0(\Omega)} u \big\|_{\cZ^1} = \big\|\nabla\phi \big\|_{\Hilbert}=\big\|\nabla\phi \big\|_{L^2(\Omega)}
=\big\|\Pi^{\Hilbertzero}_{\nabla H^1_0(\Omega)} u \big\|_{\Hilbertzero} 
\eeqs
and thus \eqref{eq:CG1alt} immediately holds when $j=1$. To prove \eqref{eq:CG1alt} for $j=2,\ldots,\newell+1$, 
we apply the regularity result of Theorem \ref{thm:Weber} with $u=\nabla\phi$, $\zeta=\epsilon$, and 
$\kappa=\newell$ (so that the regularity assumptions on $\{\Omega_j\}_{j=1}^n$ and $\epsilon$ are satisfied by Assumption \ref{ass:regularity}).
By the definition of $\cZ^j$ \eqref{eq:Zj} and the regularity result \eqref{eq:Weber1} with $\ell+1=j-1$ (so that $\ell=0,\ldots,k-1$ corresponds to $j=2,\ldots,\newell+1$), 
there exists $C>0$ such that, for $j=2,\ldots,\newell+1$, 
\begin{align*}
\N{\nabla\phi}_{\cZ^{j}}
=\N{\nabla\phi}_{\newZ^{j}_\wn}
=\N{\nabla\phi}_{\Hpwo{j-1}(\Omega)} &\leq C\Big(\N{\nabla \phi}_{L^2(\Omega)} + \N{\wn ^{-1}\Delta w}_{\Hpwo{j-2}(\Omega)}\Big)\\
&\hspace{-1cm}\leq C\Big(\N{\nabla \phi}_{L^2(\Omega)} + \N{\nabla w}_{\Hpwo{j-1}(\Omega)}\Big)\\
&\hspace{-1cm}\leq C \Big( \N{\nabla\phi}_{L^2(\Omega)} + \N{\nabla w}_{\newZ^{j}_\wn}\Big)= C \Big( \N{\nabla\phi}_{\Hilbertzero} + \N{\nabla w}_{\cZ^{j}}\Big); 
\end{align*}
i.e., \eqref{eq:CG1alt} holds for $j=2,\ldots,\newell+1$ and the proof is complete.

\subsection{The assumptions of the second part of Lemma \ref{lem:oscil}}\label{sec:10.2}

\ble
Suppose that $\Omega$ is $C^{\newell+1}$ with respect to the partition $\{\Omega_j\}_{j=1}^n$ (in the sense of Definition \ref{def:Crpartition}) and 
$\epsilon,\mu \in C^{\newell}(\overline{\Omega_j})$ 
 for all $j=1,\ldots, n$.

If $\dive f=0$ then there exists $\widetilde\Pi_0$ such that (i)  $\Pi_0^*f=\widetilde\Pi_0^*f$, (ii) $\Pi_0 \widetilde\Pi_0 = \widetilde\Pi_0$, and (iii)  
the bound \eqref{eq:tired2} holds.
\ele

\bpf
Let 
\beq\label{eq:tildePi}
\widetilde\Pi_0:= \Pi^{\Hilbertzero}_{K_N(\Omega)} \Pi_0,
\eeq
where (as above) $\Pi^{\Hilbertzero}_{K_N(\Omega)}$ is the $\Hilbertzero$-orthogonal projection onto $K_N(\Omega)$ \eqref{eq:KN}. 
 Then Point (ii) holds since $\widetilde{\Pi}_0$ maps into the kernel. Furthermore, $\Pi_0^* f = \widetilde\Pi_0^* f$ if and only if, for all $v\in \Hilbert$, 
\beq\label{eq:levoit1}
\langle f, \Pi_0 v\rangle_{\Hilbert^*\times \Hilbert} 
=\langle f, \widetilde\Pi_0 v\rangle_{\Hilbert^*\times \Hilbert};  
\quad\text{ i.e., } \quad \big\langle f, (I-\Pi^{\Hilbertzero}_{K_N(\Omega)})\Pi_0 v\big\rangle_{\Hilbert^*\times \Hilbert}=0.
\eeq
On the kernel, $I-\Pi^{\Hilbertzero}_{K_N(\Omega)}$ projects to $\grad H^1_0(\Omega)$, so to prove \eqref{eq:levoit1} it is sufficient to prove that $(f,\nabla \phi)_{L^2(\Omega)}=0$ for all $\phi\in H_0^1(\Omega)$. This last statement is the condition that $\dive f=0$, so Point (i) holds. 
For (iii), we first observe that,  by \eqref{eq:tildePi}, the fact that $(\Pi^{\Hilbertzero}_{K_N(\Omega)} )^*= \iota \Pi^{\Hilbertzero}_{K_N(\Omega)} \iota^{-1}$ (by \eqref{eq:Riesz} and the self-adjointness of $\Pi^{\Hilbertzero}_{K_N(\Omega)}$), and \eqref{eq:sun1}, 
\begin{align}\label{eq:tired3}
\big\|\iota^{-1} \widetilde\Pi_0^* \iota\big\|_{\cZ^{m-1}\to \cZ^{m+1}} 
= \big\|\iota^{-1}\Pi_0^*\iota\Pi^{\Hilbertzero}_{K_N(\Omega)}\big\|_{\cZ^{m-1}\to \cZ^{m+1}} 
\leq 
C\big\|\Pi^{\Hilbertzero}_{K_N(\Omega)}\big\|_{\cZ^{m-1}\to \cZ^{m+1}}.
\end{align}
By Lemma \ref{lem:CoDaNi}, $\|\Pi^{\Hilbertzero}_{K_N(\Omega)}\|_{\cZ^{m-1}\to \cZ^{m+1}}\leq C$, and this combined with \eqref{eq:tired3} gives the desired bound \eqref{eq:tired2}.
\epf

\subsection{Proof that $\gamma_{\rm dv}(\operator), \gamma_{\rm dv}(\operator^*)\leq C \wn h (1+\wn h) $}\label{sec:gamma_dv_main}

\ble[Bound on $\gamma_{\rm dv}(\operator) , \gamma_{\rm dv}(\operator^*)$]\label{lem:gamma_dv_main}
Let $P= \Pone -\Ptwo$ with $\Pone$ and $\Ptwo$ defined by \eqref{eq:Maxwell_PDE} with $\mu$ and $\epsilon$ satisfying \eqref{eq:coefficients_sign}.
Suppose that $\Omega$ is $C^{2}$ with respect to the partition $\{\Omega_j\}_{j=1}^n$ (in the sense of Definition \ref{def:Crpartition}) and 
$\epsilon,\mu \in C^{1}(\overline{\Omega_j})$ 
 for all $j=1,\ldots, n$.
Then
\beqs
\max\big\{ \gamma_{\rm dv}(\operator) , \gamma_{\rm dv}(\operator^*)\big\} \leq C \wn h (1+ \wn h).
\eeqs
\ele

Lemma \ref{lem:gamma_dv_main} is a consequence of (i) the following abstract bound on $\gamma_{\rm dv}$ and (ii) properties of the 
interpolation operator $\CJ_h$ recapped in \S\ref{sec:interpolation}. 

\ble\label{lem:gammadv1}
Suppose there exists $J: \Hilbert_h \to \Hilbert_h$
such that (i) $Jw_h=w_h$ for all $w_h\in \Hilbert_h$, 
(ii) $J: \Pi_1 \Hilbert_h \to \Hilbert_h$, 
and (iii) $J\Pi_0 w_h \in  \Hilbert_h\cap\Ker \Pone$ for all $w_h\in \Hilbert_h$. 
Then there exists $C>0$ such that, if 
\beq\label{eq:og1}
w_h\in \Hilbert_h \text{ satisfies } \big\langle \Ptwo  w_h, v_h \big\rangle_{\Hilbertzero^*\times\Hilbertzero}=0 \,\,\tfa v_h \in \Hilbert _h \cap \Ker \Pone 
\eeq
then
\beqs
\N{\Pi_0 w_h}_{\Hilbertzero} \leq C \N{ (I-J) \Pi_1w_h}_{\Hilbertzero}.
\eeqs
\ele

The following result then holds immediately from the definition of $\gamma_{\rm dv}(\operator)$ \eqref{eq:gamma_dv}.

\begin{corollary}\label{cor:gammadv}
Under the assumptions of Lemma \ref{lem:gammadv1}, 
there exists $C>0$ such that
\begin{align*}
&\gamma_{\rm dv}(\operator)\leq  C\sup \bigg\{
\frac{
\N{ (I-J) \Pi_1w_h}_{\Hilbertzero}
}{\N{w_h}_{\Hilbert }} \, :\, \\
&\hspace{4cm}w_h\in \Hilbert_h \text{ satisfies } \big\langle\Ptwo  w_h, v_h \big\rangle_{\Hilbertzero^*\times\Hilbertzero}=0 \,\,\tfa v_h \in \Hilbert _h \cap \Ker \Pone 
\bigg\}.
\end{align*}
\end{corollary}

\bpf[Proof of Lemma \ref{lem:gamma_dv_main} using Corollary \ref{cor:gammadv}]
We apply 
Corollary \ref{cor:gammadv} 
with $J= \CJ_h$. The assumptions (i) and (iii) on $J$ in Lemma \ref{lem:gammadv1} 
are satisfied by the properties of $\CJ_h$ recapped in 
Theorem \ref{thm:interpolation} and \eqref{eq:commuting}.

To show that $J: \Pi_1 \Hilbert_h \to \Hilbert_h$ (i.e., the assumption (ii) in Lemma \ref{lem:gammadv1}), we apply the regularity result of Theorem \ref{thm:Weber} with $\kappa=1$ and $\ell=0$. 
As in the proof of Lemma \ref{lem:Maxwell} (b), for the operator $\operator$, $\dive(\epsilon \Pi_1 w_h)=0$ for all $w_h\in \Hilbert_h$. Similarly, for the operator $\operator^*$, $\dive(\epsilon^* \Pi_1 w_h)=0$.
Therefore, in both cases, by Theorem \ref{thm:Weber} (with $\zeta$ equal either $\epsilon$ or $\epsilon^*$) and \eqref{eq:coefficients_sign}, 
\begin{align}\nonumber
\N{\Pi_1 w_h}_{\Hpwo{1}(\Omega)} &\leq C \Big( \N{\Pi_1 w_h}_{L^2(\Omega)} + \big\|\wn ^{-1} \curl(\Pi_1 w_h)\big\|_{L^2(\Omega)}\Big)\\
&= C \Big( \N{\Pi_1 w_h}_{L^2(\Omega)} + \big\|\wn ^{-1} \curl  w_h\big\|_{L^2(\Omega)}\Big),
\label{eq:final1}
\end{align}
since $\curl (\Pi_0 w_h)=0$ and thus $\curl (\Pi_1 w_h) =\curl w_h$. 
Therefore, given $w_h \in \Hilbert_h$, $\Pi_1 w_h \in H_0(\curl, \Omega)\cap \Hpw{1}(\Omega)$.
If we can show that $\curl (\Pi_1 w_h) \in \Hpw{1}(\Omega)$, then $\Pi_1 w_h\in Z^2$ (with $Z^j$ defined by \eqref{eq:Zj}), and then 
$J\Pi_1 w_h\in \Hilbert_h$ by Theorem \ref{thm:interpolation}. However, $\curl (\Pi_1 w_h)= \curl w_h$ (as established above), and 
 a standard inverse inequality (see, e.g., \cite[\S12.1]{ErGu:21} for the case of simplicial meshes and, e.g., \cite[Appendix A.1]{ChSp:25} for the case of curved meshes) implies that 
\beq\label{eq:inverse_estimate}
\|\wn^{-1}\curl w_h\|_{H^1_\wn(K)}\leq C\big(1 + (\wn h_K)^{-1}\big)\|\wn^{-1}\curl w_h\|_{L^2(K)}.
\eeq
Since $\cT_h$ is conforming with the partition $\{\Omega_i\}_{i=1}^n$ of $\Omega$ from Assumption \ref{ass:regularity}, $\curl (\Pi_1 w_h)= \curl w_h \in \Hpw{1}(\Omega)$, and we have established that $J: \Pi_1 \Hilbert_h \to \Hilbert_h$.

We now bound $\N{ (I-\CJ_h) \Pi_1w_h}_{L^2(\Omega)}$ appearing in the bound of Corollary \ref{cor:gammadv}. 
By \eqref{eq:interpolation1} with $r=1$, the definition of $H^1_\wn(K)$, the fact that $\curl (\Pi_1 w_h) =\curl w_h$, 
and the inverse estimate \eqref{eq:inverse_estimate},
given $k_0>0$ there exists $C>0$ such that, for all $k\geq k_0$,
\begin{align*}
\|(I- \CJ_h)\Pi_1 w_h\|_{L^2(K)}
&\leq
C \wn h_K
\Big(\|\Pi_1 w_h\|_{H^{1}_\wn(K)}
+
\wn h_K\|\wn^{-1}\curl  w_h\|_{H^{1}_k(K)}
\Big)\\
&\leq
C \wn h_K
\Big(\|\Pi_1 w_h\|_{H^{1}_k(K)}
+
\big(\wn h_K + 1\big)\|\wn^{-1}\curl  w_h\|_{L^2(K)}
\Big).
\end{align*}
Summing over $K\in \cT_h$, 
recalling that $\cT_h$
is conforming with the partition $\{\Omega_i\}_{i=1}^n$ of $\Omega$, using that $\hK\leq h$,  and using  
the definition of 
$\|\cdot\|_{\Hpwo{1}(\Omega)}$ \eqref{eq:pw_norm} gives
\beqs
\|(I- \CJ_h)\Pi_1 w_h\|_{L^2(\Omega)}
\leq
C \wn h 
\Big(\|\Pi_1 w_h\|_{\Hpwo{1}(\Omega)}
+
\big(\wn h+1\big)\|\wn^{-1}\curl  w_h\|_{L^2(\Omega)}
\Big).
\eeqs
Therefore, by \eqref{eq:final1} and the boundedness of $\Pi_1: L^2(\Omega)\to L^2(\Omega)$,
\begin{align*}
\|(I- \CJ_h)\Pi_1 w_h\|_{L^2(\Omega)}
&\leq 
C \wn h \Big( \N{ w_h}_{L^2(\Omega)} + \big(\wn h+1\big)\N{\wn ^{-1} \curl  w_h}_{L^2(\Omega)}\Big),
\end{align*}
and the result follows.
 \epf

\bpf[Proof of Lemma \ref{lem:gammadv1}]
We prove that, for $w_h$ as in \eqref{eq:og1}, 
\beq\label{eq:gammadvSTP}
C_{P_2} \N{\Pi_0 w_h}^2_{\Hilbertzero} \leq 
 \big|\big\langle \Ptwo \Pi_0 w_h,(I-J) \Pi_1 w_h\big\rangle_{\Hilbertzero^*\times\Hilbertzero}
\big|,
\eeq
and the result  then follows from  the boundedness of $\Ptwo: \Hilbertzero\to \Hilbertzero^*$.

Since $J$ is well defined on both $\Hilbert_h$ and $\Pi_1\Hilbert_h$, it is well defined on $\Pi_0 \Hilbert_h$. 
Now, by assumption $(I-J)w_h=0$; therefore $(I-J)\Pi_0 w_h=-(I-J)\Pi_1 w_h$ and
\beq\label{eq:train1}
\Pi_0 w_h = (I-J)\Pi_0 w_h + J \Pi_0 w_h = -(I-J)\Pi_1 w_h + J \Pi_0 w_h.
\eeq
By \eqref{eq:P2} with $v=\Pi_0 w_h$ and \eqref{eq:train1}, to prove \eqref{eq:gammadvSTP} it is sufficient to prove that 
\beq\label{eq:gammadvSTP2}
 \big\langle \Ptwo \Pi_0 w_h, J\Pi_0 w_h\big\rangle_{\Hilbertzero^*\times\Hilbertzero}=0.
 \eeq
If $v_h \in \Ker \Pone$ then $v_h =\Pi_0 v_h$ and, by \eqref{eq:key2}, 
\begin{align}\nonumber
\big\langle \Ptwo  w_h, v_h \big\rangle_{\Hilbertzero^*\times\Hilbertzero}
= \big\langle \Ptwo  w_h, \Pi_0 v_h \big\rangle_{\Hilbertzero^*\times\Hilbertzero}
=\big\langle \Pi_0^*\Ptwo  w_h,  v_h \big\rangle_{\Hilbertzero^*\times\Hilbertzero}
&=\big\langle \Pi_0^*\Ptwo  \Pi_0 w_h, v_h \big\rangle_{\Hilbertzero^*\times\Hilbertzero}\\
&=\big\langle \Ptwo  \Pi_0 w_h,  v_h \big\rangle_{\Hilbertzero^*\times\Hilbertzero}.
\label{eq:og2}
\end{align}
Since $J\Pi_0 w_h \in \Ker \Pone$, \eqref{eq:gammadvSTP2} immediately follows from \eqref{eq:og1} and \eqref{eq:og2}, and the proof is complete.
\epf

\begin{appendix}

\section{The Maxwell radial PML problem}\label{sec:PML}

This section recaps the definition of the Maxwell radial PML problem from \cite{CoMo:98a}, \cite[\S13.5.3.2, Page 378]{Mo:03}, \cite{BaWu:05} (using slightly different notation)
and shows that the coeffcients $\mu$ and $\epsilon$ in this case satisfy \eqref{eq:coefficients_sign} (see Lemma \ref{lem:Maxwell2} below).

\paragraph{The scattering problem.}

Let $\Omega_-\subset\mathbb{R}^3$ be such that its open complement $\Omega_+:= \mathbb{R}^3\setminus\overline{\Omega_-}$ is connected. Let $n$ be the outward-pointing unit normal vector to $\Omega_-$. 
Let $\epsilon_{\rm scat},\mu_{\rm scat}$ be real-valued symmetric positive definite matrix functions on $\Omega_+$ 
such that $\supp(\epsilon_{\rm scat}- I), \supp(\mu_{\rm scat}-I)\subset B_{R_{\rm scat}}$ for some $R_{\rm scat}>0$. The scattering problem is then:~given $f\in L^2_{\rm comp}(\Rea^3)$, find $E_{\rm scat}\in H_{\rm loc}(\curl,\Omega)$ with $E_{\rm scat}\times n=0$ on $\partial\Omega_-$ such that  
\beq\label{eq:Maxwell_scat}
\wn ^{-2}\curl (\mu_{\rm scat}^{-1} \curl E_{\rm scat}) - \epsilon_{\rm scat} E_{\rm scat} = f \quad\tin \Omega_+,
\eeq
and $E_{\rm scat}$ satisfies the Silver-M\"uller radiation condition (see, e.g., \cite[Equation 1.29]{Mo:03}).

\paragraph{PML definition.}

Let $\Rtr >\RPMLo>R_{\rm scat}$ and let $\Omega_{\tr}\subset \mathbb{R}^d$ be a bounded Lipschitz open set with $B_{\Rtr }\subset \Omega_{\tr} \subset  B_{C\Rtr }$ for some $C>0$ (i.e., $\Omega_{\tr}$ has characteristic length scale $\Rtr $).
Let $\Omega:=\Omega_{\tr}\cap \Omega_+$. 
For $0\leq \theta<\pi/2$, let the PML scaling function $f_\theta\in C^{1}([0,\infty);\mathbb{R})$
 be defined by $f_\theta(r):=f(r)\tan\theta$ for some $f$ satisfying
\begin{equation}
\label{e:fProp}
\begin{gathered}
\big\{f(r)=0\big\}=\big\{f'(r)=0\big\}=\big\{r\leq \RPMLo\big\},\quad f'(r)\geq 0,\quad f(r)\equiv r \text{ on }r\geq \RPMLt;
\end{gathered}
\end{equation}
i.e., the scaling ``turns on'' at $r=\RPMLo$, and is linear when $r\geq \RPMLt$.  Note that $\Rtr $ can be $<\RPMLt$, i.e., truncation can occur before linear scaling is reached. 
Given $f_\theta(r)$, let 
\beqs
\alpha(r) := 1 + \ri f_\theta'(r) \quad \tand\quad \beta(r) := 1 + \ri f_\theta(r)/r,
\eeqs
and let
\beq\label{eq:Ac}
\mu := 
\begin{cases}
\mu_{\rm scat}
\hspace{-1ex}
& \tin B_{\RPMLo},\\
HDH^T 
\hspace{-1ex}
&\tin (B_{\RPMLo})^c
\end{cases}
\tand
\epsilon:= 
\begin{cases}
\epsilon_{\rm scat} 
\hspace{-1ex}
& \tin B_{\RPMLo},\\
HDH^T 
\hspace{-1ex}
&\tin (B_{\RPMLo})^c
\end{cases}
\eeq
where, in spherical polar coordinates $(r,\varphi, \phi)$,
\beq\label{eq:DH3}
D =
\left(
\begin{array}{ccc}
\beta(r)^2\alpha(r)^{-1} &0 &0\\
0 & \alpha(r) &0 \\
0 & 0 &\alpha(r)
\end{array}
\right) 
\tand
H =
\left(
\begin{array}{ccc}
\sin \varphi \cos\phi & \cos \varphi \cos\phi & - \sin \phi \\
\sin \varphi \sin\phi & \cos \varphi \sin\phi & \cos \phi \\
\cos \varphi & - \sin \varphi & 0 
\end{array}
\right) 
\eeq
(observe that $\mu_{\rm scat} =\epsilon_{\rm scat}=I$  when $r=\RPMLo$ and thus $\mu$ and $\epsilon$ are continuous at $r=\RPMLo$).

The perfectly-matched-layer approximation to $E_{\rm scat}$ is then the solution of \eqref{eq:Maxwell_scat} in $\Omega$ with coefficients \eqref{eq:Ac}.

We highlight that, in other papers on PMLs, the scaled variable, which in our case is $r+\ri f_\theta(r)$, is often written as $r(1+ \ri \widetilde{\sigma}(r))$ with $\widetilde{\sigma}(r)= \sigma_0$ for $r$ sufficiently large; see, e.g., \cite[\S4]{HoScZs:03}, \cite[\S2]{BrPa:07}. Therefore, to convert from our notation, set $\widetilde{\sigma}(r)= f_\theta(r)/r$ and $\sigma_0= \tan\theta$.

\begin{lemma}\label{lem:Maxwell2}
Given $\epsilon_{\rm scat}, \mu_{\rm scat}$ as above and  
a scaling function $f(r)$ satisfying \eqref{e:fProp}, let $\epsilon,\mu$ be defined by \eqref{eq:Ac}. Given $\varepsilon>0$, the following is true.

(i) There exists $C>0$ such that, for all $\varepsilon \leq \theta \leq \pi/2-\varepsilon$, $x \in \Omega$, and $\xi,\zeta \in \mathbb{C}^d$,
\beqs
\max\big\{\big|\big(\mu^{-1}(x)\xi,\zeta\big)_2\big|\,,\,\big|\big(\epsilon(x)\xi,\zeta\big)_2\big|\big\}
\leq C\|\xi\|_2 \|\zeta\|_2.
\eeqs

(ii) If, additionally, $f(r)/r$ is nondecreasing, then there exists $c>0$ such that, for all $\varepsilon \leq \theta \leq \pi/2-\varepsilon$, $x \in \Omega$, and $\xi,\zeta \in \mathbb{C}^d$,
\beqs
\min\big\{\big(\Re (\mu^{-1}(x))\xi,\xi\big)_2\,,\,\big(\Re(\epsilon(x))\xi,\xi\big)_2\big\}
\geq c \|\xi\|_2^2 .
\eeqs
\end{lemma}

\bpf[Sketch proof]
Part (i) follows in a straightforward way from the definitions of $\mu$ and $\epsilon$. The proof of Part (ii) is very similar to the proof of the analogous Helmholtz result in \cite[Lemma 2.2]{GLSW1}.
\epf

We highlight that the assumption in Part (ii) of Lemma \ref{lem:Maxwell2} that $f(r)/r$ is nondecreasing
 is standard in the literature; e.g., in the alternative notation described above it is that 
$\widetilde{\sigma}$ is non-decreasing -- see \cite[\S2]{BrPa:07}.

\section{Proof of Theorem \ref{thm:interpolation} (interpolation results in $\Hilbert_h$)}\label{app:interpolation}

Recall that $D \LF_K$ is the Jacobian matrix of $\LF_K$.
By the first bound in \eqref{eq_mapK} with $|\alpha|=1$ and the fact that $d=3$, there exists $C>0$ such that, for all $K \in \mathcal T_h$,
\begin{equation}
\label{eq_element_det}
\frac{1}{C} \hK^3 \leq \operatorname{det} (D \LF_K) \leq C\hK^3
\quad\text{ in }
\widehat K.
\end{equation}
For $v \in L^2(K)$, we introduce the curl- and divergence-conforming Piola transformations: 
\begin{align}\label{eq:Piola_curl}
\LF_K^{\rm c}(v)
&:=
(D\LF_K)^T (v \circ \LF_K),
\\
\LF_K^{\rm d}(v)
&:=
\operatorname{det} (D\LF_K)
(D\LF_K)^{-1} (v \circ \LF_K);
\label{eq:Piola_div}
\end{align}
see, e.g., \cite[\S3.9]{Mo:03}, \cite[\S9.2.1]{ErGu:21}. 
Recall that 
\beq\label{eq:peanut1}
\curl (\LF^{\rm c}_K (v)\big)=  \LF^{\rm d}_K(\curl v)
\eeq
for all $v\in C^1(K)$ by, e.g., \cite[Corollary 9.9]{ErGu:21}.

In analogue with the definition of the space $Z^j$ \eqref{eq:Zj}, let
\beqs
Z^j(\cT_h) := \Big\{ u \in H_0(\curl, \Omega): u|_K \in H^{j-1}(K) \tand (\curl u)|_K \in H^{j-1}(K)\Big\}.
\eeqs

We denote by $\widehat I^{\rm c}, \widehat I^{\rm d}$ the canonical N\'ed\'elec
and Raviart-Thomas interpolants of degree $p$ on the reference element $\widehat K$ 
(see, e.g., \cite[\S5.4--5.5]{Mo:03}, \cite[Chapter 16]{ErGu:21}).
We then consider the interpolation operators 
$I^{\rm c}_h : 
\newZ^2(\cT_h) 
\to V_h$ 
and $I^{\rm d}_h: H^1_{\rm pw}(\mathcal T_h) \to V_h$ 
by setting
\beq\label{eq:I_h}
I^{\rm c}_h|_K := (\LF_K^{\rm c})^{-1} \circ \widehat I^{\rm c} \circ \LF_K^{\rm c}
\eeq
and 
\beqs
I^{\rm d}_h|_K := (\LF_K^{\rm d})^{-1} \circ \widehat I^{\rm d} \circ \LF_K^{\rm d}
\eeqs
(see, e.g., \cite[Proposition 9.3]{ErGu:21}). 
By standard commuting properties, 
\beq\label{eq:commute_appendix}
(\curl \circ I^{\rm c}_h)|_K =( I^{\rm d}_h \circ \curl  )|_K; 
\eeq
see, e.g., \cite[Lemma 16.8]{ErGu:21}.

We prove below that Theorem \ref{thm:interpolation} holds with $\CJ_h = I^{\rm c}_h$. 
The following two lemmas are key ingredients in this proof. 

\begin{lemma}[Norm bounds on Piola transformations]\label{lem:norms}
With $\LF_K^{\rm c}(v)$ and $\LF_K^{\rm d}(v)$ defined by \eqref{eq:Piola_curl} and \eqref{eq:Piola_div}, respectively, there exists $C>0$ such that, for $\newellfour\in \{1,\ldots,p\}$, for all $K\in \cT_h$, and 
for all $v \in H^\newellfour(K)$, 
\begin{equation}
\label{eq_piola_curl}
\hK^{3/2} |\LF_K^{\rm c}(v)|_{H^\newellfour(\widehat K)}
\leq
C L \left (\frac{\hK}{L}\right )^{\newellfour+1}
\sum_{j=1}^{\newellfour} L^j |v|_{H^j(K)}
\end{equation}
and
\begin{equation}
\label{eq_piola_div}
\hK^{3/2}\big|\LF_K^{\rm d} (v)\big|_{H^\newellfour(\widehat K)}
\leq
CL^2
\left (
\frac{\hK}{L}
\right )^{\newellfour+2}
\sum_{j=0}^{\newellfour}
L^j |v|_{H^j(K)}.
\end{equation}
\end{lemma}

\begin{lemma}[Derivative of co-factor matrix]\label{lem:cofactor}
There exists $C>0$ such that, for all $K \in \mathcal T_h$,
\begin{equation}
\label{eq_piola_coefficient}
\big\|\partial^{\alpha}\big(\operatorname{det}(D \LF_K)(D\LF_K)^{-1}\big)\big\|_{L^\infty(\widehat K)}
\leq
C L^2 \left ( \frac{h}{L} \right )^{|\alpha|+2}
\end{equation}
for $1\leq |\alpha|\leq p$.
\end{lemma}

\begin{proof}[Proof of Theorem \ref{thm:interpolation} using Lemmas \ref{lem:norms} and \ref{lem:cofactor}]
Let  $\widehat v := \LF^{\rm c}_K(v)$, so that, by \eqref{eq:I_h},
\begin{equation*}
\LF^{\rm c}_K(v-I^{\rm c}_h v) = \widehat u-\widehat I^{\rm c} \widehat u.
\end{equation*}
We now claim that 
\begin{equation*}
\|v-I^{\rm c}_h v\|_{L^2(K)}
\leq
C
\hK^{-1}\hK^{3/2} \|\widehat v-\widehat I^{\rm c} \widehat v\|_{L^2(\widehat K)};
\end{equation*}
indeed, the $\hK^{3/2}$ comes from the Jacobian in the change of variable with~\eqref{eq_element_det},
and the $\hK^{-1}$ comes from the factor $(D\LF_K^{-1})$ 
via \eqref{eq_mapK}.
On the reference element, the proof of \cite[Theorem 3.14]{Hi:02} implies that there exists $C>0$ such that, for $\newellfour\in \{1,\ldots,p\}$, 
\begin{equation*}
\|\widehat v-\widehat I^{\rm c} \widehat v\|_{L^2(\widehat K)}
\leq
C
\left (
|\widehat v|_{H^\newellfour(\widehat K)} + |\curl \widehat v|_{H^\newellfour(\widehat K)}
\right ).
\end{equation*}
so that
\begin{equation}\label{eq:baby1}
\|v-I^{\rm c}_h v\|_{L^2(K)}
\leq
C
\hK^{-1} \hK^{3/2}
\left (
|\widehat v|_{H^\newellfour(\widehat K)} + |\curl \widehat v|_{H^\newellfour(\widehat K)}
\right ).
\end{equation}
By \eqref{eq_piola_curl},
\begin{equation}\label{eq:baby2}
\hK^{-1} \hK^{3/2}|\widehat v|_{H^\newellfour(\widehat K)}
\leq
C \left (\frac{\hK}{L}\right )^{\newellfour} \sum_{j=1}^\newellfour L^j|v|_{H^j(K)}.
\end{equation}
We now let $w := \curl v$ and $\widehat w := \curl \widehat v$, so that $\widehat w= \LF^{\rm d}_K(w)$ by \eqref{eq:peanut1}.
By \eqref{eq_piola_div}, \eqref{eq_element_det}, and \eqref{eq_piola_coefficient} with $|\alpha|=1$, 
\begin{align}\nonumber
&\hK^{-1} \hK^{3/2}|\curl \widehat v|_{H^\newellfour(\widehat K)}
=
\hK^{-1} \hK^{3/2}|\widehat w|_{H^\newellfour(\widehat K)}
\\
&\qquad\leq
C L^2 \hK^{-1} \left (\frac{\hK}{L}\right )^{\newellfour+2} \sum_{j=1}^\newellfour L^j|w|_{H^j(K)}
=
C \hK \left (\frac{\hK}{L}\right )^\newellfour \sum_{j=1}^\newellfour L^j|\curl v|_{H^j(K)}.
\label{eq:baby3}
\end{align}
The bound \eqref{eq:interpolation1} then follows from the combination of \eqref{eq:baby1}, \eqref{eq:baby2}, and \eqref{eq:baby3}.

To prove \eqref{eq:interpolation2}, first observe that, by \eqref{eq:commute_appendix},
\begin{equation}\label{eq:butter1}
\|\curl(v-I_h^{\rm c} v)\|_{L^2(K)}
=
\|w-I_h^{\rm d} w\|_{L^2(K)}.
\end{equation}
By the definition of $\LF^{\rm d}$ \eqref{eq:Piola_div}, the lower bound in \eqref{eq_element_det}, and the first bound in \eqref{eq_mapK} with $|\alpha|=1$, 
\beq\label{eq:butter2}
\|w-I_h^{\rm d} w\|_{L^2(K)}
\leq
\hK^{-2} \hK^{3/2} \|\widehat w -\widehat I^{\rm d} \widehat w\|_{L^2(\widehat K)}.
\eeq
Since $I^{\rm d}$ is continuous over $H^1(\widehat K)$ and preserves polynomials
of degree $p-1$, 
the Bramble-Hilbert lemma (see, e.g., \cite[Theorem 28.1]{Ci:91}, \cite[\S11.3]{ErGu:21}) implies that
there exists $C>0$ such that, for $\newellfour\in \{1,\ldots,p\}$, 
\begin{equation}\label{eq:butter3}
\|\widehat w -\widehat I^{\rm d} \widehat w\|_{L^2(\widehat K)}
\leq
C |\widehat w|_{H^\newellfour(\widehat K)};
\end{equation}
the result \eqref{eq:interpolation2} then follows from the combination of \eqref{eq:butter1}, \eqref{eq:butter2}, \eqref{eq:butter3}, and \eqref{eq_piola_div}.
\end{proof}

It therefore remains to prove Lemmas \ref{lem:norms} and \ref{lem:cofactor}.

\begin{proof}[Proof of Lemma \ref{lem:cofactor}]
We first observe that
\begin{equation*}
\operatorname{det}(D \LF_K)(D\LF_K)^{-1}
\end{equation*}
is just the cofactor matrix of $D\LF_K$. Since this is a $3 \times 3$ matrix,
its entries are sum of products of pairs of elements of $D\LF_K$. As a result, we
just need to estimate terms of the form $\partial_m\LF_K^r \partial_n\LF_K^q$, which
easily follows by the product rule:
\begin{equation*}
\partial^\alpha(\partial_m\LF_K^r \partial_n\LF_K^q)
=
\sum_{\beta \leq \alpha}
\bigg(
\begin{array}{c}
\alpha \\ \beta
\end{array}
\bigg)
\partial^\beta\partial_m\LF_K^r \partial^{\alpha-\beta} \partial_n\LF_K^q,
\end{equation*}
leading to
\begin{equation*}
\|\partial^{\alpha}(\operatorname{det} D \LF_K (D \LF_K)^{-1})\|_{L^\infty(\widehat K)}
\leq
C\sum_{\beta\leq\alpha}
|\partial^{\beta} (D \LF_K)||\partial^{\alpha-\beta} (D \LF_K)|;
\end{equation*}
the result then follows from the first bound in \eqref{eq_mapK}. 
\end{proof}

We now need to describe how partial derivatives of functions are modified under
the element mappings.

\begin{lemma}[Sobolev norms of composed functions]
Given $m\geq 1$ there exists $C>0$ such that if 
$K \in \mathcal T_h$ and $u \in H^\newellfour(K)$ then
\begin{equation}
\label{eq_partial_composed}
h^{3/2} |u \circ \LF_K|_{H^\newellfour(\widehat K)}
\leq
C \left (\frac{h}{L}\right )^\newellfour
\sum_{j=1}^\newellfour L^j |u|_{H^j(K)}.
\end{equation}
\end{lemma}

\begin{proof}
In this proof we denote the $j$th component of $\LF_K$ ($j=1,2,3$) by $\LF_K^j$. 
We claim that, for any multi-index $\alpha \geq 0$, 
\begin{equation}
\label{tmp_expression_partial}
\partial^\alpha(u\circ\LF_K)
=
\sum_{\beta \leq \alpha}
\Psi_\beta (\partial^\beta u) \circ \LF_K,
\end{equation}
where each $\Psi_\beta$ is of the form
\begin{equation}\label{eq:Psi_beta}
\Psi_\beta = \sum_{\ell=1}^{N_\beta}
\prod_{j=1}^{|\beta|} \partial^{\gamma_j^\ell} \LF_K^{\mu_j^{\ell}},
\end{equation}
for some integer $N_\beta$, multi-indices $\gamma_j^\ell$ with $\sum_{j=1}^{|\beta|}|\gamma^\ell_j| = |\alpha|$,
and $\mu_j^{\ell} \in \{1,2,3\}$.

Once~\eqref{tmp_expression_partial} is established, the result follows, since 
\begin{equation*}
|\Psi_{\beta}|
\leq
C L^{|\beta|} \left (\frac{h}{L}\right )^{|\alpha|}
\end{equation*}
by (i) the first bound in~\eqref{eq_mapK} and (ii) using \eqref{eq_element_det} to 
take into account the change of variable in the $L^2(K)$ integrals. 

We prove \eqref{tmp_expression_partial} by induction. When $|\alpha|=1$, 
\begin{equation*}
\partial_m (u \circ \LF_K) = \sum_{r=1}^3(\partial_m \LF_K^r) \big((\partial_r u) \circ \LF_K\big)
\end{equation*}
for all $m \in \{1,2,3\}$, and so \eqref{tmp_expression_partial} holds. 
Suppose that \eqref{tmp_expression_partial} holds for all $\alpha$ with $|\alpha|= M\geq 1$. 
By  \eqref{tmp_expression_partial}, \eqref{eq:Psi_beta}, and the chain and product rules,
\begin{align}\nonumber
&\partial_m \Big(\partial^\alpha(u\circ\LF_K)\Big)\\
&=
\sum_{\beta \leq \alpha}
\bigg[
\sum_{\ell=1}^{N_\beta}
\partial_m
\bigg(
\prod_{j=1}^{|\beta|} \partial^{\gamma_j^\ell} \LF_K^{\mu_j^{\ell}}
\bigg)
(\partial^{\beta} u) \circ \LF_K+
\prod_{j=1}^{|\beta|} \partial^{\gamma_j^\ell} \LF_K^{\mu_j^{\ell}}
\sum_{r=1}^3(\partial_m\LF_K^r)
\big(\partial_r (\partial^{\beta} u) \circ \LF_K\big)
\bigg]\nonumber
\\
&=\sum_{\beta' \leq \alpha+ e_m}
\Psi_{\beta'} (\partial^{\beta'} u) \circ \LF_K,
\nonumber
\end{align}
with $\Psi_{\beta'}$ of the form \eqref{eq:Psi_beta} except now $\sum_{j=1}^{|\beta|}|\gamma^\ell_j| = |\alpha|+1$.
That is, \eqref{tmp_expression_partial} holds for all  $\alpha$ with $|\alpha|\leq M+1$ and the proof is complete.
\end{proof}

\begin{proof}[Proof of Lemma \ref{lem:norms}]
The bound \eqref{eq_piola_curl} follows from 
the definition of $\LF^{\rm c}_K$ \eqref{eq:Piola_curl}, 
the product rule, the first bound in \eqref{eq_mapK}, and \eqref{eq_partial_composed}.
The bound \eqref{eq_piola_div} follows in a similar way from 
the definition of $\LF^{\rm d}_K$ \eqref{eq:Piola_div}, the product rule, \eqref{eq_piola_coefficient}, and  \eqref{eq_partial_composed}.
\end{proof}

\end{appendix}

\section*{Acknowledgements}

JG was supported by EPSRC grants EP/V001760/1 and EP/V051636/1.

\footnotesize{
\bibliographystyle{plain}
\bibliography{biblio_combined_sncwadditions}
}

\end{document}